\documentclass[11pt,reqno]{amsart}
\usepackage[left=1in,right=1in,top=1in,bottom=0.75in]{geometry}
\usepackage[T1]{fontenc}
\usepackage[utf8]{inputenc}
\usepackage[tt=false]{libertine}
\usepackage{microtype}
\usepackage{amsmath,amssymb,mathtools}
\usepackage{graphicx}
\usepackage{aliascnt}
\usepackage{enumitem}
\usepackage{xcolor}
\usepackage[colorlinks=true]{hyperref}
\usepackage{cleveref}

\definecolor{arxivblue}{RGB}{0,55,110}
\hypersetup{
  citecolor=violet,
  linkcolor=violet,
  urlcolor=arxivblue,
  pdftitle={A Priori Estimates for Singular Fractional Stochastic Burgers Equations},
  pdfauthor={Xiaohao Ji}
}
\numberwithin{equation}{section}
\allowdisplaybreaks[2]

\newtheorem{theorem}{Theorem}[section]
\newaliascnt{proposition}{theorem}
\newtheorem{proposition}[proposition]{Proposition}
\aliascntresetthe{proposition}
\newaliascnt{lemma}{theorem}
\newtheorem{lemma}[lemma]{Lemma}
\aliascntresetthe{lemma}
\newaliascnt{corollary}{theorem}
\newtheorem{corollary}[corollary]{Corollary}
\aliascntresetthe{corollary}
\theoremstyle{remark}
\newaliascnt{remark}{theorem}
\newtheorem{remark}[remark]{Remark}
\aliascntresetthe{remark}
\newenvironment{nouppercase}{\renewcommand{\uppercasenonmath}[1]{}}{}

\crefname{theorem}{Theorem}{Theorems}
\Crefname{theorem}{Theorem}{Theorems}
\crefname{proposition}{Proposition}{Propositions}
\Crefname{proposition}{Proposition}{Propositions}
\crefname{lemma}{Lemma}{Lemmas}
\Crefname{lemma}{Lemma}{Lemmas}
\crefname{corollary}{Corollary}{Corollaries}
\Crefname{corollary}{Corollary}{Corollaries}
\crefname{remark}{Remark}{Remarks}
\Crefname{remark}{Remark}{Remarks}

\newcommand{\T}{\mathbb T}
\newcommand{\R}{\mathbb R}
\newcommand{\dd}{\mathop{}\!\mathrm d}
\newcommand{\eps}{\varepsilon}
\DeclarePairedDelimiter{\norm}{\lVert}{\rVert}
\DeclarePairedDelimiter{\abs}{\lvert}{\rvert}
\DeclarePairedDelimiterX{\ip}[2]{\langle}{\rangle}{#1,#2}
\newcommand{\Ccal}{\mathcal C}
\newcommand{\Dcal}{\mathcal D}
\newcommand{\Acal}{\mathcal A}
\newcommand{\Rcal}{\mathcal R}
\newcommand{\Kcal}{\mathcal K}
\newcommand{\Ical}{\mathcal I}
\newcommand{\zzone}{\text{\resizebox{.7em}{!}{\includegraphics{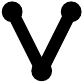}}}}
\newcommand{\zztwo}{\text{\resizebox{.7em}{!}{\includegraphics{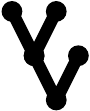}}}}
\newcommand{\zzthree}{\text{\resizebox{.7em}{!}{\includegraphics{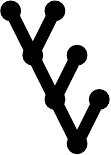}}}}
\newcommand{\zzfour}{\text{\resizebox{1em}{!}{\includegraphics{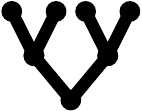}}}}
\newcommand{\zzthreereso}{\text{\resizebox{.7em}{!}{\includegraphics{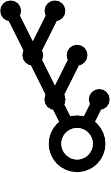}}}}
\newcommand{\zzfivebal}{\text{\resizebox{1em}{!}{\includegraphics{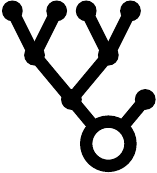}}}}
\newcommand{\zzfivecomb}{\text{\resizebox{.7em}{!}{\includegraphics{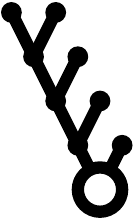}}}}
\newcommand{\zzfivetwothree}{\text{\resizebox{.9em}{!}{\includegraphics{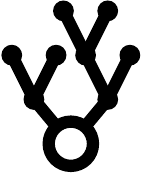}}}}
\newcommand{\zzfivebalprod}{\text{\resizebox{1em}{!}{\includegraphics{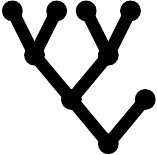}}}}
\newcommand{\zzfivecombprod}{\text{\resizebox{.7em}{!}{\includegraphics{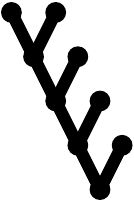}}}}
\newcommand{\zzfivetwothreeprod}{\text{\resizebox{.9em}{!}{\includegraphics{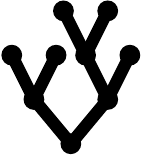}}}}
\newcommand{\Xone}{X^{\zzone}}
\newcommand{\Xtwo}{X^{\zztwo}}
\newcommand{\Xthree}{X^{\zzthree}}
\newcommand{\Xfour}{X^{\zzfour}}
\newcommand{\XthreeRes}{X^{\zzthreereso}}
\newcommand{\XoneFourB}{X^{\zzfivebal}}
\newcommand{\XoneFourC}{X^{\zzfivecomb}}
\newcommand{\XtwoThree}{X^{\zzfivetwothree}}
\newcommand{\XoneFourBFull}{X^{\zzfivebalprod}}
\newcommand{\XoneFourCFull}{X^{\zzfivecombprod}}
\newcommand{\XtwoThreeFull}{X^{\zzfivetwothreeprod}}
\newcommand{\mpara}{\,\mathord{\prec\!\!\!\prec}\,}
\DeclareMathOperator{\PV}{PV}
\DeclareMathOperator{\supp}{supp}

\title[A Priori Estimates for Singular Fractional SBE]{\LARGE A Priori Estimates for Singular Fractional Stochastic Burgers Equations}
\author{Xiaohao Ji}
\date{\textnormal{August 30, 2026\qquad
Email: \href{mailto:jixh1020@gmail.com}{\texttt{jixh1020@gmail.com}}}}

\begin{document}
\begin{nouppercase}
\maketitle
\end{nouppercase}
\vspace{-2em}
\begin{center}
\small Institut f\"ur Mathematik, Freie Universit\"at Berlin,
14195 Berlin, Germany
\end{center}
\vspace{-0.5em}

\begin{abstract}
We consider the periodic fractional stochastic Burgers equation
\[
 (\partial_t+\Lambda^\gamma)u
 =\partial_x(u^2)+\abs{\partial_x}^{1-\alpha}\xi,
 \qquad 1<\gamma\le2,
\]
where
\(\widehat{\Lambda^\gamma f}(k)=\abs{2\pi k}^{\gamma}\widehat f(k)\)
and \(\xi\) is space--time white noise.  We prove pathwise \(L^1\), energy,
and Besov estimates for the transformed remainder under
\[
  \alpha>
  \max\left\{
    \frac{7-4\gamma}{2},
    \frac{15-8\gamma}{6}
  \right\}.
\]
At \(\gamma=2\), the admissibility condition becomes \(\alpha>-1/6\),
matching the first-order paracontrolled regularity range.  At the canonical
value \(\alpha=0\), the estimates yield the periodic KPZ/SBE global
consequence stated below.  The enhancement and Zvonkin constructions build
on \cite{GubinelliPerkowski2017KPZReloaded,ZhangZhuZhu2022SingularHJB}.
For \(1<\gamma<2\), a conservative Zvonkin transform produces a compensated
nonlocal diffusion, and a cubic-increment estimate adapted from the modified
K\'arm\'an--Howarth--Monin argument of
\cite{GoldmanJosienOtto2015BurgersKS} closes the energy bound.
\end{abstract}

\tableofcontents
\enlargethispage{2pt}
\setcounter{tocdepth}{2}
\addtocontents{toc}{\protect\setcounter{tocdepth}{1}}

\section{Introduction}
\label{sec:introduction}

The stochastic Burgers equation is a basic example of a singular stochastic
conservation law. We consider in this paper the periodic equation
\begin{equation}\label{eq:formal-sbe}
  (\partial_t+\Lambda^\gamma)u
  =\partial_x(u^2)+\abs{\partial_x}^{1-\alpha}\xi,
  \qquad 1<\gamma\le2,
\end{equation}
where \(\xi\) is space--time white noise.  Here \(\gamma\) is the dissipation
order and \(\alpha\) measures the smoothing of the forcing.  The point
\((\gamma,\alpha)=(2,0)\) is the canonical periodic stochastic Burgers
equation, or equivalently the spatial derivative of the KPZ equation
introduced in \cite{KardarParisiZhang1986}.  Its finite tree expansion,
resonant products, and paracontrolled calculus are developed in
\cite{GubinelliPerkowski2017KPZReloaded}.

More precisely, we work on the normalized torus \(\T=\R/\mathbb Z\).  For
\(s\in\R\) and \(k\in\mathbb Z\setminus\{0\}\), set
\[
 \widehat{\Lambda^s f}(k)=\abs{2\pi k}^s\widehat f(k),
 \qquad \abs{\partial_x}^s:=\Lambda^s,
\]
and set every homogeneous multiplier to zero on the spatial zero mode.
We regard \(\Lambda^s\) as annihilating affine functions, and hence as
acting on smooth affine-periodic functions; the precise convention is
recorded in \hyperref[app:notation]{Appendix~A}.  We also define
\begin{equation}\label{eq:Rgamma}
 \Rcal_\gamma:=\Lambda^{\gamma-2}\partial_x,
 \qquad \partial_x\Rcal_\gamma=-\Lambda^\gamma.
\end{equation}
For \(1<\gamma<2\), let
\begin{equation}\label{eq:periodic-fractional-kernel}
 K_\gamma^{\mathrm{per}}(r)
 :=\sum_{k\in\mathbb Z}\frac1{\abs{r+k}^{1+\gamma}},
 \qquad r\in\T\setminus\{0\}.
\end{equation}
Choose \(c_\gamma>0\) so that
\begin{equation}\label{eq:periodic-fractional-laplacian}
 \Lambda^\gamma f(x)
 =c_\gamma\PV\int_\T
 [f(x)-f(z)]K_\gamma^{\mathrm{per}}(x-z)\dd z.
\end{equation}
In \cite{GubinelliPerkowski2017KPZReloaded}, local paracontrolled
well-posedness at \((\gamma,\alpha)=(2,0)\) is formulated with a blow--up
criterion, while an optimal-control representation based on the
Bou\'e--Dupuis formula yields pathwise global existence.  Here we instead seek
a priori bounds for a suitable transformation of \eqref{eq:formal-sbe} at the energy level throughout the
range specified in \eqref{eq:admissible-alpha-gamma}, and the focus is to exploit
coercivity of the Burgers nonlinearity to derive a priori estimates under singular stochastic forcing even
in the relatively weak-dissipation regime of smaller \(\gamma\).

\subsection{Related Literature}

For the one-dimensional Burgers equation driven directly by additive
space--time white noise, one spatial derivative smoother than the conservative
forcing, classical mild well-posedness and a Cole--Hopf construction were
obtained in \cite{DaPratoDebusscheTemam1994,
BertiniCancriniJonaLasinio1994}, respectively.  At the conservative endpoint,
\cite{BertiniGiacomin1997} derives Burgers/KPZ limits from particle and
interface models.  Energy solutions arising from weakly asymmetric particle
systems and their uniqueness are developed in
\cite{GoncalvesJara2014,GubinelliPerkowski2015Energy}.  Pathwise solution
theories are developed in
\cite{Hairer2013SolvingKPZ,GubinelliPerkowski2017KPZReloaded}, and an
infinitesimal-generator and martingale-problem approach in
\cite{GubinelliPerkowski2018Generator}.  Earlier local constructions of
generalized KPZ were obtained by renormalization-group methods in
\cite{KupiainenMarcozzi2017GeneralizedKPZ} and through the
regularity-structure/BPHZ framework in
\cite{BrunedChandraChevyrevHairer2021}; the flow approach gives an alternative
local theory in \cite{ChandraFerdinand2024FlowGKPZ}.  

Stationary solutions on
the line,
jointly stationary periodic Burgers flows, large-torus equilibrium
fluctuations, and stationary smooth-noise Burgers evolutions in dimensions two
and three are studied, respectively, in
\cite{DunlapGrahamRyzhik2019,DunlapGu2024,DunlapGuKomorowski2023,Dunlap2019}.
Related stationary, weak-coupling, and Gaussian-fluctuation regimes for
two-dimensional anisotropic KPZ and stochastic Burgers in \(d\ge2\) are
treated in
\cite{CannizzaroErhardToninelli2020,CannizzaroErhardToninelli2021,
CannizzaroGubinelliToninelli2023}.  These works provide the classical
\(\gamma=2\) background and its neighboring stationary and scaling regimes.

The deterministic theory of Burgers equations with fractional dissipation,
including the criticality transition, is developed in
\cite{BilerFunakiWoyczynski1998,
KiselevNazarovShterenberg2008FractalBurgers}.  For stochastic variants,
global mild well-posedness is proved in \cite{BrzezniakDebbi2007}, H\"older-space
well-posedness and approximation in \cite{ArabDebbi2021}, and ergodic
properties in \cite{BrzezniakDebbiGoldys2011}.  Stationary energy solutions
for fractional Burgers equations and Hilbert-space extensions are developed
in \cite{GubinelliJara2012,GraefnerPerkowskiPopat2026}.  Fluctuation limits of
long-range interacting particle systems are studied in
\cite{Sethuraman2016,GoncalvesJara2018Density,CardosoGoncalves2024}.  Recent
well-posedness results for rough Burgers-type equations appear in
\cite{ZhangLuoYin2026FractionalRoughBurgers,Yamazaki2025GlobalBurgersNoise}.
Yamazaki adapts the dynamic high--low energy method of
\cite{HairerRosati2024SNS} for singular two-dimensional stochastic
Navier--Stokes; related stochastic energy estimates yield uniform-in-time
bounds and ergodicity under weak high-frequency forcing in
\cite{HairerZhao2025ErgodicitySNS}.

More broadly, global well-posedness for dynamic \(\Phi^4_2\), non-explosion
and coming down from infinity for \(\Phi^4_3\), and global a priori bounds for
the generalised parabolic Anderson model are established in
\cite{MourratWeber2017Global,MourratWeber2017Infinity,
ChandraDeLimaFeltesWeber2026}, respectively.  Space--time localisation and
pathwise a priori estimates, including full-subcritical, quasilinear, and
coercive settings, are developed in
\cite{MoinatWeber2018,ChandraMoinatWeber2019,OttoSauerSmithWeber2021,
ChevyrevGubinelli2025}.

Our cubic-increment estimate adapts the one-dimensional modified
K\'arm\'an--Howarth--Monin balance of
\cite[Lemma~2.4]{GoldmanJosienOtto2015BurgersKS} for the forced inviscid
Burgers equation.  This is a scalar finite-difference modification of the
classical correlation identities of
\cite{deKarmanHowarth1938IsotropicTurbulence,Monin1959LocallyIsotropicTurbulence}:
it replaces the signed third-order increment by a nonnegative cubic increment
and refines the coercive use of Burgers transport in the
Kuramoto--Sivashinsky bounds of \cite{Otto2009OptimalKS}.  Here this coercive
increment must be retained after the Zvonkin change of variables, which
transforms \(\Lambda^\gamma\) into a variable nonlocal operator.

The coordinate-change method originates in \cite{Zvonkin1974}; stable- and
L\'evy-driven variants are developed in
\cite{Priola2012StableSDE,ChenSongZhang2018LevyFlows}.  Modern Zvonkin
transformations for SDEs with singular and Lipschitz drifts are developed in
\cite{ZhangYuan2021Zvonkin}.  In the Laplacian setting,
\cite{ZhangZhuZhu2022SingularHJB} constructs a \(C^1\)-diffeomorphism from a
paracontrolled linear equation and uses it to remove the distribution-valued
part of the drift in singular HJB equations, with applications to KPZ on the
real line.  Initial-condition-independent lower and oscillation bounds for
torus KPZ are derived from a conditional stochastic-control representation in
\cite{PerkowskiVillanuevaMariz2025KPZControl}.  In the present fractional
setting the coordinate change produces a variable nonlocal operator;
controlling its compensated form links the Zvonkin reduction to the
cubic-increment estimate.

In the present framework, the range \eqref{eq:admissible-alpha-gamma} is set by
requiring the response--Zvonkin ansatz to produce coefficients above the
energy-duality threshold \(1-\gamma/2\); once this is achieved, the coercive
closure imposes no further restriction on \((\alpha,\gamma)\).  It would be
interesting to investigate whether these a priori bounds remain robust in more
singular regimes under a different local construction, for instance a
fractional adaptation of the flow approach in
\cite{ChandraFerdinand2024FlowGKPZ}.

\subsection{Outline and Main Result}
Readers familiar with
\cite{GubinelliPerkowski2017KPZReloaded} and the singular-HJB analysis of
\cite{ZhangZhuZhu2022SingularHJB} may proceed directly to the transformed
equation \eqref{eq:transformed-equation}.  The intermediate construction
follows those works, with fractional scaling and additional enhancement
sectors.  Notation is collected in \Cref{app:notation}. Set
\[
L_\gamma:=\partial_t+\Lambda^\gamma.
\]  The stochastic objects
used below are defined in \eqref{eq:base-stochastic-coordinates},
\eqref{eq:sbe-tree-coordinates}, and
\eqref{eq:sbe-tree-response-expansion}, and constructed in
\Cref{prop:fractional-enhancement}.

\noindent\emph{Step 1: Linear Ansatz.}
Following
\cite{DaPratoDebussche2003StochasticQuantization}, the linear
decomposition 
\[
u=X+\Xone+r
\] gives \eqref{eq:r-equation}, containing the
additive
stochastic forcing \(L_\gamma(2\Xtwo+\Xfour)\) and transport by the rough
coefficient \(X+\Xone\). The response equation \eqref{eq:ell-equation} then removes this forcing.  With
\(r=\ell+v\), it gives \eqref{eq:v-equation}; only the transport
\((X+\Xone)v\) remains singular.  The fluxes \(v^2\), \(2\ell v\), and
\(\ell^2\) are defined at the energy regularity.

\noindent\emph{Step 2: Zvonkin map and conservative change of variables.}
To remove \((X+\Xone)v\), solve \eqref{eq:zvonkin-equation} for
\[
 \Psi=x+\psi,
 \qquad J=\partial_x\Psi.
\]
The resolvent family \eqref{eq:resolvent-model-uniform} defines
\((X+\Xone)J\).  For sufficiently large \(\lambda\),
\Cref{prop:zvonkin} makes \(\Psi\) an orientation-preserving bi-Lipschitz map.
Set
\[
 v(t,x)=J(t,x)m(t,\Psi(t,x)),
\]
as in \eqref{eq:m-def}.  The rough transports in the equations for \(v\) and
\(J\) cancel, giving
\[
 m_t-\partial_y\Dcal_\Psi^\gamma m
 =\partial_y(am^2+2Q_\lambda m+R),
\]
namely \eqref{eq:transformed-equation}, with coefficients specified in
\eqref{eq:coefficients}.
At \(\gamma=2\), the ordinary chain rule gives the local divergence-form
flux
\[
 \Dcal_\Psi^{2}f=(J^2\circ\Psi^{-1})f_y,
\]
while for \(1<\gamma<2\), it is
the compensated nonlocal object
\[
 \Dcal_\Psi^\gamma f
 =\left[
   \Rcal_\gamma\bigl(J(f\circ\Psi)\bigr)
   +(f\circ\Psi)\Lambda^\gamma\psi
 \right]\circ\Psi^{-1},
\]
as in \eqref{eq:Dflux}.  Its paralinearization and form bounds are proved
in \Cref{prop:transformed-operator-paralinearization}, and its coercivity
in \Cref{prop:transformed-operator-coercivity}.

\noindent\emph{Step 3: deterministic a priori closure.}
The entropy estimate \Cref{prop:global-L1} gives an \(L^1\) bound independent
of the energy.  Direct interpolation closes the energy estimate in the
subrange \(\alpha>(9-5\gamma)/2\); the modified
K\'arm\'an--Howarth--Monin cubic-increment argument removes this additional
restriction and yields the Besov range of \Cref{thm:khm}.

Combining these steps, the variable estimated in the main result is reached
through the sequence
\[
 u=X+\Xone+r
   =X+\Xone+\ell+v
   =X+\Xone+\ell+J(m\circ\Psi).
\]

The admissible range is
\begin{equation}\label{eq:admissible-alpha-gamma}
 1<\gamma\le2,
 \qquad
 \alpha>\max\left\{
   \frac{7-4\gamma}{2},\frac{15-8\gamma}{6}
 \right\}.
\end{equation}
At \(\gamma=2\), this becomes \(\alpha>-1/6\), the first-order
paracontrolled SBE regime.  We can now state the main estimate.
\begin{samepage}
\begin{theorem}[Main a priori estimate]
\label{thm:main-result}
Let \((\alpha,\gamma)\) satisfy
\eqref{eq:admissible-alpha-gamma}, fix \(T>0\), and let \(m\) be a smooth
transformed solution on \([0,T]\), associated with a smooth enhancement in
the class of \Cref{prop:fractional-enhancement} and obtained through the
Da Prato--Debussche, linear response, and Zvonkin transformations of
\Cref{sec:enhanced-zvonkin}.  Initialize all stochastic coordinates at zero
and assume that \(m(0)=u_0\in C^\infty(\T)\).  Then
\[
 \norm m_{L^\infty(0,T;L^1)}
 {}+\norm m_{L^\infty(0,T;L^2)}
 {}+\norm m_{L^2(0,T;H^{\gamma/2})}
 {}+\norm m_{L^3(0,T;B^s_{3,3})}
\le C_{T,s}
\]
for every \(0<s<1/3\).  At \(\gamma=2\), the Besov term may be taken with
any \(0<s<1/2\).  The constant depends locally uniformly on \(T\), \(s\),
\(\norm{u_0}_{L^2}\), the enhanced-model norm, the uniform norm of the resolvent family,
and the resulting coefficient and Zvonkin norms.  If the resolvent parameter is
selected locally uniformly as in \Cref{prop:zvonkin}, the bound is therefore
uniform over bounded families of enhanced inputs.
\end{theorem}

\begin{proof}
The coefficient bounds \eqref{eq:verified-coefficient-bounds}, together
with \Cref{prop:transformed-operator-paralinearization,%
prop:transformed-operator-coercivity} when \(1<\gamma<2\) and
\eqref{eq:local-check} when \(\gamma=2\), verify the hypotheses of
\Cref{thm:khm}.  That theorem gives the stated \(L^1\) and energy estimates
together with the bound on \(K_T^s\).  The Besov estimate follows from
\eqref{eq:Besov-from-KHM}; at \(\gamma=2\), \eqref{eq:energy-besov} gives
the improved range.  The asserted dependence of the constants follows from
the local uniformity in
\Cref{prop:response,prop:zvonkin,prop:coefficient-output}.
\end{proof}
\end{samepage}

The routine smooth-approximation step for the finite models is understood as
in \cite[Section~3.1 and Theorem~3.7]{GubinelliPerkowski2017KPZReloaded}.
Canonical compactness is stated and proved in
\Cref{prop:canonical-compactness}.  We use throughout standard paracontrolled
conventions and collect the notation in \Cref{app:notation}.

\addtocontents{toc}{\protect\setcounter{tocdepth}{2}}
\section{Enhanced equation and Zvonkin transform}
\label{sec:enhanced-zvonkin}
\subsection{Burgers enhancement and linear response}
\label{sec:response}

We record only the enhanced coordinates and equations used below; the
fixed-point argument follows
\cite{GubinelliPerkowski2017KPZReloaded,ZhangZhuZhu2022SingularHJB}.
Recall \(L_\gamma=\partial_t+\Lambda^\gamma\).  On an interval starting at
\(t_0\), write the Duhamel operator as
\begin{equation}\label{eq:fractional-duhamel}
 \Ical_{\gamma,t_0}f(t)=\int_{t_0}^t
 e^{-(t-s)\Lambda^\gamma}f(s)\dd s,
 \qquad \Kcal_{\gamma,t_0}=\partial_x\Ical_{\gamma,t_0},
\end{equation}
and define \(\Ical_{\gamma,\lambda,t_0}\),
\(\Kcal_{\gamma,\lambda,t_0}\) by replacing \(\Lambda^\gamma\) with
\(\Lambda^\gamma+\lambda\).  When \(t_0\) is fixed, we suppress it from this
notation.  Throughout this subsection, every stochastic object introduced
below and each of its mollified approximations has zero initial condition at
\(t_0\); every Duhamel integration in its recursive definition starts at
\(t_0\).  For homogeneity bookkeeping, set
\[
 \delta_\gamma:=\gamma-1,
 \qquad
 q_{\gamma,\alpha}:=\alpha+\frac{3\gamma-5}{2}.
\]
Under fractional parabolic scaling, space--time white noise has homogeneity
\(-(\gamma+1)/2-\).  Applying
\(\abs{\partial_x}^{1-\alpha}\) and then the stochastic convolution gives
\[
 -\frac{\gamma+1}{2}-(1-\alpha)+\gamma
 =\alpha+\frac{\gamma-3}{2}
 =q_{\gamma,\alpha}-\delta_\gamma,
\]
up to an arbitrarily small loss.  Thus
\[
X\in\Ccal^{q_{\gamma,\alpha}-\delta_\gamma-}.
\]
Each further Burgers interaction
\(Y\mapsto\Kcal_\gamma(XY)\) raises the formal homogeneity by
\(q_{\gamma,\alpha}\).  The two base coordinates are
\begin{equation}\label{eq:base-stochastic-coordinates}
 \begin{gathered}
 L_\gamma X=\abs{\partial_x}^{1-\alpha}\xi,
 \qquad
 \Xone=\Kcal_\gamma\mathopen{:}X^2\mathclose{:}.
 \end{gathered}
\end{equation}
The colon denotes Wick centering.  Define the rough drift by
\begin{equation}\label{eq:rough-drift}
 B:=X+\Xone.
\end{equation}
Here an \emph{enhancement} means the finite list of stochastic fields that
must be supplied in order to give every singular product below a canonical
meaning.  Whenever the enhancement supplies the resonant product
\(F\odot_{\rm ren}G\), write
\[
 F\diamond G:=F\prec G+F\succ G+F\odot_{\rm ren}G,
 \qquad F^{\diamond2}:=F\diamond F.
\]
Here every occurrence of \(F\odot_{\rm ren}G\) in the stochastic coordinate
denotes the stored canonical limit of the corresponding mollified
Bony resonances. The remaining coordinates generated by the equation with at most four noise leaves are
\begin{equation}\label{eq:sbe-tree-coordinates}
\begin{gathered}
 \Xtwo=\Kcal_\gamma(X\diamond\Xone),
 \qquad \Xfour=\Kcal_\gamma((\Xone)^{\diamond2}),
 \\
 Q_X=\Kcal_\gamma X,
 \qquad \mathbb X_Q=Q_X\odot_{\rm ren}X,
 \qquad \XthreeRes=\Kcal_\gamma(\Xtwo\odot_{\rm ren}X).
\end{gathered}
\end{equation}
All products in \eqref{eq:base-stochastic-coordinates} and
\eqref{eq:sbe-tree-coordinates} are constructed recursively from mollified
noise. We
use the \(X^{\mathsf t}\)-notation for these coordinates in prose and
displayed formulas.  These six stochastic coordinates are the fractional
counterparts of the Burgers enhancement in
\cite[Definition~3.2 and Theorem~9.1]{GubinelliPerkowski2017KPZReloaded}
and of the stochastic inputs used in
\cite{ZhangZhuZhu2022SingularHJB}; only the semigroup multiplier and the
regularity exponents change.

We next motivate the additional coordinates at the mollified level and
suppress the mollification parameter in this calculation.  Applying the Da Prato--Debussche
decomposition to \eqref{eq:formal-sbe}, write
\[
 u=X+\Xone+r,
 \qquad r(t_0)=u(t_0).
\]
Substitution, with \(B\) from \eqref{eq:rough-drift}, gives
\begin{equation}\label{eq:r-equation}
 L_\gamma r=\partial_x(r^2+2Br)+L_\gamma(2\Xtwo+\Xfour).
\end{equation}
Following the linear-response strategy of
\cite{ZhangZhuZhu2022SingularHJB}, we remove the singular inhomogeneous term
by solving
\begin{equation}\label{eq:ell-equation}
 L_\gamma\ell=2\partial_x(B\ell)
 +L_\gamma(2\Xtwo+\Xfour),
 \qquad \ell(t_0)=0,
\end{equation}
or, equivalently,
\begin{equation}\label{eq:ell-duhamel}
 \ell=2\Xtwo+\Xfour+2\Kcal_\gamma(B\ell).
\end{equation}
By symmetry, the derived full field at the next level is
\[
 \Xthree:=\Kcal_\gamma(X\diamond\Xtwo)
 =\Kcal_\gamma(\Xtwo\prec X+\Xtwo\succ X)+\XthreeRes.
\]
The standard modified-paraproduct ansatz closes from the preceding
enhancement when the next root resonances have positive regularity.  The
remaining response regime is
\begin{equation}\label{eq:sbe-tree-response-regime}
 \frac{12}{7}<\gamma\le2,
 \qquad
 \frac{15-8\gamma}{6}<\alpha
 \le\frac{21-11\gamma}{10}.
\end{equation}
Equivalently, \(5q_{\gamma,\alpha}\le2\delta_\gamma\).  To close the ansatz
throughout the full admissible range and use one model space, we introduce
the three further stored resonances
\begin{equation}\label{eq:sbe-tree-response-expansion}
 \XoneFourB=\Kcal_\gamma(X\odot_{\rm ren}\Xfour),\qquad
 \XoneFourC=\Kcal_\gamma(X\odot_{\rm ren}\Xthree),\qquad
 \XtwoThree=\Kcal_\gamma(\Xone\odot_{\rm ren}\Xtwo).
\end{equation}
Outside the regime \eqref{eq:sbe-tree-response-regime}, they
are canonically reconstructed from the lower enhancement; we retain them in
every regime so that the model space is fixed.  The enhancement and its
assigned spatial regularities are
\begin{equation}\label{eq:fractional-model-summary}
\begin{gathered}
 \mathcal T_{\mathrm{SBE}}:=
 \{X,\Xone,\Xtwo,\Xfour,\mathbb X_Q,\XthreeRes\}
 \cup\{\XoneFourB,\XoneFourC,\XtwoThree\},\\[1ex]
 s_{\alpha,\gamma}:\mathcal T_{\mathrm{SBE}}\longrightarrow\mathbb R,
 \qquad
 s_{\alpha,\gamma}(\tau):=
 \begin{cases}
  q_{\gamma,\alpha}-\delta_\gamma,
    &\tau=X,\\
  \min\{2q_{\gamma,\alpha}-\delta_\gamma,q_{\gamma,\alpha}\},
    &\tau=\Xone,\\
  \min\{3q_{\gamma,\alpha}-\delta_\gamma,q_{\gamma,\alpha}\},
    &\tau=\Xtwo,\\
  \min\{4q_{\gamma,\alpha}-\delta_\gamma,
          2q_{\gamma,\alpha},q_{\gamma,\alpha}+\delta_\gamma\},
    &\tau=\Xfour,\\
  2q_{\gamma,\alpha}-\delta_\gamma,
    &\tau=\mathbb X_Q,\\
  \min\{4q_{\gamma,\alpha}-\delta_\gamma,2q_{\gamma,\alpha}\},
    &\tau=\XthreeRes,\\
  \min\{5q_{\gamma,\alpha}-\delta_\gamma,
          3q_{\gamma,\alpha},2q_{\gamma,\alpha}+\delta_\gamma\},
    &\tau\in\{\XoneFourB,\XtwoThree\},\\
  2q_{\gamma,\alpha},
    &\tau=\XoneFourC.
 \end{cases}\\[1ex]
 \mathbb X_{\alpha,\gamma}:=
 \bigl(
 X,\Xone,\Xtwo,\Xfour,\mathbb X_Q,\XthreeRes,
 \XoneFourB,\XoneFourC,\XtwoThree
 \bigr).
\end{gathered}
\end{equation}
The value \(s_{\alpha,\gamma}(\tau)\) is the assigned spatial regularity of
the stored coordinate \(\tau\); the minima account for contributions of
different regularities.  For \(I=[t_0,T]\) and every sufficiently small
\(\kappa>0\), set
\[
 \mathbb X_{\alpha,\gamma}\in
 \mathcal M_{\alpha,\gamma}^\kappa(I)
 :=\prod_{\tau\in\mathcal T_{\mathrm{SBE}}}
 C\bigl(I;\Ccal^{s_{\alpha,\gamma}(\tau)-\kappa}\bigr),
 \qquad
 Q_X\in C(I;\Ccal^{q_{\gamma,\alpha}-\kappa}).
\]

\begin{proposition}[Construction and convergence of the fractional Burgers enhancement]
\label{prop:fractional-enhancement}
Let \((\alpha,\gamma)\) satisfy
\eqref{eq:admissible-alpha-gamma}.  Fix a smooth even spatial mollifier
\(\phi\) of unit mass, set
\(\xi^\eps=\phi_\eps*_x\xi\), and construct on \([t_0,T]\) the spatially
mollified approximation of the tuple
\(\mathbb X^\eps_{\alpha,\gamma}\) in
\eqref{eq:fractional-model-summary} according to the standing initial-time
convention.

For every sufficiently small \(\kappa>0\) and every \(p<\infty\), there is
a limit \(\mathbb X_{\alpha,\gamma}\) such that, as \(\eps\downarrow0\),
\[
 \mathbb X_{\alpha,\gamma}^\eps
 \longrightarrow\mathbb X_{\alpha,\gamma}
 \quad\text{in}\quad
 L^p\bigl(\Omega;\mathcal M_{\alpha,\gamma}^\kappa(I)\bigr).
\]
The limit is independent of the mollifier, and the convergence is uniform
over restart times in compact sets.  The same convergence holds with the
coordinatewise time moduli specified in
\Cref{app:fractional-enhancement-proof}.
\end{proposition}

The stored resonant coordinates determine the corresponding complete
fifth-order Duhamel fields by
\[
 \begin{aligned}
 \XoneFourBFull
 &:=\Kcal_\gamma(X\prec\Xfour+X\succ\Xfour)+\XoneFourB,\\
 \XoneFourCFull
 &:=\Kcal_\gamma(X\prec\Xthree+X\succ\Xthree)+\XoneFourC,\\
 \XtwoThreeFull
 &:=\Kcal_\gamma(\Xone\prec\Xtwo+\Xone\succ\Xtwo)+\XtwoThree.
 \end{aligned}
\]

The proof is postponed to \Cref{app:fractional-enhancement-proof}; see
\eqref{eq:sbe-basic-multipliers}, \eqref{eq:worked-balanced-kernel}, and
\eqref{eq:sbe-tree-covariance-kernel} for the multiplier, recursive-kernel,
and covariance estimates.  Write \(\norm{\mathbb X_{\alpha,\gamma}}_T\) for
the corresponding model norm.  Then, for every \(p<\infty\),
\[
 \mathbb E\,\norm{\mathbb X_{\alpha,\gamma}}_T^p<\infty.
\]
We now fix this canonical enhancement.  Below, \(\mathfrak C_T\) denotes a
locally bounded increasing function of
\(\norm{\mathbb X_{\alpha,\gamma}}_T\) and may change from line to line; after
the resolvent lift its argument is \(\mathfrak N_{T,\kappa}^{\rm Zv}\).

We now proceed to construct the linear response $\ell$, and we shall use the modified semigroup commutator
\[
 \mathrm{Com}_{\Kcal_\gamma}(f,g)
 :=\Kcal_{\gamma,t_0}(f\prec g)
   -f\mpara\Kcal_{\gamma,t_0}g,
\]
where \(\mpara\) is the time-modified paraproduct defined in
\Cref{app:notation}.  The commutator estimate is recorded in part~(i) of
\Cref{lem:fractional-commutators}.

\begin{proposition}[Well-posedness and regularity of the linear response equation]\label{prop:response}
Let \((\alpha,\gamma)\) satisfy
\eqref{eq:admissible-alpha-gamma}, and let
\(\mathbb X_{\alpha,\gamma}\) be the canonical enhancement constructed in
\Cref{prop:fractional-enhancement}.  Then
\eqref{eq:ell-equation} has a unique modified-paraproduct solution. 
\begin{samepage}
Set
\[
 \theta_{\alpha,\gamma}:=
 \min\bigl\{1,q_{\gamma,\alpha},
          3q_{\gamma,\alpha}-\delta_\gamma\bigr\}.
\]
Then
\begin{equation}\label{eq:ell-response-estimate}
 \norm\ell_{C([t_0,T];C^{\theta_{\alpha,\gamma}-})}
 +[\ell]_{C^{\frac{\theta_{\alpha,\gamma}}{\gamma}-}([t_0,T];L^\infty)}
 \lesssim \mathfrak C_T(\norm{\mathbb X_{\alpha,\gamma}}_T),
\end{equation}
and the solution map is locally Lipschitz, uniformly over restart times on
compact intervals.
\end{samepage}
\end{proposition}

\begin{proof}
Without loss of generality, assume
\[
 \frac{\delta_\gamma}{3}<q_{\gamma,\alpha}
 \le\frac{2\delta_\gamma}{5},
 \qquad
 \theta_{\alpha,\gamma}
 =3q_{\gamma,\alpha}-\delta_\gamma>0.
\]
Indeed, outside this regime the resonant products in the response equation
are reconstructed from the lower enhancement, and the proof reduces to the
standard modified-paraproduct construction for the linear Zvonkin equation in
\cite[Section~5.1 and the proof of
Lemma~5.5]{ZhangZhuZhu2022SingularHJB}, with the heat Schauder estimate
replaced by its fractional counterpart.

Since \(Q_X=\Kcal_\gamma X\), the commutator identity gives
\[
 \begin{aligned}
 \XoneFourBFull-\Xfour\mpara Q_X
 &=\Kcal_\gamma(X\prec\Xfour)
   +\mathrm{Com}_{\Kcal_\gamma}(\Xfour,X)+\XoneFourB,\\
 \XoneFourCFull-\Xthree\mpara Q_X
 &=\Kcal_\gamma(X\prec\Xthree)
   +\mathrm{Com}_{\Kcal_\gamma}(\Xthree,X)+\XoneFourC.
\end{aligned}
\]
The coordinatewise time moduli in \Cref{prop:fractional-enhancement} supply
the corresponding hypotheses of the commutator estimate and will be
suppressed below.  The two identities, the analogous paraproduct estimate
for \(\XtwoThreeFull\), and the model regularities place the first
difference in \(C_t\Ccal^{5q_{\gamma,\alpha}-\delta_\gamma-}\) and the
other two terms in \(C_t\Ccal^{2q_{\gamma,\alpha}-}\).  Consequently,
\begin{equation}\label{eq:sbe-tree-model-space}
 \begin{aligned}
 &2\XoneFourBFull+8\XoneFourCFull+4\XtwoThreeFull\\
 &\qquad
 -2(\Xfour+4\Xthree)\mpara Q_X
 \in C([t_0,T];\Ccal^{2q_{\gamma,\alpha}-}).
 \end{aligned}
\end{equation}

We use the single controlled high-response term \(S\):
\begin{equation}\label{eq:response-sbe-tree-space}
 \begin{aligned}
 \ell&=2\Xtwo+\Xfour+4\Xthree+S,\\
 S-2(\Xfour+4\Xthree+S)\mpara Q_X
 &\in C([t_0,T];\Ccal^{2q_{\gamma,\alpha}-}).
 \end{aligned}
\end{equation}
Indeed, substituting this ansatz into \eqref{eq:ell-duhamel} and collecting
the extracted third-, fourth-, and fifth-order terms gives
\begin{equation}\label{eq:response-sbe-tree-remainder-equation}
 \begin{aligned}
 S={}&2\XoneFourBFull+8\XoneFourCFull+4\XtwoThreeFull\\
 &+2\Kcal_\gamma\bigl(
 X\diamond S
 +\Xone\diamond(\Xfour+4\Xthree+S)
 \bigr).
 \end{aligned}
\end{equation}

For the contraction, use controlled pairs
\[
 S=S'\mpara Q_X+S^\sharp,\qquad
 S'\in C_t\Ccal^{q_{\gamma,\alpha}-},\qquad
 S^\sharp\in C_t\Ccal^{2q_{\gamma,\alpha}-},
\]
with the corresponding \(C_t^{q_{\gamma,\alpha}/\gamma-}L^\infty\)
moduli for \(S\) and \(S'\) included in the norm.  The only singular
resonance is reconstructed by
\[
 S\odot_{\rm ren}X
 =S'\mathbb X_Q
 +\mathrm{Com}_{\mpara}(S',Q_X,X)
 +S^\sharp\odot X.
\]
All three terms have regularity
\(3q_{\gamma,\alpha}-\delta_\gamma->0\) by part~(ii) of
\Cref{lem:fractional-commutators}.  Given a controlled pair, let
\(\widetilde S\) be the right-hand side of
\eqref{eq:response-sbe-tree-remainder-equation} and set
\[
 \widetilde S'=2(\Xfour+4\Xthree+S).
\]
Then
\begin{align*}
 \widetilde S-\widetilde S'\mpara Q_X
 ={}&\bigl[
 2\XoneFourBFull+8\XoneFourCFull+4\XtwoThreeFull
 -2(\Xfour+4\Xthree)\mpara Q_X
 \bigr]\\
 &+2\mathrm{Com}_{\Kcal_\gamma}(S,X)
 +2\Kcal_\gamma(X\prec S+S\odot_{\rm ren}X)\\
 &+2\Kcal_\gamma\bigl(
 \Xone\diamond(\Xfour+4\Xthree+S)
 \bigr).
\end{align*}
Each line belongs to \(C_t\Ccal^{2q_{\gamma,\alpha}-}\) by
\eqref{eq:sbe-tree-model-space}, the two commutator estimates, and the
fractional Schauder estimate.  For the last line one uses
\[
 (2q_{\gamma,\alpha}-\delta_\gamma)+q_{\gamma,\alpha}
 =3q_{\gamma,\alpha}-\delta_\gamma>0.
\]
The usual short-time factors therefore give a contraction on controlled
pairs in the standard weighted norm.  At its fixed point,
\(S'=2(\Xfour+4\Xthree+S)\), so the controlled relation is exactly
\eqref{eq:response-sbe-tree-space}.  Restart iteration and the corresponding
time-increment estimate yield
\eqref{eq:ell-response-estimate}, locally Lipschitz in the enhanced model
and uniformly over restart times on compact intervals.
\end{proof}

If \(q_{\gamma,\alpha}>\delta_\gamma/2\), all relevant resonances are classical,
whereas for \(2\delta_\gamma/5<q_{\gamma,\alpha}\le\delta_\gamma/2\),
parts~(i) and~(ii) of \Cref{lem:fractional-commutators} give
\begin{equation}\label{eq:Xthree-controlled}
 \begin{aligned}
 \Xthree&=\Xtwo\mpara Q_X+(\Xthree)^\sharp,
 &(\Xthree)^\sharp
 &\in C([t_0,T];\Ccal^{4q_{\gamma,\alpha}-\delta_\gamma-}),\\
 \Xthree\odot X
 &=\Xtwo\,\mathbb X_Q
 +\mathrm{Com}_{\mpara}(\Xtwo,Q_X,X)
 +(\Xthree)^\sharp\odot X.
\end{aligned}
\end{equation}
At the endpoint \(q_{\gamma,\alpha}=\delta_\gamma/2\), the implicit losses are
chosen sufficiently small so that the strict hypotheses in part~(ii) remain
valid.
The commutator inputs have strictly positive total regularity
\((3q_{\gamma,\alpha}-\delta_\gamma)+q_{\gamma,\alpha}
+(q_{\gamma,\alpha}-\delta_\gamma)
=5q_{\gamma,\alpha}-2\delta_\gamma>0\); the product
\(\Xtwo\mathbb X_Q\) and the resonance
\((\Xthree)^\sharp\odot X\) have the same total.  Likewise,
\[
 (q_{\gamma,\alpha}-\delta_\gamma)
 +(4q_{\gamma,\alpha}-\delta_\gamma)
 =(2q_{\gamma,\alpha}-\delta_\gamma)
 +(3q_{\gamma,\alpha}-\delta_\gamma)
 =5q_{\gamma,\alpha}-2\delta_\gamma>0.
\]
Hence the other two fifth-order resonances are ordinary Bony
reconstructions, and the standard one-level modified-paraproduct ansatz closes
using only the lower enhancement.

With \(\ell\) now constructed, write \(r=\ell+v\).  Subtracting
\eqref{eq:ell-equation} from
\eqref{eq:r-equation} gives
\begin{equation}\label{eq:v-equation}
 L_\gamma v-2\partial_x(Bv)
 =\partial_x(v^2+2\ell v+\ell^2).
\end{equation}
Here \(v(t_0)=u(t_0)\), since \(\ell(t_0)=0\).
By \eqref{eq:ell-response-estimate},
\(v^2\), \(\ell v\), and \(\ell^2\) are canonically defined at the energy
regularity.

\subsection{Zvonkin transform and the transformed fractional operators}
\label{sec:zvonkin}

After the response transform, the singularity is concentrated in the
transport term.  The Zvonkin equation constructs a controlled Jacobian
\(J\) for which \(B\diamond J\) is defined in a paracontrolled sense; the
resulting change of variables cancels the transport exactly.
Let
\[
  \Psi(t,x)=x+\psi(t,x),
  \qquad J=\partial_x\Psi=1+\partial_x\psi.
\]
For a parameter \(\lambda\ge0\), solve
\begin{equation}\label{eq:zvonkin-equation}
  (\partial_t+\Lambda^\gamma+\lambda)\psi=2BJ,
  \qquad \psi(t_0)=0.
\end{equation}
After differentiation,
\begin{equation}\label{eq:J-equation}
  J_t+\Lambda^\gamma J-2\partial_x(BJ)=-\lambda(J-1).
\end{equation}
With the resolvent version of the Duhamel operators in
\eqref{eq:fractional-duhamel},
the Zvonkin equation can equivalently be written as
\begin{equation}\label{eq:zvonkin-duhamel}
  \psi=2\Ical_{\gamma,\lambda}(BJ),
  \qquad
  J-1=2\Kcal_{\gamma,\lambda}(BJ).
\end{equation}

\paragraph{Resolvent extension of \(\mathbb X_Q\).}
For \(\lambda\ge0\), set
\[
 Q_{X,\lambda}^\eps:=\Kcal_{\gamma,\lambda}X^\eps,
 \qquad
 Q_{X,\lambda}:=\Kcal_{\gamma,\lambda}X.
\]
The existing resonance \(\mathbb X_Q\) extends canonically to the family
\begin{equation}\label{eq:resolvent-model-uniform}
 \mathbb X_{Q,\lambda}^\eps
 :=Q_{X,\lambda}^\eps\odot_{\rm ren}X^\eps
 \longrightarrow \mathbb X_{Q,\lambda},
 \qquad \lambda\ge0,
 \qquad
 \mathbb X_{Q,0}=\mathbb X_Q.
\end{equation}
For every \(p<\infty\) and every sufficiently small \(\kappa>0\), the
convergence is uniform in \(\lambda\ge1\) and over restart times on compact
intervals.  Moreover, the Zvonkin input norm
\[
 \mathfrak N_{T,\kappa}^{\rm Zv}
 :=\norm{\mathbb X_{\alpha,\gamma}}_T
   +\sup_{\lambda\ge1}
 \norm{\mathbb X_{Q,\lambda}}_
 {C(I;\Ccal^{2q_{\gamma,\alpha}-\delta_\gamma-\kappa})}
 <\infty
\]
has moments of every finite order.
The uniform second-chaos estimate and mollifier convergence are recorded in
\eqref{eq:uniform-resolvent-second-chaos}. Together with \(\mathbb X_{\alpha,\gamma}\), this family supplies the controlled
interpretation of \(BJ\) used below; the required regularity margin is verified
in the proof of \Cref{prop:zvonkin}.
For a smooth approximation, positivity of \(J\) gives an inverse map;
\Cref{prop:zvonkin} guarantees the same property for the limiting map on
the intervals considered below.  Henceforth write
\[
 \Phi:=\Psi^{-1},
 \qquad
 f^\Theta:=f\circ\Theta
\]
for every scalar function \(f\) and coordinate map \(\Theta\), with
composition in the spatial variable.  The
conservative transformed unknown is defined by
\begin{equation}\label{eq:m-def}
  v=Jm^\Psi.
\end{equation}
Since \(\dd y=J\dd x\), this is the natural push-forward of a density.
For the Cauchy problem and the global argument we take \(t_0=0\).  The
standing stochastic initial-time convention, together with \(\ell(0)=0\)
and \(\psi(0)=0\), gives
\begin{equation}\label{eq:transformed-initial-data}
 X(0)=\Xone(0)=\ell(0)=0,
 \qquad \Psi(0)=\mathrm{id},
 \qquad J(0)=1,
 \qquad m(0)=v(0)=r(0)=u_0.
\end{equation}

With the odd Fourier multiplier \(\Rcal_\gamma\) from \eqref{eq:Rgamma},
equation \eqref{eq:v-equation} is the conservation law
\begin{equation}\label{eq:v-conservation}
  v_t=\partial_x\bigl[
    \Rcal_\gamma v+2Bv+v^2+2\ell v+\ell^2
  \bigr].
\end{equation}

\begin{proposition}[Zvonkin change-of-variables identity]\label{prop:transformed}
For smooth data, the variable \(m\) from \eqref{eq:m-def} solves
\begin{equation}\label{eq:transformed-equation}
  m_t-\partial_y\Dcal_{\Psi}^{\gamma}m
  =\partial_y\bigl(a m^2+2Q_\lambda m+R\bigr),
\end{equation}
where
\begin{equation}\label{eq:coefficients}
\begin{aligned}
  a&=(J^2)^\Phi,
  &\qquad Q_0&=(\ell J)^\Phi,\\
  Q_\lambda&=Q_0+\frac\lambda2\psi^\Phi,
  &\qquad R&=(\ell^2)^\Phi.
\end{aligned}
\end{equation}
For a scalar field \(f\), define the transformed flux operator by
\begin{equation}\label{eq:Dflux}
  \Dcal_{\Psi}^{\gamma}f
  =\bigl[
    \Rcal_\gamma(Jf^\Psi)
    +f^\Psi\Lambda^\gamma\psi
  \bigr]^\Phi.
\end{equation}
Throughout, \(\Dcal_\Psi^\gamma\) and
\(\Acal_{\Psi}^{\gamma}:=-\partial_y\Dcal_{\Psi}^{\gamma}\) are,
respectively, the transformed flux and diffusion operators.
Here \(a=(J^2)^\Phi\) is the coefficient of the quadratic transport
flux; the principal diffusion coefficient is
\((J^\gamma)^\Phi\).
\end{proposition}

\begin{proof}
Since \(v=Jm^\Psi\), while
\(J_t=\partial_x\Psi_t\) and
\(\partial_x(m^\Psi)=J(m_y)^\Psi\),
\[
 \begin{aligned}
 v_t
 &=J_tm^\Psi+J\,(m_t)^\Psi+J\,(m_y)^\Psi\Psi_t,\\
 \partial_x(m^\Psi\Psi_t)
 &=J_tm^\Psi+J\,(m_y)^\Psi\Psi_t.
 \end{aligned}
\]
Hence
\(J\,(m_t)^\Psi=v_t-\partial_x(m^\Psi\Psi_t)\).  Substituting
\eqref{eq:v-conservation} and
\(\Psi_t=-\Lambda^\gamma\psi-\lambda\psi+2BJ\) gives
\[
 \begin{aligned}
 J\,(m_t)^\Psi
 &=\partial_x\bigl[
   \Rcal_\gamma v+2Bv+v^2+2\ell v+\ell^2-m^\Psi\Psi_t
   \bigr]\\
 &=\partial_x\bigl[
   \Rcal_\gamma v+m^\Psi\Lambda^\gamma\psi+\lambda m^\Psi\psi
   +v^2+2\ell v+\ell^2
   \bigr].
 \end{aligned}
\]
Here the rough transport cancels exactly because
\(2Bv-2m^\Psi BJ=0\).  Division by \(J\) followed by composition with
\(\Phi\) uses the coordinate identity
\[
 [J^{-1}\partial_x F]^\Phi
 =\partial_y(F^\Phi).
\]
The first two terms in the last flux, kept together, give
\(\Dcal_\Psi^\gamma m\) by \eqref{eq:Dflux}.  The remaining terms satisfy
\[
 \begin{aligned}
 (v^2)^\Phi&=am^2,
 &\qquad (2\ell v)^\Phi&=2Q_0m,\\
 (\ell^2)^\Phi&=R,
 &\qquad \lambda(m^\Psi\psi)^\Phi
 &=\lambda\psi^\Phi m.
 \end{aligned}
\]
Since \(2Q_\lambda=2Q_0+\lambda\psi^\Phi\), these identities
give \eqref{eq:transformed-equation}.
\end{proof}

\begin{remark}[Regularity notation]
We use \(\rho\) for the regularity of
\(\ell,Q_0,Q_\lambda,R\), and \(\beta\) for that of
\(J-1\) and \(a\).  Choose
\begin{equation}\label{eq:coefficient-exponent-choice}
\begin{aligned}
 1-\frac\gamma2<\rho
 &<\theta_{\alpha,\gamma},\\
 \max\{\rho,\delta_\gamma-2q_{\gamma,\alpha}\}<\beta
 &<\min\{1,q_{\gamma,\alpha}\}.
\end{aligned}
\end{equation}
The bound \(\rho>1-\gamma/2\) is the dual-flux threshold
\(C^\rho\hookrightarrow H^{1-\gamma/2}\), while
\(\rho<\theta_{\alpha,\gamma}\) comes from the response estimate.
The leading controlled term in \(J-1\) is
\(J\mpara Q_{X,\lambda}\), with
\(Q_{X,\lambda}\in C^{q_{\gamma,\alpha}-}\); hence
\(\beta<q_{\gamma,\alpha}\).  The bound \(\beta<1\) keeps the construction
in the first-order H\"older regime,
\(\beta>\delta_\gamma-2q_{\gamma,\alpha}\) makes the final Jacobian
resonance regular, and \(\beta>\rho\) preserves the \(C^\rho\) regularity
of \(\ell J\). The interval for \(\rho\) is nonempty precisely when
\[
 \alpha>\frac{7-4\gamma}{2}
 \quad\text{and}\quad
 \alpha>\frac{15-8\gamma}{6}.
\]
They also make the interval for \(\beta\) nonempty, so
\eqref{eq:coefficient-exponent-choice} adds no restriction to
\eqref{eq:admissible-alpha-gamma}.  The deterministic argument uses only
\[
 a\in C^\beta,
 \qquad Q_0,Q_\lambda,R\in C^\rho,
 \qquad \beta,\rho>1-\frac\gamma2.
\]
The corresponding stochastic homogeneities are verified in
\Cref{app:fractional-enhancement-proof}.
\end{remark}

At each fixed time, lift \(m\) periodically to \(\R\) and set
\[
 \bar m:=\int_0^1m(y)\dd y,
 \qquad
 M(y):=\int_0^y m(z)\dd z.
\]
Then \(M_y=m\) and \(M(y+1)=M(y)+\bar m\).  Changing the lower limit
changes \(M\) only by a constant and is therefore immaterial.  Since the
degree-one lift \(\Psi=x+\psi\) satisfies
\(\Psi(x+1)=\Psi(x)+1\), the pullback \(M^\Psi\) is
affine-periodic with increment \(\bar m\).  Recalling the affine-periodic
convention in \hyperref[app:notation]{Appendix~A}, for smooth \(m\) and \(\Psi\) we have
\[
 \Rcal_\gamma(Jm^\Psi)
 =\Rcal_\gamma\partial_x(M^\Psi)
 =-\Lambda^\gamma(M^\Psi),
 \qquad
 \Lambda^\gamma\Psi=\Lambda^\gamma\psi.
\]
The first identity holds on every nonzero Fourier mode.  On the zero mode,
\(\Rcal_\gamma\) annihilates the mean of
\(\partial_x(M^\Psi)\), while \(\Lambda^\gamma\) annihilates the
affine part of \(M^\Psi\).  Consequently,
\begin{equation}\label{eq:chain-commutator}
  \Dcal_{\Psi}^{\gamma}m
  =-\bigl[
    \Lambda^\gamma(M^\Psi)
    -m^\Psi\Lambda^\gamma\Psi
  \bigr]^\Phi.
\end{equation}
The bracket in \eqref{eq:chain-commutator} is periodic and extends
continuously as a single compensated object to rough regularity; its summands need not then
be meaningful separately.  In particular,
\(\Dcal_{\Psi}^{\gamma}m\) is generally neither a scalar coefficient times
\(\Lambda^{\gamma-1}m\) nor
\(a\Lambda^{\gamma-2}\partial_y m\).

When \(\gamma=2\), the ordinary chain rule reduces the compensated flux to
\begin{equation}\label{eq:local-check}
  \Dcal_{\Psi}^{2}f=af_y,
  \qquad
  \Acal_{\Psi}^{2}f=-\partial_y(af_y),
  \qquad
  a=(J^2)^\Phi.
\end{equation}
Consequently, the transformed equation becomes
\begin{equation}\label{eq:local-transformed-equation}
 m_t-\partial_y(am_y)
 =\partial_y(am^2+2Q_\lambda m+R).
\end{equation}
This is the Laplacian conservative Zvonkin equation.
\begin{samepage}
For \(f,g\in H^1(\T)\), integration by parts gives the symmetric form and
coercivity of \(\Acal_\Psi^2\):
\begin{align}
 \ip g{\Acal_{\Psi}^{2}f}
 &=\int_\T a\,g_yf_y\dd y,
 \label{eq:local-symmetric-form}\\
 \ip f{\Acal_{\Psi}^{2}f}
 &=\int_\T a\abs{f_y}^2\dd y
 \asymp_\Psi\norm f_{\dot H^1}^2.
 \notag
\end{align}
The same bounds on \(a\) also yield
\begin{equation}\label{eq:local-form-bound}
 \norm{\Dcal_{\Psi}^{2}f}_{L^2}
 +\norm{\Acal_{\Psi}^{2}f}_{H^{-1}}
 \lesssim_\Psi\norm f_{H^1}.
\end{equation}
Both operators depend continuously, with a local H\"older modulus, on
\(\Psi\) in these operator norms on bounded \(C^{1+\beta}\)
bi-Lipschitz sets.  The dependence is locally Lipschitz under bounded
\(C^{2+\beta}\) control.
\end{samepage}

For \(\gamma=2\),
\eqref{eq:zvonkin-equation} is the standard singular linear heat equation
for the Jacobian.  Its paracontrolled Schauder estimate and large-resolvent
diffeomorphism construction follow, after periodic specialization, from
\cite[Sections~3 and~5.1, in particular Lemmas~5.5--5.6]
{ZhangZhuZhu2022SingularHJB}; we do not repeat that fixed-point proof.
For \(-1/6<\alpha<0\), the drift index in the notation of that work
is \(1/2-\alpha\in(1/2,2/3)\).  The canonical value \(\alpha=0\) is covered after
an arbitrarily small H\"older loss, while \(\alpha>0\) is smoother.

\begin{proposition}[Solvability and Jacobian bounds for the Zvonkin equation]\label{prop:zvonkin}
Assume \eqref{eq:admissible-alpha-gamma}, choose \((\rho,\beta)\) as in
\eqref{eq:coefficient-exponent-choice}, and take the canonical enhancement
of \Cref{prop:fractional-enhancement} together with its resolvent family
\eqref{eq:resolvent-model-uniform}.  For all sufficiently large \(\lambda\),
equation \eqref{eq:zvonkin-equation} has a unique
paracontrolled solution.  The solutions built from smooth enhanced
approximations
converge in \(C([t_0,T];C^{1+\beta})\) to a map
\(\Psi=x+\psi\), and
\begin{equation}\label{eq:J-bounds}
 0<J_*\le J(t,x)\le J^*<\infty
\end{equation}
uniformly on finite time intervals.  In particular, \(\Psi\) and
\(\Phi\) are uniformly bi-Lipschitz.  The convergence
\begin{equation}\label{eq:zvonkin-schauder-estimate}
 \norm{\psi}_{C([t_0,T];C^{1+\beta})}
 +\norm{J-1}_{C([t_0,T];C^\beta)}
\longrightarrow0
\qquad\text{as }\lambda\to\infty
\end{equation}
is locally uniform on sets where the enhanced-model norm and the uniform
resolvent-family norm are bounded.
Consequently \(\lambda\) may be chosen so that
\(\norm{J-1}_{L^\infty}\le1/2\), which implies \eqref{eq:J-bounds}; alternatively,
with \(\lambda=0\), the same smallness is obtained on a sufficiently short
restarted interval.
\end{proposition}

\begin{proof}[Proof Outline]
For \(\gamma=2\), this is the cited Laplacian construction; so we assume
\(1<\gamma<2\) below.  The bookkeeping indices introduced above satisfy
\[
 \theta_{\alpha,\gamma}
 =\min\{1,q_{\gamma,\alpha},3q_{\gamma,\alpha}-\delta_\gamma\}.
\]
Choose
\[
 \beta<\beta_+<\min\{1,q_{\gamma,\alpha}\}.
\]
For every sufficiently small \(\kappa>0\),
\Cref{prop:fractional-enhancement} and the uniform resolvent estimate
\eqref{eq:uniform-resolvent-second-chaos} give
\[
 \begin{gathered}
 B\in C([t_0,T];\Ccal^{q_{\gamma,\alpha}-\delta_\gamma-\kappa}),
 \qquad
 Q_{X,\lambda}:=\Kcal_{\gamma,\lambda}X
 \in C([t_0,T];\Ccal^{q_{\gamma,\alpha}-\kappa}),\\
 \Xone\in
 C([t_0,T];\Ccal^{2q_{\gamma,\alpha}-\delta_\gamma-\kappa}),
 \qquad
 \sup_{\lambda\ge1}
 \norm{\mathbb X_{Q,\lambda}}_
 {C([t_0,T];\Ccal^{2q_{\gamma,\alpha}-\delta_\gamma-\kappa})}
 \le \mathfrak N_{T,\kappa}^{\rm Zv}.
 \end{gathered}
\]

Differentiating \eqref{eq:zvonkin-duhamel} gives
\begin{equation}\label{eq:abstract-J-fixed-point}
  J-1=2\Kcal_{\gamma,\lambda}(B\diamond J).
\end{equation}
Use the controlled ansatz
\[
  J=1+2J\mpara Q_{X,\lambda}+J^\sharp,
  \qquad
  J^\sharp\in C([t_0,T];\Ccal^{\beta+q_{\gamma,\alpha}-}).
\]
Substitution produces the usual controlled-product commutator and
\[
 Q_{X,\lambda}\odot_{\rm ren}B
 =\mathbb X_{Q,\lambda}+Q_{X,\lambda}\odot\Xone.
\]
When \(2q_{\gamma,\alpha}\le\delta_\gamma\), the commutator is controlled by
part~(ii) of \Cref{lem:fractional-commutators}; when
\(2q_{\gamma,\alpha}>\delta_\gamma\), ordinary Bony continuity suffices.
The first term is supplied by the uniform resolvent family
\eqref{eq:resolvent-model-uniform} when
\(2q_{\gamma,\alpha}\le\delta_\gamma\), and is otherwise classical.  Since
\(\rho<3q_{\gamma,\alpha}-\delta_\gamma\), the second has regularity
\[
  3q_{\gamma,\alpha}-\delta_\gamma-\kappa>\rho>0
\]
for sufficiently small \(\kappa\).  Moreover, \(B\odot J^\sharp\) is
classical because \(\beta+2q_{\gamma,\alpha}-\delta_\gamma>0\).  Thus no new
Jacobian-specific SBE tree is needed. The controlled fixed point reduces to the resolvent gain, on frequency
\(2^j\),
\[
 \frac{2^{j(\gamma-(\beta_+-\beta))}}
 {\lambda+2^{\gamma j}}
 \lesssim\lambda^{-(\beta_+-\beta)/\gamma}.
\]
Hence \eqref{eq:abstract-J-fixed-point} is contractive for large
\(\lambda\), and
\[
 \norm{J-1}_{C([t_0,T];C^\beta)}
 \lesssim\lambda^{-(\beta_+-\beta)/\gamma}
 \mathfrak C_T\bigl(\mathfrak N_{T,\kappa}^{\rm Zv}\bigr),
\]
with local Lipschitz dependence.  The short-time Schauder factor gives the
restarted assertion at \(\lambda=0\).

Let \(\psi_0\) be the zero-mean primitive of \(J-1\), and define its missing
zero mode by
\begin{equation}\label{eq:psi-zero-mode}
  \dot{\bar\psi}(t)+\lambda\bar\psi(t)
  =2\ip{B\diamond J}{1},
  \qquad \bar\psi(t_0)=0,
  \qquad \psi:=\psi_0+\bar\psi.
\end{equation}
Equation \eqref{eq:abstract-J-fixed-point} makes the residual spatially
constant, and \eqref{eq:psi-zero-mode} cancels its mean.  The scalar
resolvent formula gives
\[
  \abs{\bar\psi(t)}
  \lesssim\lambda^{-1}\mathfrak C_T
    \bigl(\mathfrak N_{T,\kappa}^{\rm Zv}\bigr)
  \lesssim\lambda^{-(\beta_+-\beta)/\gamma}\mathfrak C_T
    \bigl(\mathfrak N_{T,\kappa}^{\rm Zv}\bigr),
\]
because \(0<(\beta_+-\beta)/\gamma<1\).  The primitive estimate yields
\begin{equation}\label{eq:zvonkin-rate}
 \norm{\psi}_{C([t_0,T];C^{1+\beta})}
 +\norm{J-1}_{C([t_0,T];C^\beta)}
 \lesssim
 \lambda^{-(\beta_+-\beta)/\gamma}
 \mathfrak C_T\bigl(\mathfrak N_{T,\kappa}^{\rm Zv}\bigr).
\end{equation}
This implies \eqref{eq:zvonkin-schauder-estimate}.  At \(\lambda=0\), use
\(\bar\psi(t)=2\int_{t_0}^t\langle B\diamond J,1\rangle\dd s\).

Let \((\psi^n,J^n)\) be the solutions associated with the smooth enhanced
approximations in the proposition, and set
\(\Phi^n:=(\Psi^n)^{-1}\).  The difference estimate gives
\(J^n\to J\) in
\(C([t_0,T];C^\beta)\) and \(\psi^n\to\psi\) in
\(C([t_0,T];C^{1+\beta})\).  Choose \(\lambda\) so that
\eqref{eq:zvonkin-rate} is at most \(1/2\).  Then
\(1/2\le J^n,J\le3/2\), proving \eqref{eq:J-bounds}.  Finally,
\[
  \partial_y\Phi=(J^\Phi)^{-1},
\]
and the \(C^\beta\) composition estimate gives uniform bi-Lipschitz bounds
and uniform convergence of \(\Phi^n\) to \(\Phi\).
\end{proof}

\begin{samepage}
We next identify the transformed flux operator for \(1<\gamma<2\).  Define the
odd periodic primitive of the fractional kernel by
\[
 Q_\gamma^{\mathrm{per}}(r)
 :=\frac1\gamma\sum_{n\in\mathbb Z}
 \frac{\operatorname{sgn}(r+n)}{\abs{r+n}^{\gamma}},
 \qquad r\notin\mathbb Z.
\]
\end{samepage}

\begin{figure}[!ht]
\centering
\includegraphics[width=.82\textwidth]{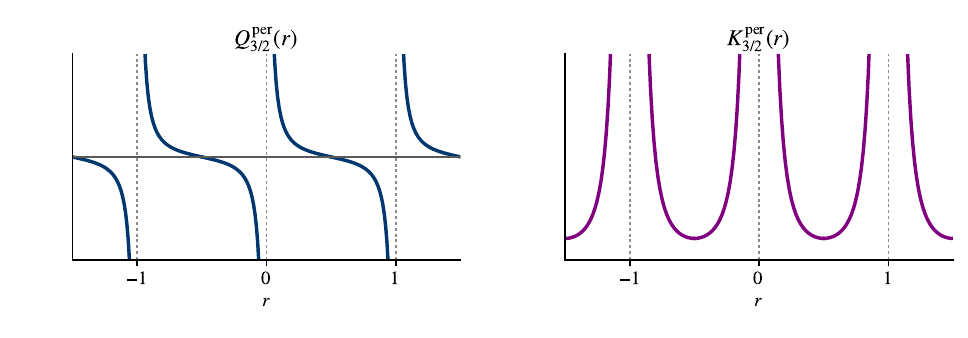}
\caption{Three periods of the kernels at \(\gamma=\frac{3}{2}\); the
vertical ranges are truncated at the singularities \(r\in\mathbb Z\).
The primitive
\(Q_\gamma^{\mathrm{per}}\) is odd and decreases on each side of every
singularity, while
\(K_\gamma^{\mathrm{per}}=-\partial_rQ_\gamma^{\mathrm{per}}\) is positive
and even.}
\label{fig:periodic-fractional-kernels}
\end{figure}

\begin{samepage}
Integration by parts in the singular-integral formula for
\(\Rcal_\gamma\) gives, for smooth \(f\),
\begin{equation}\label{eq:R-compensated-kernel}
 \Rcal_\gamma f(y)
 =c_\gamma\int_{-\frac{1}{2}}^{\frac{1}{2}}
 [f(y)-f(y-h)]Q_\gamma^{\mathrm{per}}(h)\dd h.
\end{equation}
We shall use the properties
\[
 \partial_rQ_\gamma^{\mathrm{per}}=-K_\gamma^{\mathrm{per}},
 \qquad
 Q_\gamma^{\mathrm{per}}(-r)=-Q_\gamma^{\mathrm{per}}(r),
 \qquad
 Q_\gamma^{\mathrm{per}}(\pm\tfrac12)=0.\footnote{At
 \(r=\tfrac12\), pair
 the summand indexed by \(n\) with that indexed by \(-n-1\).  Their
 arguments are opposite, so the two terms cancel.  This rearrangement is
 legitimate because the series is absolutely convergent for
 \(\gamma>1\).}
\]
\end{samepage}

The transformed flux operator admits the following \(y\)-coordinate kernel
representation.

\begin{lemma}[Kernel of the transformed fractional flux operator]
\label{lem:transformed-flux-kernel}
Let \(1<\gamma<2\), \(0<\beta<1\), and let
\(\Psi\in C^{1+\beta}(\T)\) be an orientation-preserving bi-Lipschitz map.
Choose degree-one lifts of \(\Psi\) and \(\Phi\) to \(\mathbb R\).
For every smooth \(f\),
\begin{equation}\label{eq:D-transformed-kernel}
\begin{aligned}
 \Dcal_\Psi^\gamma f(y)
 =c_\gamma\int_{-\frac{1}{2}}^{\frac{1}{2}}
 [f(y)-f(y-h)]
 Q_\gamma^{\mathrm{per}}
 \bigl(\Phi(y)-\Phi(y-h)\bigr)\dd h.
\end{aligned}
\end{equation}
Moreover,
\begin{equation}\label{eq:pullback-Q-kernel-error}
 \bigl|Q_\gamma^{\mathrm{per}}
  \bigl(\Phi(y)-\Phi(y-h)\bigr)
  -(J^\gamma)^\Phi(y)Q_\gamma^{\mathrm{per}}(h)\bigr|
 \lesssim_\Psi\abs h_\T^{-\gamma+\beta}.
\end{equation}
\end{lemma}

\begin{proof}
It is enough first to take \(f\) and \(\Psi\) smooth.  Set
\[
 F(y):=\int_0^y f(z)\dd z.
\]
Then \(F'=f\), and \(F\), \(F^\Psi\), and \(\Psi\) are smooth
affine-periodic functions.  Fix \(y\) and put \(x=\Phi(y)\).
The calculation leading to \eqref{eq:chain-commutator}, with \(F\) in
place of \(M\), gives
\[
 \Dcal_\Psi^\gamma f(y)
 =-\Lambda^\gamma(F^\Psi)(x)
  +f(y)\Lambda^\gamma\Psi(x).
\]
By the affine-periodic principal-value convention in
\hyperref[app:notation]{Appendix~A}, this is directly
\[
\begin{aligned}
 \Dcal_\Psi^\gamma f(y)
 =c_\gamma\PV\int_{-\frac{1}{2}}^{\frac{1}{2}}
 \biggl[
 &F(\Psi(x-\rho))-F(y)\\
 &+f(y)[y-\Psi(x-\rho)]
 \biggr]K_\gamma^{\mathrm{per}}(\rho)\dd\rho .
\end{aligned}
\]
Here \(F(y)\) is the compensating constant in the
singular-integral formula, permitted by
\(\Lambda^\gamma1=0\); it is not integrated separately against the
singular kernel.  The affine parts of \(F^\Psi\) and \(\Psi\)
contribute only multiples of
\(\rho K_\gamma^{\mathrm{per}}(\rho)\), whose symmetric principal values
vanish.  The two first-order increments in the displayed bracket cancel, so the
combined expression is \(O(\rho^2)\) at the origin and the displayed
integral is in fact absolutely convergent. We next set
\[
 h=y-\Psi(x-\rho).
\]
Since
\[
 \frac{\dd h}{\dd\rho}=J(x-\rho)>0,
 \qquad
 x-\rho=\Phi(y-h),
\]
the centered \(\rho\)-interval is mapped onto
\(
 (
 y-\Psi (x+\frac{1}{2}),
 y-\Psi (x-\frac{1}{2})
 ]
\),
whose length is one. Moreover, Taylor's formula gives
\(F(y-h)-F(y)+hf(y)=O(h^2)\).
The change of variables therefore yields the absolutely convergent
integral
\[
\begin{aligned}
 \Dcal_\Psi^\gamma f(y)
 =c_\gamma\int_{y-\Psi(x+\frac{1}{2})}
                     ^{y-\Psi(x-\frac{1}{2})}
 &[F(y-h)-F(y)+hf(y)]\\
 &{}\times
 \frac{K_\gamma^{\mathrm{per}}
  \bigl(\Phi(y)-\Phi(y-h)\bigr)}
 {J\bigl(\Phi(y-h)\bigr)}
 \dd h.
\end{aligned}
\]

Finally, the two differentiated factors satisfy
\[
\begin{aligned}
 -\partial_h Q_\gamma^{\mathrm{per}}
 \bigl(\Phi(y)-\Phi(y-h)\bigr)
 &=
 \frac{K_\gamma^{\mathrm{per}}
  \bigl(\Phi(y)-\Phi(y-h)\bigr)}
 {J\bigl(\Phi(y-h)\bigr)},\\
 \partial_h
 [F(y-h)-F(y)+hf(y)]
 &=f(y)-f(y-h).
\end{aligned}
\]
At \(h=0\), the boundary product is
\(O(\abs h^{2-\gamma})\).  At the two endpoints, the argument of
\(Q_\gamma^{\mathrm{per}}\) is respectively
\(-\frac{1}{2}\) and \(\frac{1}{2}\), where it vanishes.  Integration by
parts gives
\[
\begin{aligned}
 \Dcal_\Psi^\gamma f(y)
 =c_\gamma
 \int_{y-\Psi(x+\frac{1}{2})}
      ^{y-\Psi(x-\frac{1}{2})}
 [f(y)-f(y-h)]Q_\gamma^{\mathrm{per}}
 \bigl(\Phi(y)-\Phi(y-h)\bigr)\dd h.
\end{aligned}
\]
This proves \eqref{eq:D-transformed-kernel}, since the integrand is
one-periodic in \(h\).  It remains to freeze the kernel.  The fundamental
theorem of calculus
gives
\[
 \Phi(y)-\Phi(y-h)
 =h\int_0^1\Phi'(y-th)\dd t
 =\frac h{J^\Phi(y)}+O_\Psi(\abs h_\T^{1+\beta}).
\]
Near zero,
\[
 Q_\gamma^{\mathrm{per}}(r)
 =\frac1\gamma\frac{\operatorname{sgn}r}{\abs r^\gamma}
  +\text{a smooth function}.
\]
The pulled-back increment has the same sign as \(h\).  The mean-value theorem
therefore gives \eqref{eq:pullback-Q-kernel-error}; away from zero the
kernel is smooth.  Approximation of \(\Psi\) by smooth degree-one lifts
completes the proof.
\end{proof}

\begin{proposition}[Paralinearization of the transformed fractional flux operator]
\label{prop:transformed-operator-paralinearization}
Let \(1<\gamma<2\), \(0<\beta<1\), and let
\(\Psi\in C^{1+\beta}(\T)\) be an orientation-preserving bi-Lipschitz map
with Jacobian satisfying \eqref{eq:J-bounds}.
The operator \(\Dcal_{\Psi}^{\gamma}\) defined in
\eqref{eq:Dflux} admits the paralinearization
\begin{equation}\label{eq:D-paralinearization}
  \Dcal_{\Psi}^{\gamma}f
  =(J^\gamma)^\Phi\prec\Rcal_\gamma f
  +\mathcal S_\Psi f,
\end{equation}
where \(\prec\) is the spatial Bony paraproduct.  For every
\(f\in H^{\gamma/2}\) and every
\begin{equation}\label{eq:S-output-index}
 \sigma<\min\{\gamma/2,1-\gamma/2+\beta\},
\end{equation}
the remainder satisfies
\begin{equation}\label{eq:S-form-bound}
  \norm{\mathcal S_\Psi f}_{H^\sigma}
  \lesssim\norm f_{H^{\gamma/2}}.
\end{equation}
Thus \(\mathcal S_\Psi\) has order at most
\(\max\{0,\gamma-1-\beta\}+\).  Consequently,
\begin{align}
  \norm{\Dcal_{\Psi}^{\gamma}f}_{H^{1-\gamma/2}}
  &\lesssim\norm f_{H^{\gamma/2}},
  \label{eq:D-form-bound}\\
  \norm{\Acal_{\Psi}^{\gamma}f}_{H^{-\gamma/2}}
  &\lesssim\norm f_{H^{\gamma/2}}.
  \label{eq:A-form-bound}
\end{align}
The map
\[
 \Psi\longmapsto\Acal_\Psi^\gamma
 \quad\text{with values in}\quad
 \mathcal L(H^{\gamma/2},H^{-\gamma/2})
\]
is continuous on bounded subsets of \(C^{1+\beta}\) with fixed
bi-Lipschitz constants.  After any loss from \(\beta\) to
\(\beta_0<\beta\), this dependence has a local H\"older modulus.  All
displayed bounds are locally uniform under these assumptions.
\end{proposition}

\begin{proof}
It suffices first to take smooth \(f\) and \(\Psi\).  Recall that
\[
 \Phi'=(J^\Phi)^{-1},
 \qquad
 (\Phi')^{-\gamma}=(J^\gamma)^\Phi.
\]
For \(h\ne0\), set
\[
 \nabla_h^-\Phi(y)
 :=\frac{\Phi(y)-\Phi(y-h)}h
 =\int_0^1\Phi'(y-th)\dd t,
\]
and set \(\nabla_0^-\Phi:=\Phi'\).
We use the increment convention from
\hyperref[app:notation]{Appendix~A} in the form
\(\delta_hf(y-h)=f(y)-f(y-h)\).

The only singularity of the pulled-back kernel is at \(h=0\).  Choose a
fixed smooth even cutoff \(\vartheta\), equal to one near zero and supported
in a neighborhood on which the pulled-back increment meets no other lattice
singularity.  Its support can be chosen uniformly from the bi-Lipschitz
bounds.  Since
\[
 \Phi(y)-\Phi(y-h)=h\nabla_h^-\Phi(y)
\]
and \(\Phi'>0\), the pulled-back increment has the same sign as \(h\).  Moreover,
for \(\abs r\) sufficiently small,
\[
 Q_\gamma^{\mathrm{per}}(r)
 -\frac1\gamma\frac{\operatorname{sgn}r}{\abs r^\gamma}
 =\frac1\gamma\sum_{n\ne0}
 \frac{\operatorname{sgn}(r+n)}{\abs{r+n}^\gamma}
 \in C^\infty.
\]
Away from the support of \(\vartheta\), the pulled-back increment is uniformly
separated from the lattice singularities.  Consequently,
\begin{equation}\label{eq:divided-difference-singular-kernel}
 Q_\gamma^{\mathrm{per}}\bigl(\Phi(y)-\Phi(y-h)\bigr)
 -\frac{\vartheta(h)}{\gamma}
  \frac{\operatorname{sgn}h}{\abs h^\gamma}
  \bigl(\nabla_h^-\Phi(y)\bigr)^{-\gamma}
 \in C^{1+\beta}(\T_y\times\T_h).
\end{equation}
By \eqref{eq:D-transformed-kernel}, the standard variable-kernel estimate
applied to the difference in
\eqref{eq:divided-difference-singular-kernel} gives, for every
\(\sigma<\gamma/2\),
\[
\norm[\Big]{
 \Dcal_\Psi^\gamma f
 -\frac{c_\gamma}{\gamma}
 \int_{-\frac12}^{\frac12}
 \delta_hf(\cdot-h)
 \frac{\vartheta(h)\operatorname{sgn}h}{\abs h^\gamma}
 \bigl(\nabla_h^-\Phi(\cdot)\bigr)^{-\gamma}
 \dd h
 }_{H^\sigma}
 \lesssim\norm f_{H^{\gamma/2}}.
\]
Thus this regular kernel contributes to \(\mathcal S_\Psi f\), and it
remains to paralinearize
\begin{equation}
\begin{aligned}
 \frac{c_\gamma}{\gamma}
 \int_{-\frac12}^{\frac12}
 &\delta_hf(y-h)
 \frac{\vartheta(h)\operatorname{sgn}h}{\abs h^\gamma}
 \bigl(\nabla_h^-\Phi(y)\bigr)^{-\gamma}
 \dd h.
\end{aligned}
\label{eq:reduced-divided-difference-integral}
\end{equation}
Let \(f_k=\Delta_kf\).  For each \(f_k\), Bony's decomposition splits the
\(y\)-frequencies of the coefficient according to
\(j\ll k\), \(j\sim k\), and \(j\gg k\) respectively.  The terms
with \(j\ll k\) are
\[
\begin{aligned}
 \frac{c_\gamma}{\gamma}\sum_{k\ge-1}
 \int_{-\frac12}^{\frac12}
 \frac{\vartheta(h)\operatorname{sgn}h}{\abs h^\gamma}
 \delta_hf_k(y-h)
 \Delta_{<k-1}
 \bigl(\nabla_h^-\Phi(y)\bigr)^{-\gamma}
 \dd h.
\end{aligned}
\]
In this expression, add and subtract
\[
 \Delta_{<k-1}(\Phi')^{-\gamma}
 =\Delta_{<k-1}\bigl((J^\gamma)^\Phi\bigr).
\]
The resulting frozen term is
\begin{equation}
\begin{aligned}
 \sum_{k\ge-1}
 &\Delta_{<k-1}\bigl((J^\gamma)^\Phi\bigr)(y)
 \frac{c_\gamma}{\gamma}
 \int_{-\frac{1}{2}}^{\frac{1}{2}}
 \frac{\vartheta(h)\operatorname{sgn}h}{\abs h^\gamma}
 \delta_hf_k(y-h)\dd h.
\end{aligned}
\label{eq:leading-low-high-kernel}
\end{equation}
On the other hand, \eqref{eq:R-compensated-kernel} gives
\[
\begin{aligned}
&(J^\gamma)^\Phi\prec\Rcal_\gamma f(y)
\\
&\quad=c_\gamma\sum_{k\ge-1}
 \Delta_{<k-1}\bigl((J^\gamma)^\Phi\bigr)(y)
 \int_{-\frac12}^{\frac12}
 \delta_hf_k(y-h)Q_\gamma^{\mathrm{per}}(h)\dd h.
\end{aligned}
\]
Their difference is
\[
\begin{aligned}
c_\gamma\sum_{k\ge-1}
 \Delta_{<k-1}\bigl((J^\gamma)^\Phi\bigr)(y)
 \int_{-\frac12}^{\frac12}\delta_hf_k(y-h)
\biggl(
 \frac{\vartheta(h)}{\gamma}
 \frac{\operatorname{sgn}h}{\abs h^\gamma}
 -Q_\gamma^{\mathrm{per}}(h)
 \biggr)\dd h.
\end{aligned}
\]
The parenthesized kernel is a smooth odd periodic function and therefore has zero
mean.  Its Fourier coefficients decrease faster than every power, so for
every \(N>0\),
\[
\norm[\Big]{
 \int_{-\frac12}^{\frac12}
 \delta_hf_k(\cdot-h)
 \biggl(
 \frac{\vartheta(h)}{\gamma}
 \frac{\operatorname{sgn}h}{\abs h^\gamma}
 -Q_\gamma^{\mathrm{per}}(h)
 \biggr)\dd h
}_{L^2}
\lesssim_N 2^{-Nk}\norm{f_k}_{L^2}.
\]
Thus \eqref{eq:leading-low-high-kernel} equals
\((J^\gamma)^\Phi\prec\Rcal_\gamma f\) modulo
\(\mathcal S_\Psi f\).

We now choose the parameters needed to estimate the difference left by the
preceding freezing.  If \(\sigma\le0\), it suffices to prove the estimate
at any positive admissible exponent and use Sobolev embedding.  We may
therefore assume \(\sigma>0\), and choose
\[
 0<\beta_0<\beta,\qquad
 \beta_0\ne\gamma-1,\qquad
 \sigma<\min\{\gamma/2,1-\gamma/2+\beta_0\}.
\]
Such a choice is possible by \eqref{eq:S-output-index}; excluding
\(\beta_0=\gamma-1\) avoids one logarithmic integral.  All constants below
are locally uniform under the stated \(C^{1+\beta}\) and bi-Lipschitz
bounds.  Since \(r\mapsto r^{-\gamma}\) is Lipschitz on the fixed positive
range of \(\nabla_h^-\Phi\),
\begin{equation}\label{eq:divided-difference-freeze}
\begin{aligned}
 \norm{
 \bigl(\nabla_h^-\Phi\bigr)^{-\gamma}
 -(\Phi')^{-\gamma}
 }_{L^\infty}
 &\lesssim
 \sup_y\int_0^1
 \abs{\Phi'(y-th)-\Phi'(y)}\dd t
 \lesssim\abs h^{\beta_0}.
\end{aligned}
\end{equation}
For \(k\ge0\), \eqref{eq:divided-difference-freeze}, translation
invariance, and Bernstein's inequality give
\[
 \norm{\delta_hf_k(\cdot-h)}_{L^2}
 =\norm{\delta_hf_k}_{L^2}
 \lesssim\min\bigl\{1,2^k\abs h\bigr\}\norm{f_k}_{L^2}
\]
and hence
\begin{equation}\label{eq:divided-difference-low-high}
\begin{aligned}
&\Bigl\|
 \int_{-\frac12}^{\frac12}
 \frac{\vartheta(h)\operatorname{sgn}h}{\abs h^\gamma}
 \delta_hf_k(\cdot-h) \cdot
 \Delta_{<k-1}\!\bigl(
 \bigl(\nabla_h^-\Phi\bigr)^{-\gamma}
 -(\Phi')^{-\gamma}
 \bigr)\dd h
 \Bigr\|_{L^2}
\\
&\quad\lesssim
 2^{k(\gamma-1-\beta_0)_+}\norm{f_k}_{L^2}.
\end{aligned}
\end{equation}

This bound follows by splitting the resulting \(h\)-integral at
\(h=2^{-k}\). The expression in \eqref{eq:divided-difference-low-high} is Fourier
localized in an annulus at scale \(2^k\).
Since
\(\sigma+(\gamma-1-\beta_0)_+<\gamma/2\), annular
Littlewood--Paley summation places this error in
\(\mathcal S_\Psi f\).  Thus the terms with \(j\ll k\) equal
\((J^\gamma)^\Phi\prec\Rcal_\gamma f\) modulo
\(\mathcal S_\Psi f\). The dyadic calculations are given in
\Cref{app:paralinearization-dyadic-calculations}; here we retain the
two remaining estimates and their frequency supports.  For \(j\sim k\),
\begin{equation}\label{eq:divided-difference-comparable}
\begin{aligned}
&\Bigl\|
 \int_{-\frac12}^{\frac12}
 \frac{\vartheta(h)\operatorname{sgn}h}{\abs h^\gamma}
 \Delta_j\bigl(\nabla_h^-\Phi\bigr)^{-\gamma}
 \delta_hf_k(\cdot-h)\dd h
 \Bigr\|_{L^2}\lesssim
 2^{k(\gamma-1-\beta_0)_+}\norm{f_k}_{L^2},
\end{aligned}
\end{equation}
and its output has Fourier support in a ball at scale \(2^k\).
For \(k\ge0\) and \(j\gg k\), the corresponding estimate is
\begin{equation}\label{eq:divided-difference-high}
 \begin{aligned}
&\Bigl\|
 \int_{-\frac12}^{\frac12}
 \frac{\vartheta(h)\operatorname{sgn}h}{\abs h^\gamma}
 \Delta_j\bigl(\nabla_h^-\Phi\bigr)^{-\gamma}
 \delta_hf_k(\cdot-h)\dd h
 \Bigr\|_{L^2}
\\
&\qquad\lesssim
 \begin{cases}
  2^{k(\gamma-1-\beta_0)}
  2^{-(2-\gamma+\beta_0)(j-k)}\norm{f_k}_{L^2},
  &\beta_0<\gamma-1,\\[2mm]
  2^{-(j-k)}\norm{f_k}_{L^2},
 &\beta_0>\gamma-1.
 \end{cases}
 \end{aligned}
\end{equation}
Its output is Fourier localized in an annulus at scale \(2^j\).  The
estimates \eqref{eq:divided-difference-low-high} and
\eqref{eq:divided-difference-comparable} are summable because
\[
 \sigma+(\gamma-1-\beta_0)_+<\frac\gamma2.
\]
The factors in \(j-k\) in \eqref{eq:divided-difference-high} are summable
because
\[
 \begin{cases}
 2-\gamma+\beta_0-\sigma>0,
 &\beta_0<\gamma-1,\\
 1-\sigma>0,
 &\beta_0>\gamma-1,
 \end{cases}
\]
while the input-frequency powers are controlled by the preceding
inequality.  The annular and ball Littlewood--Paley summations, the
bottom input block, and the two Schur sums are carried out in
\Cref{app:paralinearization-dyadic-calculations}.  They place
\eqref{eq:divided-difference-comparable} and
\eqref{eq:divided-difference-high} in \(\mathcal S_\Psi f\).
Together with the regular kernel in
\eqref{eq:divided-difference-singular-kernel}, the smooth cutoff
replacement, and \eqref{eq:divided-difference-low-high}, this proves
\eqref{eq:D-paralinearization} and \eqref{eq:S-form-bound}.

It remains to derive the operator order, form bounds, and dependence on
\(\Psi\).
The order of the remainder follows from
\[
 \frac\gamma2-
 \min\{\gamma/2,1-\gamma/2+\beta\}
 =\max\{0,\gamma-1-\beta\}.
\]
Since \(\Rcal_\gamma\) has order \(\gamma-1\),
\[
 \norm{(J^\gamma)^\Phi
 \prec\Rcal_\gamma f}_{H^{1-\gamma/2}}
 \lesssim
 \norm{(J^\gamma)^\Phi}_{L^\infty}
 \norm{\Rcal_\gamma f}_{H^{1-\gamma/2}}
 \lesssim\norm f_{H^{\gamma/2}}.
\]
The admissible interval contains \(1-\gamma/2\), since
\[
 1-\frac\gamma2<\frac\gamma2
 \quad(\gamma>1),
 \qquad
 1-\frac\gamma2<1-\frac\gamma2+\beta.
\]
Taking
\[
 1-\frac\gamma2
 <\sigma<
 \min\{\gamma/2,1-\gamma/2+\beta\}
\]
in \eqref{eq:S-form-bound} proves \eqref{eq:D-form-bound}.  Finally,
\[
\begin{aligned}
 \norm{\Acal_\Psi^\gamma f}_{H^{-\gamma/2}}
 =\norm{\partial_y\Dcal_\Psi^\gamma f}_{H^{-\gamma/2}} \lesssim
 \norm{\Dcal_\Psi^\gamma f}_{H^{1-\gamma/2}}
 \lesssim\norm f_{H^{\gamma/2}},
\end{aligned}
\]
which proves \eqref{eq:A-form-bound}.

It remains to record the dependence needed for the approximation limit.
Let \(\Psi_1,\Psi_2\) belong to a bounded \(C^{1+\beta}\) set with common
bi-Lipschitz bounds, and write \(\Phi_i:=\Psi_i^{-1}\).  Subtracting the
two kernel decompositions controls
the smooth kernel correction by
\(\norm{\Phi_1-\Phi_2}_{C^{1+\beta_0}}\), while the principal
paraproduct and the cutoff-kernel correction are controlled by
\[
 \norm{(J_1^\gamma)^{\Phi_1}
       -(J_2^\gamma)^{\Phi_2}}_{C^{\beta_0}}.
\]
The difference estimates in
\Cref{app:paralinearization-dyadic-calculations} give the same factors for
the three frequency relations.  Apply these estimates at any exponent
strictly between
\[
 1-\frac\gamma2
 \quad\text{and}\quad
 \min\{\gamma/2,1-\gamma/2+\beta_0\}
\]
and differentiate the resulting flux bound.  This gives
\begin{equation}\label{eq:transformed-operator-difference}
\begin{aligned}
\norm{\Acal_{\Psi_1}^\gamma-\Acal_{\Psi_2}^\gamma}_
 {\mathcal L(H^{\gamma/2},H^{-\gamma/2})}
\lesssim
 \norm{\Phi_1-\Phi_2}_{C^{1+\beta_0}}
 +
 \norm{(J_1^\gamma)^{\Phi_1}
       -(J_2^\gamma)^{\Phi_2}}_{C^{\beta_0}}.
\end{aligned}
\end{equation}
Standard inversion and composition estimates yield, when
\(\norm{\Psi_1-\Psi_2}_{C^{1+\beta}}\le1\),
\[
\begin{aligned}
&\norm{\Phi_1-\Phi_2}_{C^{1+\beta_0}}
 +
 \norm{(J_1^\gamma)^{\Phi_1}
       -(J_2^\gamma)^{\Phi_2}}_{C^{\beta_0}}
\lesssim
 \norm{\Psi_1-\Psi_2}_{C^{1+\beta}}^{\beta-\beta_0}.
\end{aligned}
\]
Approximation of \(\Psi\) by smooth degree-one diffeomorphisms in
\(C^{1+\beta_0}\), with uniform \(C^{1+\beta}\) and bi-Lipschitz bounds,
followed by density of smooth \(f\) in \(H^{\gamma/2}\), completes the
proof.
\end{proof}

\begin{lemma}[Entropy and energy identities for $\Acal_\Psi^\gamma$]
\label{lem:transformed-operator-entropy-identity}
Let \(1<\gamma\le2\), \(0<\beta<1\), and let
\(\Psi\in C^{1+\beta}(\T)\) be an orientation-preserving bi-Lipschitz map
with Jacobian satisfying \eqref{eq:J-bounds}.  Given \(\eta\in C^2(\R)\),
write
\[
 \mathfrak B_\eta(u\mid v)
 :=\eta(u)-\eta(v)-\eta'(v)(u-v)
\]
for its Bregman remainder.  If \(1<\gamma<2\), then every smooth
real-valued function \(f\) satisfies
\[
 \ip{\eta'(f)}{\Acal_\Psi^\gamma f}
 =c_\gamma\iint_{\T^2}
 J(x)\mathfrak B_\eta\bigl(f^\Psi(x)\mid f^\Psi(z)\bigr)
 K_\gamma^{\mathrm{per}}(x-z)\dd x\dd z.
\]
In particular, this pairing is nonnegative when \(\eta\) is convex.
At \(\gamma=2\), with \(a=(J^2)^\Phi\), the corresponding local identity is
\[
 \ip{\eta'(f)}{\Acal_\Psi^2 f}
 =\int_\T a\,\eta''(f)\abs{f_y}^2\dd y
 =\int_\T J\,\eta''(f^\Psi)
   \abs{\partial_x(f^\Psi)}^2\dd x,
\]
and is again nonnegative for convex \(\eta\).  Returning to
\(1<\gamma<2\), take \(\eta(r)=r^2/2\).  The resulting quadratic
identity extends by density to every
real-valued \(f\in H^{\gamma/2}(\T)\):
\begin{equation}\label{eq:transformed-operator-energy-kernel}
\begin{aligned}
 \ip f{\Acal_\Psi^\gamma f}
 =\frac{c_\gamma}{4}\iint_{\T^2}
 [J(x)+J(z)]\abs{f^\Psi(x)-f^\Psi(z)}^2
 K_\gamma^{\mathrm{per}}(x-z)\dd x\dd z.
\end{aligned}
\end{equation}
Its polarization is, for real-valued
\(f,g\in H^{\gamma/2}(\T)\),
\begin{equation}\label{eq:transformed-operator-polarized-kernel}
\begin{aligned}
 \frac12\bigl(
 \ip f{\Acal_\Psi^\gamma g}
 +\ip g{\Acal_\Psi^\gamma f}\bigr)
 =\frac{c_\gamma}{4}\iint_{\T^2}
 &[J(x)+J(z)]
 [f^\Psi(x)-f^\Psi(z)]\\
 &{}\times[g^\Psi(x)-g^\Psi(z)]
 K_\gamma^{\mathrm{per}}(x-z)\dd x\dd z.
\end{aligned}
\end{equation}
\end{lemma}

\begin{proof}
Assume first that \(\Psi\) is smooth and that \(1<\gamma<2\).  Periodic
integration by parts, the change of variables
\(y=\Psi(x)\), and \eqref{eq:Dflux} give
\[
\begin{aligned}
 \ip{\eta'(f)}{\Acal_\Psi^\gamma f}
 &=\int_\T \partial_x\eta'(f^\Psi)
 \bigl[
  \Rcal_\gamma(Jf^\Psi)+f^\Psi\Lambda^\gamma\psi
 \bigr]\dd x\\
 &=\int_\T
 \eta'(f^\Psi)\Lambda^\gamma(Jf^\Psi)\dd x\\
 &\quad+\int_\T
 [\eta(f^\Psi)-f^\Psi\eta'(f^\Psi)]\Lambda^\gamma J\dd x.
\end{aligned}
\]
Here we used \(\partial_x\Rcal_\gamma=-\Lambda^\gamma\),
\(\partial_x\Lambda^\gamma\psi=\Lambda^\gamma J\), and
\(f^\Psi\partial_x\eta'(f^\Psi)
=\partial_x[f^\Psi\eta'(f^\Psi)-\eta(f^\Psi)]\).
Use the periodic kernel
representation of \(\Lambda^\gamma\) and symmetrize in \(x\) and \(z\).
After extracting the common factor \(c_\gamma/2\), set
\(u=f^\Psi(x)\) and \(v=f^\Psi(z)\).  The coefficient of \(J(x)\) is
\[
\begin{aligned}
 [\eta'(u)-\eta'(v)]u
 &+\eta(u)-u\eta'(u)-\eta(v)+v\eta'(v)\\
 &=\eta(u)-\eta(v)-\eta'(v)(u-v)
 =\mathfrak B_\eta(u\mid v),
\end{aligned}
\]
and the coefficient of \(J(z)\) is
\(\mathfrak B_\eta(v\mid u)\).  Thus the pairing equals
\[
 \frac{c_\gamma}{2}\iint_{\T^2}
 \begin{aligned}[t]
 \Bigl\{&J(x)\mathfrak B_\eta\bigl(f^\Psi(x)\mid f^\Psi(z)\bigr)\\
 &{}+J(z)\mathfrak B_\eta\bigl(f^\Psi(z)\mid f^\Psi(x)\bigr)\Bigr\}
 K_\gamma^{\mathrm{per}}(x-z)\dd x\dd z.
 \end{aligned}
\]
Since the kernel is even, exchanging \(x\) and \(z\) in the second
term turns it into the first and proves the stated identity.  The local
identity follows from \eqref{eq:local-check}, an integration by parts,
and the change of variables \(y=\Psi(x)\).

For \(\eta(r)=r^2/2\), one has
\(\mathfrak B_\eta(r\mid s)=\abs{r-s}^2/2\).  The diagonal identity and
its polarization give \eqref{eq:transformed-operator-energy-kernel} and
\eqref{eq:transformed-operator-polarized-kernel}, respectively.  If
\(I_f\subset\R\) contains the range of \(f^\Psi\), then for
\(r,s\in I_f\), Taylor's formula gives
\(\abs{\mathfrak B_\eta(r\mid s)}
\le\frac12\norm{\eta''}_{L^\infty(I_f)}\abs{r-s}^2\).
Smooth approximation of \(\Psi\) and dominated convergence extend the
nonlocal identity to the stated maps; the local identity passes directly.
In the quadratic case, boundedness of bi-Lipschitz composition on
\(H^{\gamma/2}\) and density extend
\eqref{eq:transformed-operator-energy-kernel}--
\eqref{eq:transformed-operator-polarized-kernel} to the stated Sobolev
spaces.
\end{proof}

\begin{proposition}[Coercivity of the transformed fractional diffusion operator]
\label{prop:transformed-operator-coercivity}
Let \(1<\gamma<2\), \(0<\beta<1\), and let
\(\Psi\in C^{1+\beta}(\T)\) be an orientation-preserving bi-Lipschitz map
with Jacobian satisfying \eqref{eq:J-bounds}.  For
\(f\in H^{\gamma/2}(\T)\),
\begin{equation}\label{eq:A-coercive-form}
  \ip f{\Acal_{\Psi}^{\gamma}f}
  \asymp_\Psi\norm f_{\dot H^{\gamma/2}}^2.
\end{equation}
\end{proposition}

\begin{proof}
By \eqref{eq:transformed-operator-energy-kernel} and the bounds
\(J_*\le J\le J^*\), the energy form is comparable to
\(\norm{f^\Psi}_{\dot H_x^{\gamma/2}}^2\).
Under the change of variables
\(y=\Psi(x)\), \(w=\Psi(z)\), the Jacobian factors are uniformly bounded
and
\[
 K_\gamma^{\mathrm{per}}
 \bigl(\Phi(y)-\Phi(w)\bigr)
 \asymp_\Psi K_\gamma^{\mathrm{per}}(y-w).
\]
Thus \(\norm{f^\Psi}_{\dot H_x^{\gamma/2}}\asymp_\Psi
\norm f_{\dot H_y^{\gamma/2}}\), which proves \eqref{eq:A-coercive-form}.
\end{proof}

\begin{proposition}[H\"older estimates for the transformed coefficients]
\label{prop:coefficient-output}
Assume the hypotheses of \Cref{prop:response,prop:zvonkin} and let
\(0<\rho\le\beta<1\).  Then the coefficients in
\eqref{eq:coefficients} satisfy
\begin{align}
  [a]_{C^\beta}
  &\lesssim_\Psi\norm{J-1}_{C^\beta},
  \label{eq:a-output-estimate}\\
  \norm{Q_0}_{C^\rho}
  &\lesssim_\Psi
  \norm{\ell}_{C^\rho}
  \bigl(1+\norm{J-1}_{C^\beta}\bigr),
  \\
  \norm R_{C^\rho}
  &\lesssim_\Psi\norm{\ell}_{L^\infty}\norm{\ell}_{C^\rho},
  \label{eq:R-output-estimate}\\
  \norm{Q_\lambda}_{C^\rho}
  &\lesssim_\Psi \norm{Q_0}_{C^\rho}
  +\lambda\norm{\psi}_{C^{1+\beta}}.
  \label{eq:Qlambda-output-estimate}
\end{align}
The implicit constants are locally uniform on bounded sets of maps with a
fixed positive Jacobian lower bound.
\end{proposition}

\begin{proof}
For \(0<\rho\le\beta<1\), multiplication by a \(C^\beta\) function is
bounded on \(C^\rho\), and composition with a \(C^{1+\beta}\)
bi-Lipschitz diffeomorphism preserves \(C^\rho\).  Apply these facts to
\[
  a=(J^2)^\Phi,
  \qquad Q_0=(\ell J)^\Phi,
  \qquad R=(\ell^2)^\Phi.
\]
This proves \eqref{eq:a-output-estimate}--\eqref{eq:R-output-estimate};
\eqref{eq:Qlambda-output-estimate} follows from the definition of
\(Q_\lambda\).
\end{proof}

Combining \Cref{prop:fractional-enhancement,prop:response,prop:zvonkin}
with \Cref{prop:coefficient-output} yields, on every finite interval,
\begin{equation}
\label{eq:verified-coefficient-bounds}
\begin{aligned}
 &0<J_*\le J\le J^*<\infty,
 &&\Psi,\Phi\ \text{uniformly bi-Lipschitz},\\
 &a\in L^\infty(0,T;C^\beta),
 &&Q_0,Q_\lambda\in
   L^{\gamma/(\gamma-1)}(0,T;C^\rho),\\
 &R\in L^2(0,T;H^{1-\gamma/2}).
\end{aligned}
\end{equation}
Here \(\rho,\beta>1-\gamma/2\).  The remaining Sobolev bounds used below
follow from \(C^\rho\hookrightarrow H^{1-\gamma/2}\) and finite-interval
time embeddings.

\section{A priori estimates}
\label{sec:deterministic-estimates}

\subsection{Entropy and \texorpdfstring{\(L^1\)}{L1} estimate}
\label{sec:global-L1}

We first derive an \(L^1\) bound without using the quadratic energy.  It will
serve as the low-order input for both the direct and mixed interpolation
arguments.  Fix once and for all an even, smooth, convex approximation
\(\chi\) of the absolute value such that
\begin{equation}\label{eq:chi-properties}
  \abs{\chi'}\le1,
  \qquad
  0\le\chi''\lesssim1,
  \qquad
  \supp\chi''\subset\{\abs z\le1\},
  \qquad
  0\le\chi(z)-\abs z\lesssim1.
\end{equation}
Set
\begin{equation}
  N(t):=\int_\T\chi(m(t,y))\dd y
  =\int_\T J(t,x)\chi\bigl(m^\Psi(t,x)\bigr)\dd x.
\end{equation}

\begin{proposition}[Entropy inequality and \(L^1\) bound]
\label{prop:global-L1}
There exists \(c_\Psi>0\), depending only on \(\gamma\) and the uniform
Jacobian and bi-Lipschitz bounds of \(\Psi\), such that every smooth solution
satisfies
\begin{equation}\label{eq:L1-differential-bound}
  \begin{aligned}
  N'(t)
  +c_\Psi\norm{\chi'(m(t))}_{\dot H^{\gamma/2}}^2
  \lesssim_\Psi{}&\norm{a(t)}_{\dot H^{1-\gamma/2}}^2
  +\norm{Q_0(t)}_{H^{1-\gamma/2}}^2\\
  &+\norm{R(t)}_{H^{1-\gamma/2}}^2
  +\lambda\norm{J(t)-1}_{L^1}.
  \end{aligned}
\end{equation}
Consequently, if the quantities on the right are integrable on \((0,T)\),
define
\begin{equation}
  \begin{aligned}
  N_T:={}&N(0)
  +\norm{a}_{L^2(0,T;\dot H^{1-\gamma/2})}^2
  +\norm{Q_0}_{L^2(0,T;H^{1-\gamma/2})}^2\\
  &+\norm{R}_{L^2(0,T;H^{1-\gamma/2})}^2
  +\lambda\norm{J-1}_{L^1(0,T;L^1)}.
  \end{aligned}
\end{equation}
Then
\begin{equation}\label{eq:L1-uniform-bound}
  \sup_{t\le T}\norm{m(t)}_{L^1}
  \le\sup_{t\le T}N(t)
  \lesssim_\Psi N_T<\infty.
\end{equation}
This bound depends on the initial data and the displayed coefficient norms,
but not on the \(L^2\)-energy of \(m\).
\end{proposition}

\begin{proof}
Testing \eqref{eq:transformed-equation} against \(\chi'(m)\) and changing
variables \(y=\Psi(t,x)\) gives
\begin{align*}
 N'(t)
 &+\int_\T\Bigl\{
 \chi'(m^\Psi)\Lambda^\gamma(Jm^\Psi)
 +[\chi(m^\Psi)-m^\Psi\chi'(m^\Psi)]\Lambda^\gamma J
 \Bigr\}\dd x\\
 &=-\int_\T(Jm^\Psi+\ell)^2
 \chi''(m^\Psi)\partial_x(m^\Psi)\dd x\\
 &\quad
 -\lambda\int_\T(J-1)
 [\chi(m^\Psi)-m^\Psi\chi'(m^\Psi)]\dd x.
\end{align*}
The calculation in
\Cref{lem:transformed-operator-entropy-identity} identifies the diffusion
integral on the left with
\(\ip{\chi'(m)}{\Acal_\Psi^\gamma m}\).  Moreover,
\[
 (am^2+2Q_0m+R)^\Psi=(Jm^\Psi+\ell)^2,
 \qquad
 \partial_x[\chi(m^\Psi)-m^\Psi\chi'(m^\Psi)]
 =-m^\Psi\chi''(m^\Psi)\partial_x(m^\Psi),
\]
and the second identity, together with \(\psi_x=J-1\), yields the
\(\lambda\)-term after periodic integration by parts.  For
\(1<\gamma<2\), \Cref{lem:transformed-operator-entropy-identity} gives
\[
 \ip{\chi'(m)}{\Acal_\Psi^\gamma m}
 =c_\gamma\iint_{\T^2}
 J(x)\mathfrak B_\chi\bigl(m^\Psi(x)\mid m^\Psi(z)\bigr)
 K_\gamma^{\mathrm{per}}(x-z)\dd x\dd z.
\]
Average this identity with its \(x,z\)-exchange.  Since \(J\ge J_*\)
and
\[
 \mathfrak B_\chi(u\mid v)+\mathfrak B_\chi(v\mid u)
 = [\chi'(u)-\chi'(v)](u-v),
\]
we obtain
\begin{align*}
 \ip{\chi'(m)}{\Acal_\Psi^\gamma m}
 &\ge\frac{c_\gamma J_*}{2}\iint_{\T^2}
 [\chi'(m^\Psi(x))-\chi'(m^\Psi(z))]\\
 &\hspace{39mm}{}\times[m^\Psi(x)-m^\Psi(z)]
 K_\gamma^{\mathrm{per}}(x-z)\dd x\dd z\\
 &\gtrsim_\Psi\norm{\chi'(m^\Psi)}_{\dot H^{\gamma/2}_x}^2.
\end{align*}
The last step is the elementary kernel form of the
Stroock--Varopoulos inequality; it uses that \(\chi'\) is nondecreasing
and uniformly Lipschitz.  See
\cite[Lemma~7.4]{BilerKarchMonneau2010Dislocation}.  At \(\gamma=2\), the
local part of \Cref{lem:transformed-operator-entropy-identity} and
\(0\le\chi''\lesssim1\) give
\[
\begin{aligned}
 \ip{\chi'(m)}{\Acal_\Psi^2 m}
 =\int_\T J\chi''(m^\Psi)\abs{\partial_x(m^\Psi)}^2\dd x
 \gtrsim_\Psi\norm{\partial_x\chi'(m^\Psi)}_{L^2}^2.
\end{aligned}
\]
Bi-Lipschitz composition makes the resulting homogeneous Sobolev seminorm
of \(\chi'(m^\Psi)\) comparable to
\(\norm{\chi'(m)}_{\dot H^{\gamma/2}_y}^2\) in both cases.  Returning the
nonlinear-flux term to \(y\)-coordinates, set
\[
 (\star):=-\int_\T(am^2+2Q_0m+R)\chi''(m)m_y\dd y.
\]
For \(k=0,1,2\), the support of \(\chi''\) gives
\[
 \left|\int_w^z r^k\chi''(r)\dd r\right|
 \le \left|\int_w^z\chi''(r)\dd r\right|
 =\abs{\chi'(z)-\chi'(w)}.
\]
Hence the difference characterization
\eqref{eq:besov-difference-characterization} of fractional Sobolev spaces
(and the ordinary chain rule when \(\gamma=2\)) gives
\[
 \left\|\int_0^f r^k\chi''(r)\dd r\right\|_{\dot H^{\gamma/2}}
 \lesssim\norm{\chi'(f)}_{\dot H^{\gamma/2}}.
\]
Each of the three flux summands is a coefficient multiplied by
\(\partial_y\int_0^m r^k\chi''(r)\dd r\), with \(k=2,1,0\), respectively.
The constant part of \(a\) drops out by periodicity.  Duality now
gives
\begin{align*}
 \abs{(\star)} \lesssim
 \bigl(\norm a_{\dot H^{1-\gamma/2}}
       +\norm{Q_0}_{H^{1-\gamma/2}}
       +\norm R_{H^{1-\gamma/2}}\bigr)
 \norm{\chi'(m)}_{\dot H^{\gamma/2}},
\end{align*}
and therefore
\begin{align*}
 \abs{(\star)}
 \le\frac{c_\Psi}{2}
 \norm{\chi'(m)}_{\dot H^{\gamma/2}}^2
 +C_\Psi\bigl(\norm a_{\dot H^{1-\gamma/2}}^2
 +\norm{Q_0}_{H^{1-\gamma/2}}^2
 +\norm R_{H^{1-\gamma/2}}^2\bigr).
\end{align*}
Finally,
\(\abs{\chi(z)-z\chi'(z)}\lesssim1\), so the resolvent term
is \(\lesssim\lambda\norm{J-1}_{L^1}\).  This proves
\eqref{eq:L1-differential-bound}; integration in time and
\eqref{eq:chi-properties} give \eqref{eq:L1-uniform-bound}.
\end{proof}

\subsection{Energy and Besov estimates}
\label{sec:energy}

The \(Q_\lambda\)- and \(R\)-terms are lower order, while the rough
coefficient \(a\) produces the cubic Besov functional \(Z_T^\beta\).  Testing
\eqref{eq:transformed-equation} against \(m\), using
\eqref{eq:A-coercive-form} for \(1<\gamma<2\) and
\eqref{eq:local-symmetric-form} at \(\gamma=2\), gives
\begin{equation}\label{eq:energy-identity}
 \frac12\frac{\dd}{\dd t}\norm m_{L^2}^2
 +\ip m{\Acal_\Psi^\gamma m}
 =-\int_\T(am^2+2Q_\lambda m+R)m_y\dd y,
\end{equation}
where
\[
 \ip m{\Acal_\Psi^\gamma m}
 \asymp_\Psi\norm m_{\dot H^{\gamma/2}}^2.
\]
At \(\gamma=2\), the term on the left is simply
\(\int_\T a\abs{m_y}^2\dd y\).  We shall also use the one-dimensional
Gagliardo--Nirenberg inequality: for \(p\in[1,\infty]\), with
\(1/\infty=0\),
\begin{equation}\label{eq:GN-general}
  \norm m_{L^p}
  \lesssim_p
  \norm m_{L^1}^{1-\frac{2(1-1/p)}{\gamma+1}}
  \norm m_{\dot H^{\gamma/2}}^{\frac{2(1-1/p)}{\gamma+1}}
  +\norm m_{L^1}.
\end{equation}

\begin{proposition}[Energy inequality with a cubic Besov term]
\label{prop:energy-cubic}
Let the transformed solution and coefficients be smooth, and let
\eqref{eq:verified-coefficient-bounds} hold with
\(1-\gamma/2<\beta<1\) and \(\rho>1-\gamma/2\).  Define the energy and
cubic Besov functionals by
\begin{align}
  E_T(m)
  &:=\sup_{t\le T}\norm{m(t)}_{L^2}^2
  +\int_0^T\norm{m(t)}_{H^{\gamma/2}}^2\dd t,
  \\
  Z_T^\beta(m)
  &:=\int_0^T\norm{m(t)^3}_{B^{1-\beta}_{1,1}}\dd t.
\end{align}
Once \(m\) is fixed, we suppress its argument; the same convention will be
used for \(K_T^s(m)\) below.
Then
\begin{equation}\label{eq:energy-with-Z}
  E_T
  \lesssim
    \norm{m(0)}_{L^2}^2+\norm R_{L^2(0,T;H^{1-\gamma/2})}^2
    +\left(\sup_{t\le T}[a(t)]_{C^\beta}\right)
       Z_T^\beta,
\end{equation}
where the implicit constant depends on the ellipticity and bi-Lipschitz
constants and on
\[
  \int_0^T\left[
    1+\norm{Q_\lambda(t)}_{C^\rho}^{\gamma/(\gamma-1)}
  \right]\dd t,
\]
but not on \(Z_T^\beta\).
\end{proposition}

\begin{proof}
For \(0<\beta<1\), the endpoint H\"older--Besov duality estimate
\begin{equation}\label{eq:endpoint-holder-besov}
 \abs{\ip{\partial_xa}{F}}
 \lesssim [a]_{C^\beta}\norm F_{B^{1-\beta}_{1,1}}
\end{equation}
is standard; see
\cite{BahouriCheminDanchin2011Fourier,
GubinelliImkellerPerkowski2015Paracontrolled,
GubinelliPerkowski2017KPZReloaded}.
Applying it after periodic integration by parts gives
\[
  -\int_\T am^2m_y\dd y
  =\frac13\ip{a_y}{m^3},
  \qquad
  \abs{\ip{a_y}{m^3}}
  \lesssim[a]_{C^\beta}\norm{m^3}_{B^{1-\beta}_{1,1}}.
\]

For the additive lower-order flux, Sobolev duality and Young's inequality
imply
\[
  \abs*{\int_\T Rm_y\dd y}
  \lesssim\norm R_{H^{1-\gamma/2}}
            \norm m_{\dot H^{\gamma/2}}
  \le\varepsilon\norm m_{\dot H^{\gamma/2}}^2
  +C_\varepsilon\norm R_{H^{1-\gamma/2}}^2.
\]

For the linear lower-order flux, multiplication by
\(Q_\lambda\in C^\rho\), \(\rho>1-\gamma/2\), is bounded on
\(H^{1-\gamma/2}\).  Interpolation between \(L^2\) and
\(H^{\gamma/2}\) gives
\[
 \norm m_{H^{1-\gamma/2}}
 \lesssim
 \norm m_{L^2}^{2(\gamma-1)/\gamma}
 \norm m_{H^{\gamma/2}}^{(2-\gamma)/\gamma}
 +\norm m_{L^2}.
\]
Sobolev duality therefore yields
\[
\begin{aligned}
 \abs*{\int_\T Q_\lambda m m_y\dd y}
 &\lesssim \norm m_{H^{\gamma/2}}
             \norm{Q_\lambda m}_{H^{1-\gamma/2}}\lesssim\norm{Q_\lambda}_{C^\rho}
 \norm m_{H^{\gamma/2}}
 \bigl(\norm m_{H^{1-\gamma/2}}+\norm m_{L^2}\bigr)\\
 &\lesssim\norm{Q_\lambda}_{C^\rho}
 \left(
   \norm m_{L^2}^{2(\gamma-1)/\gamma}
   \norm m_{H^{\gamma/2}}^{2/\gamma}
   +\norm m_{L^2}\norm m_{H^{\gamma/2}}
 \right)\\
 &\le\varepsilon\norm m_{\dot H^{\gamma/2}}^2
 +C_\varepsilon\left(1+
   \norm{Q_\lambda}_{C^\rho}^{\gamma/(\gamma-1)}\right)\norm m_{L^2}^2.
\end{aligned}
\]

Choosing \(\varepsilon\) small in the preceding two lower-order estimates, the
energy identity becomes
\[
\begin{aligned}
 \frac{\dd}{\dd t}\norm{m(t)}_{L^2}^2
 +c_\Psi\norm{m(t)}_{\dot H^{\gamma/2}}^2
 \lesssim_\Psi{}&\left[1+
 \norm{Q_\lambda(t)}_{C^\rho}^{\gamma/(\gamma-1)}\right]\norm{m(t)}_{L^2}^2\\
 &+\norm{R(t)}_{H^{1-\gamma/2}}^2
 +[a(t)]_{C^\beta}\norm{m(t)^3}_{B^{1-\beta}_{1,1}}.
\end{aligned}
\]
The integrating-factor argument and
\[
  \int_0^T\norm{m(t)}_{H^{\gamma/2}}^2\dd t
  \lesssim\int_0^T\norm{m(t)}_{\dot H^{\gamma/2}}^2\dd t
  +T\sup_{t\le T}\norm{m(t)}_{L^2}^2
\]
prove \eqref{eq:energy-with-Z}.
\end{proof}

\subsection{Cubic increment estimate and closure}
\label{sec:khm}

The preceding energy quantities already give, by interpolation and the
one-dimensional Besov embedding,
\begin{equation}\label{eq:energy-besov}
 \norm m_{L^3(0,T;B^s_{3,3})}^3
 \lesssim
 \norm m_{L^\infty(0,T;L^2)}
 \norm m_{L^2(0,T;H^{\gamma/2})}^2
 \lesssim E_T^{3/2},
 \qquad 0<s<\frac{2\gamma-1}{6}.
\end{equation}
In particular, this yields \(s<1/2\) when \(\gamma=2\).  The quadratic
transport flux provides a further coercive mechanism in the increment
variable through the K\'arm\'an--Howarth--Monin identity.  We now use it to
estimate the cubic increments directly.  In contrast with the
\(E_T^{3/2}\)-dependence above, the resulting estimate is linear in \(E_T\),
\(Z_T^\beta\), and the squared norm of the lower-order forcing.

For \(0<s<1/3\), define the cubic increment functional
\begin{equation}\label{eq:KHM-functional}
 K_T^s(m):=
 \int_0^T\int_0^1 h^{-1-3s}
 \norm{\delta_hm(t)}_{L^3}^3\dd h\dd t.
\end{equation}

The periodic difference characterization
\eqref{eq:besov-difference-characterization} and the corresponding
\(L^1\)-Poincar\'e estimate give
\begin{equation}\label{eq:Besov-from-KHM}
 \norm m_{L^3(0,T;B^s_{3,3})}^3
 \lesssim_s
 K_T^s(m)
 +T\sup_{t\le T}\norm{m(t)}_{L^1}^3,
 \qquad 0<s<\frac13.
\end{equation}

We first estimate \(K_T^s\) directly from the transformed equation.  Set
\(G:=2Q_\lambda m+R\); then
\begin{equation}\label{eq:A-equation}
  m_t+\Acal_{\Psi}^{\gamma}m
  =\partial_y(am^2+G).
\end{equation}
We keep \(\delta_h(\Acal_{\Psi}^{\gamma}m)\) intact and apply
\eqref{eq:A-form-bound}, or \eqref{eq:local-form-bound} at \(\gamma=2\).

\begin{proposition}[Cubic increment estimate]
\label{prop:fractional-khm-base}
Let \(0<s<1/3\), \(0<\beta<1\), and suppose
\(a\in L^\infty(0,T;C^\beta)\) with \(a\ge a_*>0\) and
\(G\in L^2(0,T;H^{1-\gamma/2})\).  Assume that
\(\Acal_\Psi^\gamma\) satisfies \eqref{eq:A-form-bound} when
\(1<\gamma<2\), and \eqref{eq:local-form-bound} when \(\gamma=2\).  Every
smooth solution of
\eqref{eq:A-equation} satisfies
\begin{equation}\label{eq:khm-base}
\begin{aligned}
  K_T^s(m)
  \lesssim_{a_*,\Psi}\biggl[{}&
    \norm{m(0)}_{L^2}^2+\norm{m(T)}_{L^2}^2
    +\int_0^T\norm{m(t)}_{H^{\gamma/2}}^2\dd t\\
    &{}+\norm a_{L^\infty(0,T;C^\beta)}Z_T^\beta
    +\norm G_{L^2(0,T;H^{1-\gamma/2})}^2
  \biggr].
\end{aligned}
\end{equation}
\end{proposition}

\begin{proof}
We first perform the calculation for smooth coefficients.  The resulting
bounds depend only on the norms in \eqref{eq:khm-base}, so the stated
estimate for \(a\in C^\beta\) and \(G\in H^{1-\gamma/2}\) follows by
approximation.  Taking an increment of \eqref{eq:A-equation} gives
\[
  \partial_t\delta_hm+\delta_h(\Acal_{\Psi}^{\gamma}m)
  =\partial_y\delta_h(am^2)+\partial_y\delta_hG.
\]
Pair this equation with a smooth approximation of
\(\abs{\delta_hm}\) and integrate in \(y\).
The variable flux satisfies the identity
\begin{equation}
\label{eq:variable-coefficient-khm-identity}
\begin{aligned}
 &\int_\T\abs{\delta_rm}\,\partial_y\delta_r(am^2)\dd y\\
 &\quad=\partial_r\int_\T\left[
   \frac{a+a_r}{6}\abs{\delta_rm}^3
   +\frac{\delta_ra}{2}\abs{\delta_rm}\,\delta_r(m^2)
 \right]\dd y
 +\frac13\ip{\partial_y\delta_ra}{\abs{\delta_r(m^3)}}.
\end{aligned}
\end{equation}
To prove it, use the symmetric decomposition of the variable-flux increment:
\[
 \delta_r(am^2)
 =\frac{a+a_r}{2}\delta_r(m^2)
  +\frac{\delta_ra}{2}(m_r^2+m^2).
\]
On the other hand, we can rewrite the density inside
the \(r\)-derivative in \eqref{eq:variable-coefficient-khm-identity} as
\[
\frac{a+a_r}{6}\abs{\delta_rm}^3
 +\frac{\delta_ra}{2}\abs{\delta_rm}\,\delta_r(m^2)
 =\frac13\delta_rm\abs{\delta_rm}
 \bigl[a_r(m+2m_r)-a(2m+m_r)\bigr].
\]
Since \(\partial_rm_r=\partial_ym_r\) and
\(\partial_ra_r=\partial_ya_r\), the two factors on the right-hand side of
the preceding identity satisfy
\[
\begin{aligned}
 \partial_r\bigl(\delta_rm\abs{\delta_rm}\bigr)
 &=2\abs{\delta_rm}\partial_ym_r,\\
 \partial_r\bigl[a_r(m+2m_r)-a(2m+m_r)\bigr]
 &=(\partial_ya_r)(m+2m_r)+(2a_r-a)\partial_ym_r.
\end{aligned}
\]
After the Leibniz rule, the two contributions containing
\(\abs{\delta_rm}\partial_ym_r\) can be merged as
\[
\begin{aligned}
 &\frac13\abs{\delta_rm}\partial_ym_r
 \Bigl\{
  2\bigl[a_r(m+2m_r)-a(2m+m_r)\bigr]
  +\delta_rm(2a_r-a)
 \Bigr\}\\
 &\qquad=
 \abs{\delta_rm}\bigl[2a_rm_r-a(m+m_r)\bigr]\partial_ym_r.
\end{aligned}
\]
Keeping the remaining \(\partial_ya_r\)-term gives
\[
\begin{aligned}
 &\partial_r\left[
   \frac{a+a_r}{6}\abs{\delta_rm}^3
   +\frac{\delta_ra}{2}\abs{\delta_rm}\,\delta_r(m^2)
 \right]\\
 &\quad=\abs{\delta_rm}\left[
   \bigl(2a_rm_r-a(m+m_r)\bigr)\partial_ym_r
   +\frac{\delta_rm}{3}(\partial_ya_r)(m+2m_r)
 \right]\\
 &\quad=\abs{\delta_rm}\,\partial_y\delta_r(am^2)
 -\frac13(\partial_y\delta_ra)\abs{\delta_rm}
   (m^2+mm_r+m_r^2)\\
 &\qquad
 -\frac13\partial_y\left[
   a\,\delta_rm\abs{\delta_rm}(2m+m_r)
 \right].
\end{aligned}
\]
The final term is a spatial divergence and integrates to zero by periodicity.
Since
\(\abs{\delta_rm}(m^2+mm_r+m_r^2)=\abs{\delta_r(m^3)}\),
integration of the preceding identity proves
\eqref{eq:variable-coefficient-khm-identity}.
Combining the tested equation with
\eqref{eq:variable-coefficient-khm-identity}, integrating over
\((t,r)\in(0,T)\times(0,h)\), and using \(\delta_0m=0\) yields
\begin{equation}
\begin{aligned}
 \frac16\int_0^T\int_\T
 (a+a_h)\abs{\delta_hm}^3\dd y\dd t=\mathrm{(I)}+\mathrm{(II)}+\mathrm{(III)}+\mathrm{(IV)},
\end{aligned}
\end{equation}
where
\begin{align*}
 \mathrm{(I)}
 &:={}
 \frac12\left[
   \int_0^h\int_\T
   \delta_r m\abs{\delta_r m}\dd y\dd r
 \right]_{t=0}^{t=T}, \quad
 \mathrm{(II)}
 :={}
 \int_0^T\int_0^h
 \ip{\abs{\delta_r m}}
    {\delta_r(\Acal_{\Psi}^{\gamma}m)}
 \dd r\dd t,\\
 \mathrm{(III)}
 &:={}
 -\int_0^T\int_0^h
 \ip{\abs{\delta_r m}}{\partial_y\delta_r G}
 \dd r\dd t,\\
 \mathrm{(IV)}
 &:={}
 -\frac12\int_0^T\int_\T
 (\delta_ha)\abs{\delta_hm}\,\delta_h(m^2)\dd y\dd t\\
 &\quad
 -\frac13\int_0^T\int_0^h
 \ip{\partial_y\delta_ra}{\abs{\delta_r(m^3)}}\dd r\dd t.
\end{align*}
Multiply by \(h^{-1-3s}\) and integrate over \(h\in(0,1)\).
Since \((a+a_h)/6\ge a_*/3\), the left-hand side controls
\(K_T^s(m)\).  For \(\mathrm{(I)}\),
\[
 \abs{\mathrm{(I)}}
 \lesssim
 h\sum_{t\in\{0,T\}}\norm{m(t)}_{L^2}^2.
\]
Consequently, since \(3s<1\),
\[
 \int_0^1h^{-1-3s}\abs{\mathrm{(I)}}\dd h
 \lesssim
 \sum_{t\in\{0,T\}}\norm{m(t)}_{L^2}^2
 \int_0^1h^{-3s}\dd h.
\]
For \(\mathrm{(II)}\), translations are isometries on \(H^r(\T)\) for
every \(r\in\R\).  Since \(\gamma/2\le1\), Lipschitz composition by the
absolute value, with the weak chain rule at \(\gamma=2\), gives
\[
 \begin{aligned}
 \norm{\abs{\delta_r m}}_{H^{\gamma/2}}
 &\lesssim\norm{\delta_r m}_{H^{\gamma/2}}
 \lesssim\norm m_{H^{\gamma/2}},\\
 \norm{\delta_r(\Acal_\Psi^\gamma m)}_{H^{-\gamma/2}}
 &\lesssim\norm{\Acal_\Psi^\gamma m}_{H^{-\gamma/2}}
 \lesssim_\Psi\norm m_{H^{\gamma/2}}.
 \end{aligned}
\]
Sobolev duality therefore gives
\[
 \abs{\ip{\abs{\delta_r m}}
 {\delta_r(\Acal_\Psi^\gamma m)}}
 \lesssim_\Psi\norm m_{H^{\gamma/2}}^2.
\]
This argument keeps the translated output
\(\delta_r(\Acal_\Psi^\gamma m)\) intact and introduces no commutator.
It follows directly that
\[
 \abs{\mathrm{(II)}}
 \lesssim_\Psi
 h\int_0^T\norm m_{H^{\gamma/2}}^2\dd t.
\]
Hence
\[
 \int_0^1h^{-1-3s}\abs{\mathrm{(II)}}\dd h
 \lesssim_\Psi
 \int_0^T\norm m_{H^{\gamma/2}}^2\dd t
 \int_0^1h^{-3s}\dd h.
\]

For \(\mathrm{(III)}\), Sobolev duality and the same composition bound lead to
\[
 \begin{aligned}
 \abs{\ip{\abs{\delta_r m}}{\partial_y\delta_r G}}
 \lesssim
 \norm{\delta_r m}_{H^{\gamma/2}}
 \norm{\delta_r G}_{H^{1-\gamma/2}}\lesssim
 \norm m_{H^{\gamma/2}}^2+\norm G_{H^{1-\gamma/2}}^2,
 \end{aligned}
\]
and therefore
\[
 \abs{\mathrm{(III)}}
 \lesssim h\left(
 \int_0^T\norm m_{H^{\gamma/2}}^2\dd t
 +\norm G_{L^2_tH^{1-\gamma/2}}^2
 \right).
\]
Thus
\[
 \int_0^1h^{-1-3s}\abs{\mathrm{(III)}}\dd h
 \lesssim
 \left(
 \int_0^T\norm m_{H^{\gamma/2}}^2\dd t
 +\norm G_{L^2_tH^{1-\gamma/2}}^2
 \right)
 \int_0^1h^{-3s}\dd h.
\]

For the first term in \(\mathrm{(IV)}\), the factorization of the
cubic increment and an elementary quadratic bound give
\[
 \abs{\delta_hm}\,\abs{\delta_h(m^2)}
 \lesssim\abs{\delta_h(m^3)}.
\]
Consequently, using
\(\abs{\delta_ha}\le[a]_{C^\beta}h^\beta\) and
\(h^{\beta-1-3s}\le h^{\beta-2}\),
its scale integral is bounded by
\[
 \begin{aligned}
 \norm a_{L^\infty_tC^\beta}
 \int_0^T\int_0^1
 h^{\beta-1-3s}\norm{\delta_h(m^3)}_{L^1}\dd h\dd t
 \lesssim
 \norm a_{L^\infty_tC^\beta}Z_T^\beta.
 \end{aligned}
\]
For the second term, since \(0<1-\beta<1\), boundedness under translations and
Lipschitz composition by the absolute value give
\[
 \begin{aligned}
 \norm{\partial_y\delta_r a}_{C^{\beta-1}}
 &\lesssim[a]_{C^\beta},\\
 \norm{\abs{\delta_r(m^3)}}_{B^{1-\beta}_{1,1}}
 &\lesssim\norm{\delta_r(m^3)}_{B^{1-\beta}_{1,1}}
 \lesssim\norm{m^3}_{B^{1-\beta}_{1,1}}.
 \end{aligned}
\]
Applying \eqref{eq:endpoint-holder-besov} to \(\delta_r a\) therefore gives
at each time, uniformly in \(r\),
\[
 \abs{\ip{\partial_y\delta_r a}{\abs{\delta_r(m^3)}}}
 \lesssim
 [a]_{C^\beta}\norm{m^3}_{B^{1-\beta}_{1,1}}.
\]
Its \(r\)-integral is thus bounded by \(h\) times the right-hand side.
After integration in \(h\) and \(t\), the two parts give
\[
 \int_0^1h^{-1-3s}\abs{\mathrm{(IV)}}\dd h
 \lesssim
 \norm a_{L^\infty_tC^\beta}Z_T^\beta.
\]
Combining the estimates of \(\mathrm{(I)}\)--\(\mathrm{(IV)}\) proves
\eqref{eq:khm-base}.
\end{proof}

With the Burgers normalization
\(m\partial_ym=\frac12\partial_y(m^2)\), the transport calculation in the
proof of
\cite[Lemma~4.1, equations~(34)--(35)]{GoldmanJosienOtto2015BurgersKS}
yields
\[
\begin{aligned}
 \frac12\abs{\delta_rm}\,\partial_y\delta_r(m^2)
 ={}&\frac16\partial_r\abs{\delta_rm}^3\\
 &+\frac12\partial_y\left(
   m\abs{\delta_rm}\delta_rm+\frac13\abs{\delta_rm}^3
 \right).
\end{aligned}
\]
Periodic integration in \(y\) removes the spatial-divergence term and leaves
\[
 \frac12\int_\T\abs{\delta_rm}\,\partial_y\delta_r(m^2)\dd y
 =\frac16\partial_r\int_\T\abs{\delta_rm}^3\dd y.
\]
This is exactly \eqref{eq:variable-coefficient-khm-identity} when
\(a\equiv\frac12\): both the \(\delta_ra\)-correction in the scale flux and
the final \(\partial_y\delta_ra\)-pairing vanish.  The Zvonkin change of
coordinates makes \(a\) spatially variable, so these two terms must instead be
retained and are precisely the two contributions estimated in
\(\mathrm{(IV)}\).

\medskip
Returning to the closure, it now suffices to control \(Z_T^\beta\).  The
following deterministic interpolation lemma treats the two parameter
regimes.

\begin{lemma}[Interpolation bounds for the cubic Besov functional]
\label{lem:cubic-besov-interpolation}
Let \(1<\gamma\le2\), \(1-\gamma/2<\beta<1\), and let \(m\) be smooth
with \(\sup_{t\le T}\norm{m(t)}_{L^1}\le M\).
Let \(E_T\) and \(Z_T^\beta\) be as in
\Cref{prop:energy-cubic}.  For every \(\varepsilon>0\), the following
bounds hold with a finite constant \(C_\varepsilon\), whose value may
change from line to line and may depend on \(T,M,\gamma,\beta\), but not on
\(m\).
\begin{enumerate}[label=(\roman*),leftmargin=*]
\item If \(\beta>2-\gamma\), then
 \begin{equation}\label{eq:direct-energy-absorption}
 Z_T^\beta
 \le\varepsilon E_T+C_\varepsilon.
\end{equation}

\item If \(1<\gamma<2\) and
\(1-\gamma/2<\beta\le2-\gamma\), then there exists
\(s_0\in(1/12,1/3)\), depending only on
\(\gamma\) and \(\beta\), such that
 \begin{equation}\label{eq:mixed-cubic-young}
 Z_T^\beta
 \le\varepsilon\bigl(E_T+K_T^{s_0}\bigr)
 +C_\varepsilon.
 \end{equation}
\end{enumerate}
\end{lemma}

\begin{proof}
For part~(i), the difference characterization
\eqref{eq:besov-difference-characterization}, H\"older's inequality, and the
factorization
\(\delta_h(m^3)=\delta_hm\,(m_h^2+m_hm+m^2)\) give
\[
 \norm{m^3}_{B^{1-\beta}_{1,1}}
 \lesssim \norm m_{L^4}^2\norm m_{B^{1-\beta}_{2,1}}.
\]
Since \(1-\beta<\gamma/2\), the Besov--Gagliardo--Nirenberg interpolation
\cite[Sections~2.4--2.5]{BahouriCheminDanchin2011Fourier} gives
\[
 \norm m_{B^{1-\beta}_{2,1}}
 \lesssim
 \norm m_{L^1}^{1-\frac{3-2\beta}{\gamma+1}}
 \norm m_{H^{\gamma/2}}^{\frac{3-2\beta}{\gamma+1}}
 +\norm m_{L^1}.
\]
Combining these estimates with \eqref{eq:GN-general} at \(p=4\) and the
assumed \(L^1\) bound yields
\[
 \norm{m(t)^3}_{B^{1-\beta}_{1,1}}
 \lesssim_M
 1+\norm{m(t)}_{H^{\gamma/2}}^{\frac{6-2\beta}{\gamma+1}}.
\]
The exponent is strictly below \(2\) precisely when
\(\beta>2-\gamma\).  Integration in time and Young's inequality prove
\eqref{eq:direct-energy-absorption}. For part~(ii), choose \(s_0\in(1/12,1/3)\) sufficiently close to \(1/3\),
and then choose \(\theta\) such that
\[
 \frac{1-3s_0}{1+s_0}<\theta<
 \min\left\{
 1,\frac{\beta-(1-\gamma/2)}{\gamma/2-s_0}
 \right\}.
\]
Such a choice is possible because the left endpoint tends to zero as
\(s_0\uparrow1/3\), whereas the nontrivial upper bound remains positive
because \(\beta>1-\gamma/2\).  Fractional Leibniz and strict dyadic
interpolation \cite[Sections~2.4, 2.5 and~2.8]
{BahouriCheminDanchin2011Fourier} give
\[
 \norm{m^3}_{B^{1-\beta}_{1,1}}
 \lesssim
 \norm m_{L^{12/(3+\theta)}}^2
 \norm m_{B^{1-\beta}_{6/(3-\theta),1}},
\]
where we further have the controls
\[
\begin{aligned}
 \norm m_{B^{1-\beta}_{6/(3-\theta),1}}
 &\lesssim
 \norm m_{H^{\gamma/2}}^{1-\theta}
 \norm m_{B^{s_0}_{3,3}}^\theta,\\
 \norm m_{L^{12/(3+\theta)}}
 &\lesssim
 \norm m_{L^1}^{1-\frac{9-\theta}{4(3s_0+2)}}
 \norm m_{B^{s_0}_{3,3}}^{\frac{9-\theta}{4(3s_0+2)}}
 +\norm m_{L^1}.
\end{aligned}
\]
The interpolation exponent in the second line belongs to \((0,1)\):
its upper bound is equivalent to \(\theta>1-12s_0\), which follows
from \(s_0>1/12\) and \(\theta>0\).
For the middle estimate, the strict summability condition is
\[
 (1-\theta)\frac\gamma2+\theta s_0>1-\beta,
\]
which is exactly the stated upper bound on \(\theta\).  The preceding
estimates and the \(L^1\) bound imply
\[
\begin{aligned}
 \norm{m^3}_{B^{1-\beta}_{1,1}}
 \lesssim_M{}&1+
 \norm m_{H^{\gamma/2}}^{1-\theta}
 \left(
 1+\norm m_{B^{s_0}_{3,3}}^{
  \theta+\frac{9-\theta}{2(3s_0+2)}}
 \right).
\end{aligned}
\]
The only nonautomatic condition for time integrability is
\[
 \frac{1-\theta}{2}
 +\frac13\left(
  \theta+\frac{9-\theta}{2(3s_0+2)}
 \right)<1,
\]
and this is equivalent to the lower bound on \(\theta\).  H\"older's and
Young's inequalities therefore give
\[
 Z_T^\beta
 \le\varepsilon\left(
 E_T+\norm m_{L^3(0,T;B^{s_0}_{3,3})}^3
 \right)+C_\varepsilon.
\]
Using \eqref{eq:Besov-from-KHM} and relabeling \(\varepsilon\) after
absorbing its \(s_0\)-dependent constant proves
\eqref{eq:mixed-cubic-young}.
\end{proof}

\begin{theorem}[A priori bounds for the transformed equation]
\label{thm:khm}
Let \(m\) be a smooth transformed solution.  When \(1<\gamma<2\), assume
that \(\Acal_\Psi^\gamma\) satisfies
\Cref{prop:transformed-operator-paralinearization,%
prop:transformed-operator-coercivity}; when \(\gamma=2\), let it be the
Laplacian divergence-form operator in \eqref{eq:local-check}.  Assume also
that \eqref{eq:verified-coefficient-bounds} holds with
\[
 \rho>1-\frac\gamma2,
 \qquad
 1-\frac\gamma2<\beta<1.
\]
Then
\begin{equation}\label{eq:main-L1}
  \sup_{t\le T}\norm{m(t)}_{L^1}
  \lesssim_\Psi N_T<\infty.
\end{equation}
Moreover, there is a finite constant \(C_T\), depending only on \(T\), the
initial norms, and the coefficient, ellipticity, bi-Lipschitz, and operator
bounds in these assumptions, such that
\begin{equation}\label{eq:main-energy}
  \sup_{t\le T}\norm{m(t)}_{L^2}^2
  +\int_0^T\norm{m(t)}_{H^{\gamma/2}}^2\dd t
  +Z_T^\beta(m)
  \le C_T.
\end{equation}
For every \(0<s<1/3\), there is a finite constant \(C_{T,s}\), with the
same dependence as \(C_T\) and additionally on \(s\), such that
\begin{equation}\label{eq:khm-conclusion}
  K_T^s(m)\le C_{T,s}<\infty.
\end{equation}
\end{theorem}

\begin{proof}
The global \(L^1\) estimate \Cref{prop:global-L1} gives
\eqref{eq:main-L1} independently of \(E_T\).  With the endpoint convention
\(2\gamma/(2-\gamma)=\infty\) at \(\gamma=2\), interpolation between
\(L^\infty_tL^2\) and \(L^2_tH^{\gamma/2}\) gives, for
\(1<\gamma\le2\),
\[
\norm m_{L^{2\gamma/(2-\gamma)}_tH^{1-\gamma/2}}^2
 \lesssim E_T.
\]
Since multiplication by \(C^\rho\), \(\rho>1-\gamma/2\), is bounded on
\(H^{1-\gamma/2}\), H\"older's inequality yields
\[
 \norm{Q_\lambda m}_{L^2_tH^{1-\gamma/2}}
 \lesssim
 \norm{Q_\lambda}_{L^{\gamma/(\gamma-1)}_tC^\rho}
 \norm m_{L^{2\gamma/(2-\gamma)}_tH^{1-\gamma/2}} \lesssim E_T^{1/2}.
\]
Consequently,
\[
 \norm G_{L^2_tH^{1-\gamma/2}}^2
 \lesssim 1+E_T,
\]
where the prescribed coefficient and \(R\)-norms are absorbed into the
implicit constant.  \Cref{prop:energy-cubic} gives
\[
 E_T\lesssim1+Z_T^\beta.
\]
Apply \Cref{lem:cubic-besov-interpolation} with
\(M:=\sup_{t\le T}\norm{m(t)}_{L^1}\).  By \eqref{eq:main-L1},
\(M\lesssim_\Psi N_T\).

Suppose first that \(\beta>2-\gamma\).  Part~(i) of the lemma yields
\[
 Z_T^\beta\le\varepsilon E_T+C_\varepsilon.
\]
Choosing \(\varepsilon\) after fixing the coefficient norms closes
\(E_T+Z_T^\beta\).

In the remaining case \(1<\gamma<2\) and
\(1-\gamma/2<\beta\le2-\gamma\), let \(s_0\) be supplied by part~(ii) of
\Cref{lem:cubic-besov-interpolation}.  Applying
\Cref{prop:fractional-khm-base} at \(s_0\) and using the forcing estimate
gives
\[
\begin{aligned}
 E_T+K_T^{s_0}&\lesssim1+Z_T^\beta,\\
 Z_T^\beta&\le
 \varepsilon\bigl(E_T+K_T^{s_0}\bigr)
 +C_\varepsilon.
\end{aligned}
\]
Choosing \(\varepsilon\) after the implicit constant closes
\(E_T+K_T^{s_0}+Z_T^\beta\).  Thus \eqref{eq:main-energy} holds in both
regimes.  Finally, for any \(0<s<1/3\), the forcing estimate and
\Cref{prop:fractional-khm-base} give \eqref{eq:khm-conclusion} from the
now established bounds on \(E_T\) and \(Z_T^\beta\).
\end{proof}

\subsection{Canonical approximation and global consequences}
\label{sec:canonical-approximation-global-consequences}

The main estimate also yields compactness of the canonical approximations.

\begin{proposition}[Canonical approximation and compactness]
\label{prop:canonical-compactness}
Let \((\alpha,\gamma)\) satisfy \eqref{eq:admissible-alpha-gamma}, fix
\(T>0\), and let \(\eps_n\downarrow0\).  Let
\(\mathbb X^n_{\alpha,\gamma}\) be the canonical mollified enhancements of
\Cref{prop:fractional-enhancement}, and let
\(\mathbb X^n_{Q,\lambda}\) be the associated resolvent families from
\eqref{eq:resolvent-model-uniform}.  Suppose that \(u_0^n\in C^\infty(\T)\)
converges strongly to \(u_0\) in \(L^2(\T)\). After passage to a deterministic subsequence, not relabeled, there is an
event \(\Omega_T\) of probability one for which the following hold.
\begin{enumerate}[label=(\roman*),leftmargin=*,nosep]
\item\label{item:canonical-inputs-uniform-estimates}
On \(\Omega_T\), the model convergence of
\Cref{prop:fractional-enhancement} and the uniform resolvent-family
convergence in \eqref{eq:resolvent-model-uniform} hold along the chosen
subsequence.  For \(\omega\in\Omega_T\), use the local uniformity in
\Cref{prop:zvonkin} to choose one resolvent parameter
\(\lambda=\lambda(T,\omega)\) for the limiting and all approximate Zvonkin
maps. For each \(n\), let \(u^n\) be the smooth solution of
\eqref{eq:formal-sbe} driven by \(\xi^{\eps_n}\), with initial condition
\(u_0^n\).  Let \(\ell^n\) be its response, and let
\((\Psi^n,J^n)\) be the Zvonkin pair constructed with the common value of
\(\lambda\).  Define \(m^n\) by
\[
 u^n
 =X^n+X^{\zzone,n}+\ell^n+J^n(m^n)^{\Psi^n}.
\]
Let \((\ell,\Psi,J)\) be the corresponding limiting objects and define
\(a,Q_\lambda,R\) by \eqref{eq:coefficients}.  Then \(m^n(0)=u_0^n\), and the
bounds of \Cref{thm:main-result} hold uniformly in \(n\) for this fixed
realization.

\item\label{item:pathwise-compactness-identification}
For each \(\omega\in\Omega_T\), there is a further subsequence, possibly
depending on \(\omega\), and a limit \(m\) such that
\begin{align}
 m^n&\longrightarrow m
 &&\text{strongly in }L^2(0,T;L^2),\label{eq:m-strong-limit}\\
 m^n&\rightharpoonup m
 &&\text{weakly in }L^2(0,T;H^{\gamma/2}),\notag\\
 m^n&\rightharpoonup^\ast m
 &&\text{in }L^\infty(0,T;L^2).\notag
\end{align}
The limit has a representative in \(C_{\mathrm w}([0,T];L^2)\) with
\(m(0)=u_0\), solves \eqref{eq:transformed-equation} in distributions on
\((0,T)\times\T\), with the flux \eqref{eq:Dflux} when \(1<\gamma<2\) and
\eqref{eq:local-check} when \(\gamma=2\), and inherits the bounds of
\Cref{thm:main-result}.  When \(\gamma=2\), this includes the Besov range
\(0<s<1/2\).
\end{enumerate}
\end{proposition}

\begin{proof}
For \ref{item:canonical-inputs-uniform-estimates}, choose
\((\rho,\beta)\) as in
\eqref{eq:coefficient-exponent-choice}.  By
\Cref{prop:fractional-enhancement} and the uniform resolvent convergence
\eqref{eq:resolvent-model-uniform}, quantified in
\eqref{eq:uniform-resolvent-second-chaos}, a standard summable-subsequence
argument yields a deterministic subsequence and an event \(\Omega_T\) of
probability one on which both enhanced inputs converge; their pathwise norms
are therefore uniformly bounded.  This is the fractional, higher-order
counterpart of the canonical Burgers-model convergence in
\cite[Definition~3.2 and Theorem~9.1]
{GubinelliPerkowski2017KPZReloaded}.

Fix \(\omega\in\Omega_T\).  The local uniformity in \Cref{prop:zvonkin}
permits one choice \(\lambda=\lambda(T,\omega)\) for the limiting and all
approximate Zvonkin maps.  Choose
\(1-\gamma/2<\rho'<\rho<\beta'<\beta\).  Local Lipschitz continuity of the
response and Zvonkin maps gives uniform Jacobian bounds and
\[
 \Psi^n\longrightarrow\Psi\text{ in }C_tC^{1+\beta'},
 \qquad
 \ell^n\longrightarrow\ell\text{ in }C_tC^{\rho'}.
\]
Multiplication and composition then give
\[
 a^n\longrightarrow a\text{ in }C_tC^{\beta'},
 \qquad
 (Q_0^n,Q_\lambda^n,R^n)\longrightarrow(Q_0,Q_\lambda,R)
 \text{ in }C_t(C^{\rho'})^3.
\]
Since \(\rho'>1-\gamma/2\), in particular
\(R^n\to R\) in \(L^2_tH^{1-\gamma/2}\).  For \(1<\gamma<2\), the
operator-difference estimate \eqref{eq:transformed-operator-difference},
applied with the same small H\"older loss, yields convergence of
\(\Acal_{\Psi^n}^{\gamma}\) to \(\Acal_{\Psi}^{\gamma}\) in the operator
norm from \(H^{\gamma/2}\) to \(H^{-\gamma/2}\).  At \(\gamma=2\),
\eqref{eq:local-check} gives, in one step,
\[
\begin{aligned}
 (\Acal_{\Psi^n}^{2}-\Acal_{\Psi}^{2})f
 &=-\partial_y\bigl((a^n-a)f_y\bigr),&
 \norm{(\Acal_{\Psi^n}^{2}-\Acal_{\Psi}^{2})f}_{H^{-1}}
 &\lesssim\norm{a^n-a}_{L^\infty}\norm f_{H^1}.
\end{aligned}
\]
Thus, in both cases,
\[
 \sup_{t\le T}
 \norm{\Acal_{\Psi^n(t)}^\gamma-\Acal_{\Psi(t)}^\gamma}
 _{\mathcal L(H^{\gamma/2},H^{-\gamma/2})}
 \longrightarrow0.
\]

For \ref{item:pathwise-compactness-identification}, the transformed
equation and form bound give \(\partial_tm^n\) bounded in
\(L^2(0,T;H^{-2})\).
The diffusion is bounded in \(L^2_tH^{-\gamma/2}\), while the uniform
\(L^\infty_tL^2\) bound gives
\(\partial_y(a^n(m^n)^2)\) in \(L^\infty_tH^{-2}\), and
\(\partial_y(Q_\lambda^nm^n)\) and \(\partial_yR^n\) are bounded in
\(L^2_tH^{-1}\) and \(L^2_tH^{-\gamma/2}\), respectively, by
\eqref{eq:verified-coefficient-bounds}.  Aubin--Lions--Simon, weak compactness,
and one extraction give \eqref{eq:m-strong-limit} and almost-everywhere
convergence.  Arzel\`a--Ascoli also gives
\(m^n\to m\) in \(C([0,T];H^{-2})\).  Together with the uniform
\(L^\infty_tL^2\) bound, this gives
\(m\in C_{\mathrm w}([0,T];L^2)\): first test against \(H^2\) functions
and then use their density in \(L^2\).  Since
\(m^n(0)=u_0^n\to u_0\) in \(L^2\), this gives the initial trace
\(m(0)=u_0\).

To identify the limit in \eqref{eq:m-strong-limit}, split the moving-operator
term as
\[
 \Acal_{\Psi^n}^{\gamma}m^n-\Acal_{\Psi}^{\gamma}m
 =(\Acal_{\Psi^n}^{\gamma}-\Acal_{\Psi}^{\gamma})m^n
  +\Acal_{\Psi}^{\gamma}(m^n-m).
\]
The first term converges strongly to zero in
\(L^2_tH^{-\gamma/2}\) by operator-norm convergence and the uniform energy
bound; the second converges weakly to zero there.  Moreover, strong
\(L^2_{t,y}\) convergence and strong convergence of the coefficients give
\[
 a^n(m^n)^2\longrightarrow am^2\quad\text{in }L^1_{t,y},
 \qquad
 Q_\lambda^nm^n\longrightarrow Q_\lambda m\quad\text{in }L^2_{t,y}.
\]
Together with \(R^n\to R\), these limits identify
\eqref{eq:transformed-equation}.  The uniform Jacobian bounds give
\[
 \norm{(m^n)^{\Psi^n}-m^{\Psi^n}}_{L^2_{t,x}}
 =\norm{(m^n-m)^{\Psi^n}}_{L^2_{t,x}}
 \lesssim\norm{m^n-m}_{L^2_{t,y}}\longrightarrow0.
\]
For \(g\in C^\infty([0,T]\times\T)\),
\[
\begin{aligned}
 \norm{m^{\Psi^n}-m^\Psi}_{L^2_{t,x}}
 &\lesssim \norm{m-g}_{L^2_{t,y}}
   +\norm{g^{\Psi^n}-g^\Psi}_{L^2_{t,x}}.
\end{aligned}
\]
Letting first \(n\to\infty\), and then \(g\to m\) in \(L^2_{t,y}\), gives
\((m^n)^{\Psi^n}\to m^\Psi\) strongly in \(L^2_{t,x}\).  Since
\(J^n\to J\) uniformly, it follows that
\(J^n(m^n)^{\Psi^n}\to Jm^\Psi\) strongly there.  After a further
subsequence, \((m^n)^{\Psi^n}\to m^\Psi\) almost everywhere on
\((0,T)\times\T\).

It remains to transfer the bounds in
\ref{item:pathwise-compactness-identification}.  The Lipschitz property of
\(\chi\) passes the \(L^1\) estimate to almost
every time.  The functional
\(f\mapsto\int_\T\chi(f)\dd x\) is convex and weakly lower semicontinuous on
\(L^2\); the representative \(m\in C_{\mathrm w}([0,T];L^2)\) therefore
extends this estimate to every time.  For \(1<\gamma<2\), Fatou's lemma
applied to the nonnegative kernel in
\eqref{eq:transformed-operator-energy-kernel}, together with
\(J^n\to J\) uniformly, gives
lower semicontinuity of the energy form.  For \(\gamma=2\), the uniform
convergence \(a^n\to a\) and the weak convergence of \(m^n\) in
\(L^2_tH^1\) imply
\[
 \sqrt{a^n}\,\partial_ym^n
 \rightharpoonup
 \sqrt a\,\partial_ym
 \quad\text{in }L^2((0,T)\times\T),
\]
and weak lower semicontinuity gives the corresponding energy bound.
Fatou's lemma applied to
\eqref{eq:besov-difference-characterization} for
\(B^{1-\beta}_{1,1}\) and \(B^s_{3,3}\) gives,
respectively,
\[
 Z_T^\beta(m)
 \lesssim\liminf_{n\to\infty}Z_T^\beta(m^n),
 \qquad
 K_T^s(m)\le
 \liminf_{n\to\infty}K_T^s(m^n)
 \quad(0<s<1/3).
\]
Together with \eqref{eq:Besov-from-KHM}, this transfers the
\(L^3_tB^s_{3,3}\) bound for \(0<s<1/3\).  At \(\gamma=2\), the inherited
\(L^\infty_tL^2\) and \(L^2_tH^1\) bounds, followed by
\eqref{eq:energy-besov}, give the stronger range \(0<s<1/2\).  Weak-star
lower semicontinuity and the weakly continuous representative transfer the
remaining energy bounds, completing the proof.
\end{proof}

\begin{corollary}[Global well-posedness for the periodic SBE]
\label{cor:laplacian-kpz-sbe}
Let \((\gamma,\alpha)=(2,0)\), let \(u_0\in L^2(\T)\) have zero spatial
mean, and let \(\mathbb X_{0,2}\) be the canonical periodic Burgers
enhancement.  Then the canonical stochastic Burgers equation
\[
 (\partial_t-\partial_x^2)u
 =\partial_x(u\diamond u)+\partial_x\xi,
 \qquad u(0)=u_0,
\]
admits, almost surely, a unique global paracontrolled solution in the sense of
\cite{GubinelliPerkowski2017KPZReloaded}.  On every finite interval it has
the representation
\[
 u=X+\Xone+\ell+Jm^\Psi,
\]
where
\[
 m\in L^\infty_tL^1\cap L^\infty_tL^2\cap L^2_tH^1,
 \qquad
 m\in L^3_tB^s_{3,3}\quad\text{for every }0<s<\frac12.
\]
Moreover, for every \(\eps_n\downarrow0\) and every sequence of smooth,
mean-zero initial data \(u_0^n\to u_0\) strongly in \(L^2\), the full
canonical mollified solution sequence converges to \(u\) in probability in
the local solution topology of
\cite[Theorem~3.7]{GubinelliPerkowski2017KPZReloaded} on each finite interval.
\end{corollary}

\begin{proof}
Apply \Cref{prop:canonical-compactness} on \([0,N]\) for
\(N\in\mathbb N\), using a diagonal deterministic extraction.  The endpoint
estimate \eqref{eq:energy-besov} gives the asserted range \(0<s<1/2\).
Model identification and the local theory are those of
\cite[Definition~3.2 and Theorems~3.5, 3.7, and~6.12]
{GubinelliPerkowski2017KPZReloaded}.  Work on the full-probability intersection
of the model-convergence events and let \(T^*\) be the maximal lifetime of
the local paracontrolled solution.  Local uniqueness identifies that solution
with the compactness limit on \([0,T]\) for every \(T<T^*\).
If \(T^*<\infty\), choose an integer \(N>T^*\) and use the common transform
fixed on \([0,N]\).  The initial-time identities
\eqref{eq:transformed-initial-data} give \(m^n(0)=u_0^n\), and
\(u_0^n\to u_0\) in \(L^2\).  The uniform a priori estimate on \([0,N]\)
therefore passes to the compactness limit in
\Cref{prop:canonical-compactness}, giving
\(m\in L^\infty(0,N;L^2)\).  By the preceding local identification,
\[
 \sup_{T<T^*}\norm m_{L^\infty(0,T;L^2)}
 \le \norm m_{L^\infty(0,N;L^2)}<\infty.
\]
Choose an exponent \(\vartheta>1/2\) in the admissible range of the
singular-initial-data theory in the cited work.  In one dimension,
\(L^2\hookrightarrow\Ccal^{-\vartheta}\), so the response estimate and the
bi-Lipschitz bounds for \(\Psi\) imply
\[
 \sup_{t<T^*}\norm{u(t)}_{\Ccal^{-\vartheta}}
 \lesssim
 1+\norm{\mathbb X_{0,2}}_{T^*}
  +\sup_{T<T^*}\norm m_{L^\infty(0,T;L^2)}
 <\infty.
\]
This contradicts \cite[Theorem~6.12]{GubinelliPerkowski2017KPZReloaded}, so
\(T^*=\infty\).  Local uniqueness gives global uniqueness.

For any deterministic subsequence, the \(L^p(\Omega)\) convergence of the
enhanced models and resolvent families yields a further subsequence converging
almost surely.  Local continuous dependence, patched using global existence
and uniqueness, then gives almost-sure solution convergence on each fixed
finite interval.  The subsequence criterion proves convergence in probability
of the full sequence.
\end{proof}

The endpoint \(\gamma=1\) lies outside this argument: the rough-drift gain
\(\gamma-1\) vanishes and the response and Zvonkin constructions become
genuinely critical.

\addtocontents{toc}{\protect\setcounter{tocdepth}{1}}
\appendix
\setlength{\parskip}{2.5pt plus .5pt}
\setlength{\abovedisplayskip}{5pt plus 1.5pt minus 2pt}
\setlength{\belowdisplayskip}{5pt plus 1.5pt minus 2pt}
\setlength{\abovedisplayshortskip}{2.5pt plus 1pt minus 1pt}
\setlength{\belowdisplayshortskip}{3.5pt plus 1pt minus 1.5pt}
\setlength{\jot}{2pt}

\section{Notation and auxiliary analytic estimates}
\label{app:notation}
\smallskip
\paragraph{\small\bfseries Function-space conventions.}
We use \(H^s=H^s(\T)\), \(B^s_{p,q}=B^s_{p,q}(\T)\), and
\(\Ccal^s=B^s_{\infty,\infty}\).  For \(h\in\T\), set
\[
f_h(x):=f(x+h)\qquad \delta_hf(x):=f_h(x)-f(x).
\]
For \(0<s<1\) and \(1\le p,q<\infty\), the periodic first-difference
characterization is
\begin{equation}\label{eq:besov-difference-characterization}
 \norm f_{B^s_{p,q}}
 \asymp
 \norm f_{L^p}
 +\left(
   \int_0^1 h^{-1-sq}\norm{\delta_hf}_{L^p}^q\dd h
  \right)^{1/q}.
\end{equation}
For positive noninteger \(s\), \(C^s\) denotes the equivalent classical
H\"older space.  A dot on a function-space symbol denotes the homogeneous
seminorm; the zero Fourier mode is ignored, without imposing a mean-zero
condition.  For a time interval \(I\) and
a Banach function space \(E\) on \(\T\), we abbreviate
\[
 L_t^pE:=L^p(I;E),\qquad
 C_tE:=C(I;E),\qquad
 C_t^\vartheta E:=C^\vartheta(I;E).
\]
Spatial subscripts are omitted when the coordinate is clear; subscripts
\(x\) and \(y\) are retained only when the original and transformed
coordinates occur in the same calculation.  In that case,
\(L^p_{t,x}:=L^p(I\times\T,\dd t\dd x)\), and similarly for \(y\).
For \(z\in\R\), set
\(z_+:=\max\{z,0\}\).  The notation \(s-\) (respectively \(s+\)) means any
fixed exponent below (respectively above) \(s\), and \(\ip f g\) denotes
the \(L^2\) or distributional duality pairing.  All unqualified norms are
spatial.  The symbols \(\lesssim\) and \(\asymp\)
suppress dependence on the fixed exponents, the torus, and locally bounded
enhanced-model and Zvonkin norms; a subscript is shown only when that
dependence matters.
For background on distributions, Sobolev and Besov spaces, Fourier
multipliers, paraproducts, and the difference characterization above, see
the lecture notes
\cite{vanZuijlen2022FunctionSpaces}.

\paragraph{\small\bfseries Affine-periodic extension.}
A smooth function \(F\) is affine-periodic if
\(F(x+1)=F(x)+c_F\) for some constant \(c_F\).  Declaring that
\(\Lambda^s\) annihilates affine functions, we define
\[
 \Lambda^sF:=\Lambda^s(F-c_Fx),
\]
where the right-hand side is the periodic Fourier multiplier defined in
the introduction.  This definition is independent of additive constants
and of the origin used for the affine part.  For \(1<\gamma<2\), the
periodic singular-integral formula extends to such \(F\) as
\[
 \Lambda^\gamma F(x)
 =c_\gamma\PV\int_{-\frac12}^{\frac12}
 [F(x)-F(x-\rho)]K_\gamma^{\mathrm{per}}(\rho)\dd\rho,
\]
with symmetric principal value.  Indeed, the affine contribution is a
multiple of \(\rho K_\gamma^{\mathrm{per}}(\rho)\), which is odd.

\paragraph{\small\bfseries Littlewood--Paley and Bony conventions.}
Fix an even spatial Littlewood--Paley partition
\((\Delta_j)_{j\ge-1}\), and write
\(\Delta_{<j-1}:=\sum_{i\le j-2}\Delta_i\).  Our Bony convention is
\[
 f\prec g:=\sum_{j\ge-1}(\Delta_{<j-1}f)\Delta_jg,
 \qquad f\succ g:=g\prec f,
 \qquad f\odot g:=\sum_{\abs{i-j}\le1}\Delta_if\,\Delta_jg,
\]
so that \(fg=f\prec g+f\odot g+f\succ g\).
Writing \(\rho_j\) for the Fourier multiplier of \(\Delta_j\), the
resonant multiplier is
\begin{equation}\label{eq:resonant-multiplier}
 \chi_\odot(k,\ell)
 :=\sum_{\abs{r-s}\le1}\rho_r(k)\rho_s(\ell).
\end{equation}

\paragraph{\small\bfseries Wick products and symmetrization.}
For jointly centered Gaussian first-chaos variables
\(Y_1,\ldots,Y_n\), the notation
\(\mathopen{:}Y_1\cdots Y_n\mathclose{:}\) denotes their Wick product,
namely the projection of \(Y_1\cdots Y_n\) onto the \(n\)-th homogeneous
Wiener chaos.  In particular,
\[
 \mathopen{:}Y_1Y_2\mathclose{:}
 =Y_1Y_2-\mathbb E[Y_1Y_2],
 \qquad
 \mathopen{:}(X^\eps)^2\mathclose{:}
 =(X^\eps)^2-\mathbb E[(X^\eps)^2].
\]
For random fields, colons denote the corresponding limit in the model
topology.  For any \(n\)-input function \(F\), its normalized
symmetrization is
\[
 \operatorname{Sym}_nF(\zeta_1,\ldots,\zeta_n)
 :=\frac1{n!}\sum_{\pi\in S_n}
 F(\zeta_{\pi(1)},\ldots,\zeta_{\pi(n)}).
\]

\paragraph{\small\bfseries Bony's products}
When the enhancement supplies a resonant product, write
\(f\odot_{\rm ren}g\) for that coordinate and
\(f\diamond g:=f\prec g+f\succ g+f\odot_{\rm ren}g\) for the complete
enhanced product. For the time-modified paraproduct, fix a nonnegative
\(\eta\in C_c^\infty(\R)\), supported in \([0,\infty)\), with
\(\int_\R\eta=1\).  On \(I=[t_0,T]\), extend
\(f\) by \(\bar f_I(s):=f((s\vee t_0)\wedge T)\), and set
\[
 \mathsf Q_j^\gamma f(t)
 :=\int_\R 2^{\gamma j}
   \eta\bigl(2^{\gamma j}(t-s)\bigr)\bar f_I(s)\dd s,
 \qquad
 f\mpara g
 :=\sum_{j\ge-1}
   \bigl(\mathsf Q_j^\gamma\Delta_{<j-1}f\bigr)\Delta_jg.
\]
Thus the low-frequency coefficient is averaged on the fractional parabolic
scale \(2^{-\gamma j}\).

\begin{lemma}[Fractional semigroup and central paracontrolled commutators]
\label{lem:fractional-commutators}
Let \(I=[t_0,T]\) and assume all functions considered are smooth.
\begin{enumerate}[label=(\roman*),leftmargin=*]
\item If \(0<a<1\) and \(b\in\R\), then
\[
\begin{aligned}
 \norm{\mathrm{Com}_{\Kcal_\gamma}(f,g)}
   _{C(I;\Ccal^{a+b+\gamma-1})}
 \lesssim
 \left(
  \norm f_{C(I;\Ccal^a)}
  +[f]_{C^{a/\gamma}(I;L^\infty)}
 \right)
 \norm g_{C(I;\Ccal^b)}.
\end{aligned}
\]
The same estimate holds with \(\Kcal_{\gamma,\lambda,t_0}\), uniformly in
\(\lambda\ge0\).

\item Let \(0<a<1\) and \(b,c\in\R\), with
\[
 b+c<0<a+b+c.
\]
Then the trilinear expression
\[
 \mathrm{Com}_{\mpara}(f,Q,g)
 :=(f\mpara Q)\odot g-f(Q\odot g),
\]
satisfies
\[
\begin{aligned}
 \norm{\mathrm{Com}_{\mpara}(f,Q,g)}
   _{C(I;\Ccal^{a+b+c})} \lesssim
 \left(
  \norm f_{C(I;\Ccal^a)}
  +[f]_{C^{a/\gamma}(I;L^\infty)}
 \right)
 \norm Q_{C(I;\Ccal^b)}
 \norm g_{C(I;\Ccal^c)}.
\end{aligned}
\]
\end{enumerate}
\end{lemma}

In the Laplacian setting, part~(i) is
\cite[Lemmas~2.8--2.9]{GubinelliPerkowski2017KPZReloaded}, while part~(ii)
combines the central commutator of
\cite[Lemma~2.4]{GubinelliImkellerPerkowski2015Paracontrolled} with the
modified-paraproduct estimate
\cite[Lemma~2.8]{GubinelliPerkowski2017KPZReloaded}; see also
\cite[Remark~2.4]{GubinelliPerkowski2017KPZReloaded}.  The same
Littlewood--Paley arguments apply with the fractional parabolic scale
\(2^{-\gamma j}\); we therefore omit the proof.

If
\(u=u'\mpara Q+u^\sharp\) and the enhanced resonance
\(Q\odot_{\rm ren}g\) is given, define
\[
 u\odot_{\rm ren}g
 :=u'(Q\odot_{\rm ren}g)
   +\mathrm{Com}_{\mpara}(u',Q,g)+u^\sharp\odot g.
\]
This and the Bony decomposition define the controlled products used below.

\subsection{Dyadic estimates for the paralinearization}
\label{app:paralinearization-dyadic-calculations}

We record the estimates deferred from
\Cref{prop:transformed-operator-paralinearization}.  Fix
\[
 0<\beta_0<\beta,
 \qquad \beta_0\ne\gamma-1,
 \qquad
 0<\sigma<\min\left\{\frac\gamma2,
  1-\frac\gamma2+\beta_0\right\}.
\]
All Littlewood--Paley blocks below act in the \(y\)-variable.  Write
\(f_k=\Delta_kf\).  We use \(j\ll k\) for the low paraproduct range
\(j\le k-2\), \(j\sim k\) for the fixed comparable-frequency overlap,
and \(j\gg k\) for the remaining higher coefficient blocks.  For
\(h\ne0\), recall
\[
 \nabla_h^-\Phi(y)
 =\frac{\Phi(y)-\Phi(y-h)}h,
 \qquad
 \nabla_0^-\Phi=\Phi'.
\]
Since \(\Psi\in C^{1+\beta}\) is bi-Lipschitz,
\(\Phi'=(J^\Phi)^{-1}\in C^\beta\), and
\(\nabla_h^-\Phi\) stays in a fixed positive interval.  Moreover,
\[
 \partial_y\bigl(\nabla_h^-\Phi(y)\bigr)^{-\gamma}
 =-\gamma
 \bigl(\nabla_h^-\Phi(y)\bigr)^{-\gamma-1}
 \frac{\Phi'(y)-\Phi'(y-h)}h.
\]
Consequently,
\[
 \big\|
 \bigl(\nabla_h^-\Phi\bigr)^{-\gamma}
 \big\|_{C^{\beta_0}}
 \lesssim1,
 \qquad
 \big\|
 \partial_y\bigl(\nabla_h^-\Phi\bigr)^{-\gamma}
 \big\|_{L^\infty}
 \lesssim\abs h^{\beta_0-1}.
\]
For \(j\ge0\), the standard Littlewood--Paley estimates
\[
 \norm{\Delta_ja}_{L^\infty}
 \lesssim2^{-\beta_0j}\norm a_{C^{\beta_0}},
 \qquad
 \norm{\Delta_ja}_{L^\infty}
 \lesssim2^{-j}\norm{\partial_ya}_{L^\infty}
\]
therefore give
\begin{equation}\label{eq:appendix-divided-difference-block}
 \left\|\Delta_j
 \bigl(\nabla_h^-\Phi\bigr)^{-\gamma}
 \right\|_{L^\infty}
 \lesssim
 \min\{2^{-\beta_0j},2^{-j}\abs h^{\beta_0-1}\}.
\end{equation}

\begin{proof}[Proof of the comparable-frequency estimate
\eqref{eq:divided-difference-comparable}]
For \(k\ge0\),
\[
 \norm{\delta_hf_k(\cdot-h)}_{L^2}
 \lesssim\min\{1,2^k\abs h\}\norm{f_k}_{L^2}.
\]
Taking absolute values and using symmetry in \(h\),
\eqref{eq:appendix-divided-difference-block} reduces the estimate for
\(j\sim k\) to the scalar integral below.  Splitting it at
\(h=2^{-k}\) gives
\[
\begin{aligned}
\int_0^{\frac12}
 h^{-\gamma}\min\{1,2^kh\}
 \min\{2^{-\beta_0j},2^{-j}h^{\beta_0-1}\}\dd h &\lesssim
 2^{k(1-\beta_0)}
 \int_0^{2^{-k}}h^{1-\gamma}\dd h
 +2^{-k}\int_{2^{-k}}^{\frac12}
 h^{\beta_0-\gamma-1}\dd h
\\
& \lesssim
 2^{k(\gamma-1-\beta_0)}
 \le 2^{k(\gamma-1-\beta_0)_+}.
\end{aligned}
\]
Multiplication by \(\norm{f_k}_{L^2}\) proves
\eqref{eq:divided-difference-comparable}.  The product in that equation
is Fourier localized in a ball at scale \(2^k\).
\end{proof}

\begin{proof}[Proof of the high-coefficient estimate
\eqref{eq:divided-difference-high}]
Suppose that \(k\ge0\) and \(j\gg k\).  Split the \(h\)-integral at
\(2^{-j}\) and \(2^{-k}\).  On \((0,2^{-j})\), combine the
\(2^{-\beta_0j}\) coefficient bound with the linear input increment.
On \((2^{-j},2^{-k})\), combine the
\(2^{-j}h^{\beta_0-1}\) coefficient bound with the same input estimate.
On \((2^{-k},\frac12)\), retain that coefficient bound and use the
uniform translation estimate.  After factoring out
\(\norm{f_k}_{L^2}\), the three ranges give
\[
\begin{aligned}
 &2^k2^{-\beta_0j}
  \int_0^{2^{-j}}h^{1-\gamma}\dd h
 +2^{k-j}
  \int_{2^{-j}}^{2^{-k}}h^{\beta_0-\gamma}\dd h\\
 &\qquad+2^{-j}\int_{2^{-k}}^{\frac12}
  h^{\beta_0-\gamma-1}\dd h.
\end{aligned}
\]
Endpoint evaluation therefore bounds the left-hand side of
\eqref{eq:divided-difference-high} by
\[
\begin{cases}
 2^{k(\gamma-1-\beta_0)}
 2^{-(2-\gamma+\beta_0)(j-k)}\norm{f_k}_{L^2},
 &\beta_0<\gamma-1,\\
 2^{-(j-k)}\norm{f_k}_{L^2},&\beta_0>\gamma-1.
\end{cases}
\]
This proves \eqref{eq:divided-difference-high}.  Since
\(j\gg k\), the product is Fourier localized in an annulus at the
coefficient scale \(2^j\).
\end{proof}

\begin{proof}[Proof of the dyadic summation in
\Cref{prop:transformed-operator-paralinearization}]
The standard annular and ball Littlewood--Paley summation estimates,
the latter using \(\sigma>0\), reduce the low and comparable sectors to
\[
 \sum_{k\ge0}
 2^{2k[\sigma+(\gamma-1-\beta_0)_+]}
 \norm{f_k}_{L^2}^2
 \lesssim
 \sum_{k\ge0}2^{k\gamma}\norm{f_k}_{L^2}^2,
\]
which is bounded by \(\norm f_{H^{\gamma/2}}^2\) because
\(\sigma+(\gamma-1-\beta_0)_+<\gamma/2\).  For \(j\gg k\), after
extracting \(2^{k\gamma/2}\norm{f_k}_{L^2}\), the two geometric factors
are \(2^{-(2-\gamma+\beta_0-\sigma)(j-k)}\) and
\(2^{-(1-\sigma)(j-k)}\).  They are summable because
\(2-\gamma+\beta_0-\sigma>0\) and \(1-\sigma>0\), so Schur's test gives
the high-coefficient bound.  Interactions with the coefficient bottom
block \(j=-1\) are either included in the low sector or are among
finitely many comparable bottom interactions.  For the bottom input
\(k=-1\), nearby coefficient blocks are finite, while the
separated tail decays as \(2^{-(2-\gamma+\beta_0)j}\) when
\(\beta_0<\gamma-1\) and as \(2^{-j}\) when
\(\beta_0>\gamma-1\); the same two inequalities give summability.  This
proves the full \(H^{\gamma/2}\to H^\sigma\) summation.
\end{proof}

\begin{proof}[Proof of the dyadic difference estimate used in the
stability argument]
If \(\Psi_1,\Psi_2\) have common \(C^{1+\beta}\) and
bi-Lipschitz bounds, set
\[
 \Phi_i:=\Psi_i^{-1},\qquad i\in\{1,2\}.
\]
The same mean-value, chain-rule, and Littlewood--Paley estimates give
\[
 \left\|\Delta_j\left[
 (\nabla_h^-\Phi_1)^{-\gamma}
 -(\nabla_h^-\Phi_2)^{-\gamma}
 \right]\right\|_{L^\infty}
 \lesssim
 \norm{\Phi_1'-\Phi_2'}_{C^{\beta_0}}
 \min\{2^{-\beta_0j},2^{-j}\abs h^{\beta_0-1}\}.
\]
The freezing difference also satisfies
\[
\begin{aligned}
 &\Big\|
 [(\nabla_h^-\Phi_1)^{-\gamma}
  -(\nabla_0^-\Phi_1)^{-\gamma}]
 -[(\nabla_h^-\Phi_2)^{-\gamma}
  -(\nabla_0^-\Phi_2)^{-\gamma}]
 \Big\|_{L^\infty}\\
 &\qquad\lesssim
 \abs h^{\beta_0}
 \norm{\Phi_1'-\Phi_2'}_{C^{\beta_0}}.
\end{aligned}
\]
Substituting these bounds into the low, comparable, and high calculations
multiplies their common \(H^\sigma\)-bound by
\(\norm{\Phi_1'-\Phi_2'}_{C^{\beta_0}}\), which is the
difference estimate used in the continuity argument.
\end{proof}

\section{Fractional enhancement and resolvent estimate}
\label{app:fractional-enhancement-estimates}

\paragraph{Scope of the appendix.}
We work at the smooth mollified level and record only the Fourier-kernel
identities, contraction estimates, and fractional power counting.  Once
the Fourier second-moment and increment bounds are established, the
finite-chaos construction, the passage to spatial regularity and time
continuity, and convergence of the mollified models follow
\cite[Definition~3.2, Theorem~9.1 and Sections~9.1--9.5]
{GubinelliPerkowski2017KPZReloaded}.  For mollifier differences, one places
the difference on a single noise leaf.  Its Fourier factor remains even,
so every oddness subtraction used below is preserved, while the resulting
small frequency weight is absorbed by the arbitrary losses in the kernel
bounds.  We use this standard completion without further comment.

\subsection{Fractional Bounds}
\label{app:fractional-bound}

Throughout this appendix, abbreviate
\begin{equation}\label{eq:fractional-exponent-relations}
 \delta:=\delta_\gamma=\gamma-1,
 \qquad
 q:=q_{\gamma,\alpha}:=\alpha+\frac{3\gamma-5}{2}.
\end{equation}
A space--time frequency is
\[\zeta=(\omega,k)\in\R\times\mathbb Z,\] where \(\omega\) is the temporal
frequency and \(k\) is the spatial Fourier mode.  For the \(i\)-th noise
input we write
\(\zeta_i=(\omega_i,k_i)\); thus \(k_i\) is always the spatial component
of \(\zeta_i\).  For a set \(A\) of input labels, set
\(\zeta_A:=\sum_{i\in A}\zeta_i\) and \(k_A:=\sum_{i\in A}k_i\).
When \(A=\{i_1,\ldots,i_r\}\), we suppress braces and commas in
subscripts, writing \(\zeta_{i_1\cdots i_r}\) and
\(k_{i_1\cdots i_r}\).  For example,
\[
 \zeta_{123}=\zeta_1+\zeta_2+\zeta_3,
 \qquad
 k_{123}=k_1+k_2+k_3.
\]
We use the fractional-parabolic norms
\[
\abs{\zeta}_\gamma:=\abs{\omega}^{1/\gamma}+\abs{k}, \quad \langle\zeta\rangle_\gamma:=1+\abs{\zeta}_\gamma.
\]
Integration in \(\zeta\) means Lebesgue integration in \(\omega\) and
counting measure in \(k\). A dyadic fractional-parabolic block
\(\langle\zeta\rangle_\gamma\asymp L\), \(L\in2^{\mathbb N_0}\), has
measure \(O(L^{\gamma+1})\).  Noise inputs have \(k_i\ne0\); all
differentiated multipliers vanish at spatial mode \(0\).

For the kernel calculation we use the stationary versions of the kernel
functions on \(\R\times\T\) for notational simplicity, with the
\(2\pi\)-normalization factors suppressed.  They are obtained by extending
every Duhamel integral from the initial time to \(-\infty\).  At each fixed
output time, the kernel function of a tree with zero initial condition is
pointwise bounded in the integration variables by its stationary version.
Fourier transformation in the time variables and Plancherel then give the
corresponding \(L^2\)-kernel bound; see
\cite[Section~9.5, the induction leading to (87), and
(87)]{GubinelliPerkowski2017KPZReloaded}.  We will also repeatedly use the
following convolution bound.
\begin{lemma}
If \(a,b>0\),
\(a+b>\gamma+1\), then
\begin{subequations}\label{eq:fractional-convolution-bounds}
\begin{equation}\label{eq:fractional-annular-convolution}
 \int_{\R\times\mathbb Z}
 \frac{\dd\zeta'}
 {\langle\zeta'\rangle_\gamma^a
  \langle\zeta-\zeta'\rangle_\gamma^b}
 \lesssim
 \langle\zeta\rangle_\gamma^{
 -\min\{a,b,a+b-(\gamma+1)\}+}.
\end{equation}
If \(a,b,c>0\) and \(a+b+c>\gamma+1\), then also
\begin{equation}\label{eq:fractional-three-weight-convolution}
\begin{aligned}
 \int_{\R\times\mathbb Z}
 \frac{\dd\zeta}
 {\langle\zeta\rangle_\gamma^a
  \langle\zeta'-\zeta\rangle_\gamma^b
  \langle\zeta''-\zeta\rangle_\gamma^c} \lesssim
 \max\{\langle\zeta'\rangle_\gamma,
        \langle\zeta''\rangle_\gamma\}^{
 -\min\{a,b,c,a+b+c-(\gamma+1)\}+}.
\end{aligned}
\end{equation}
\end{subequations}
\end{lemma}

\begin{proof}
For \(L\in2^{\mathbb N_0}\), let
\[
 A_L:=\{\zeta':\langle\zeta'\rangle_\gamma\asymp L\};
 \qquad \abs{A_L}\lesssim L^{\gamma+1}.
\]
In the region
\(\langle\zeta'\rangle_\gamma\ll\langle\zeta\rangle_\gamma\), the triangle
inequality gives
\(\langle\zeta-\zeta'\rangle_\gamma\asymp
  \langle\zeta\rangle_\gamma\).  Thus the integrand on \(A_L\) is bounded
by \(L^{-a}\langle\zeta\rangle_\gamma^{-b}\), and integration over the
\(L\)-block contributes at most
\(L^{\gamma+1-a}\langle\zeta\rangle_\gamma^{-b}\).
In the second region,
\(\langle\zeta-\zeta'\rangle_\gamma
 \ll\langle\zeta\rangle_\gamma\), the same argument holds by symmetry.
In the remaining region the two input brackets are comparable to a dyadic scale
\(L\gtrsim\langle\zeta\rangle_\gamma\), so its block contribution is
\(O(L^{\gamma+1-a-b})\).  Summing the block contributions gives
\[
\begin{aligned}
 \langle\zeta'\rangle_\gamma\ll\langle\zeta\rangle_\gamma:
 &\quad
 \langle\zeta\rangle_\gamma^{-b}
 \sum_{L\le\langle\zeta\rangle_\gamma}L^{\gamma+1-a}
 \lesssim
 \langle\zeta\rangle_\gamma^{-\min\{b,a+b-(\gamma+1)\}+},\\
 \langle\zeta-\zeta'\rangle_\gamma
 \ll\langle\zeta\rangle_\gamma:
 &\quad
 \langle\zeta\rangle_\gamma^{-a}
 \sum_{L\le\langle\zeta\rangle_\gamma}L^{\gamma+1-b}
 \lesssim
 \langle\zeta\rangle_\gamma^{-\min\{a,a+b-(\gamma+1)\}+},\\
 \langle\zeta'\rangle_\gamma
 \asymp\langle\zeta-\zeta'\rangle_\gamma
 \gtrsim\langle\zeta\rangle_\gamma:
 &\quad
 \sum_{L\ge\langle\zeta\rangle_\gamma}L^{\gamma+1-a-b}
 \lesssim
 \langle\zeta\rangle_\gamma^{-(a+b-(\gamma+1))}.
\end{aligned}
\]
Taking the largest of these three contributions proves
\eqref{eq:fractional-annular-convolution}.
The proof of \eqref{eq:fractional-three-weight-convolution} is the same
dyadic argument, now applied around the three centers
\(0,\zeta',\zeta''\), and is omitted.
\end{proof}
Estimate \eqref{eq:fractional-annular-convolution} is the fractional version
of \cite[Lemma~9.8]{GubinelliPerkowski2017KPZReloaded}.
If \(q>\delta\), then \(X\) has positive spatial regularity.  Its
convergence and time modulus follow from the standard fixed-time and
increment estimates for the stochastic convolution, while classical Bony
continuity constructs all higher driving terms and gives their convergence
and time moduli. Hence throughout the remainder of this appendix we
assume
\[
 q\le\delta
 \quad\Longleftrightarrow\quad
 \alpha\le\frac{3-\gamma}{2}<1.
\]
The first-chaos Fourier kernel of \(X^\eps\) is
\(H_\alpha^\eps\), while \(M_\gamma\) is the Fourier multiplier of
\(\Kcal_\gamma\).  For brevity, write
\begin{align}
H_i^\eps:=H_\alpha^\eps(\zeta_i), \quad
M_A:=M_\gamma(\zeta_A). 
\end{align}
Thus, for example, \(M_{123}=M_\gamma(\zeta_{123})\).

The following Lemma is the
fractional counterpart of the cancellation in
\cite[Lemma~9.5]{GubinelliPerkowski2017KPZReloaded}.

\begin{lemma}[Basic multipliers and a contracted vertex function]
\label{lem:sbe-basic-multipliers}
Let \(\phi\) be the smooth even spatial mollifier from
\Cref{prop:fractional-enhancement}.  Set
\(H_\alpha^\eps(\omega,0)=M_\gamma(\omega,0)=0\).  For \(k\ne0\), and
uniformly in \(\eps>0\),
\begin{subequations}\label{eq:sbe-basic-multipliers}
\begin{equation}\label{eq:sbe-H-multiplier}
 H_\alpha^\eps(\zeta)
 :=\frac{\abs{k}^{1-\alpha}\widehat\phi(\eps k)}
          {i\omega+\abs{k}^{\gamma}},
 \qquad
 \abs{H_\alpha^\eps(\zeta)}^2
 \lesssim
 \langle\zeta\rangle_\gamma^{-2(\alpha+\delta)}.
\end{equation}
\begin{equation}\label{eq:sbe-M-multiplier}
 M_\gamma(\zeta)
 :=\frac{ik}{i\omega+\abs{k}^{\gamma}},
 \qquad
 \abs{M_\gamma(\zeta)}^2
 \lesssim
 \langle\zeta\rangle_\gamma^{-2\delta}.
\end{equation}
\begin{equation}\label{eq:Xtwo-loop-bound}
\begin{aligned}
 L^\eps(\zeta)
 :=
 \int_{\R\times\mathbb Z}
 \abs{H_\alpha^\eps(\lambda)}^2
 M_\gamma(\zeta+\lambda)\,\dd\lambda,\quad
 \abs{L^\eps(\zeta)} \lesssim
 \langle\zeta\rangle_\gamma^{(\delta-2q)_++}.
\end{aligned}
\end{equation}
\end{subequations}
\end{lemma}

\begin{proof}
It is straightforward to see that
\[
 \abs{H_\alpha^\eps(\zeta)}^2
 \lesssim
 \langle\zeta\rangle_\gamma^{2(1-\alpha)-2\gamma}
 \overset{\eqref{eq:fractional-exponent-relations}}{=}
 \langle\zeta\rangle_\gamma^{-2(\alpha+\delta)},
 \qquad
 \abs{M_\gamma(\zeta)}^2
 \lesssim
 \langle\zeta\rangle_\gamma^{2-2\gamma}
 \overset{\eqref{eq:fractional-exponent-relations}}{=}
 \langle\zeta\rangle_\gamma^{-2\delta}.
\]

For the contracted vertex function, write
\(\zeta=(\omega,k)\) and \(\lambda=(\nu,\ell)\).
The modes \(\ell=0\) and \(\ell+k=0\) vanish because, respectively,
\(H_\alpha^\eps(\nu,0)=0\) and
\(M_\gamma(\nu+\omega,0)=0\).  On the remaining modes we use
\[
 \int_\R
 \frac{\dd\nu}
 {(\nu^2+a^2)[b+i(\nu+\omega)]}
 =
 \frac{\pi}{a(a+b+i\omega)},
 \qquad a,b>0.
\]
Taking \(a=\abs{\ell}^{\gamma}\) and
\(b=\abs{\ell+k}^{\gamma}\) in the formula above, we obtain
\[
\begin{aligned}
 L^\eps(\omega,k)
 &=
 \sum_{\substack{\ell\ne0\\ \ell+k\ne0}}
 i(\ell+k)\abs{\ell}^{2(1-\alpha)}
 \abs{\widehat\phi(\eps\ell)}^2
 \int_\R
 \frac{\dd\nu}
 {(\nu^2+\abs{\ell}^{2\gamma})
  [\abs{\ell+k}^{\gamma}+i(\nu+\omega)]}\\
 &=
 c\sum_{\ell\ne0}
 \frac{i(\ell+k)\abs{\ell}^{2(1-\alpha)-\gamma}
       \abs{\widehat\phi(\eps\ell)}^2}
 {i\omega+\abs{\ell}^{\gamma}+\abs{\ell+k}^{\gamma}}\\
 &\overset{\eqref{eq:fractional-exponent-relations}}{=}
 c\sum_{\ell\ne0}
 i\abs{\widehat\phi(\eps\ell)}^2
 \left[
  \frac{(\ell+k)\abs{\ell}^{2\delta-1-2q}}
       {i\omega+\abs{\ell}^{\gamma}+\abs{\ell+k}^{\gamma}}
  -\frac12\operatorname{sgn}(\ell)
   \abs{\ell}^{\delta-1-2q}
 \right],
\end{aligned}
\]
for some constant c. In the second line, the mode \(\ell=-k\) has been reinserted as a zero
summand.  In the final line, the second term in brackets has odd spatial
summand because \(\abs{\widehat\phi(\eps\ell)}^2\) is even and rapidly
decreasing; hence its full spatial sum vanishes.

We first record the two bounds required for the full summand:
\[
\begin{aligned}
 \abs{\ell}^{2\delta-1-2q}
 \abs{\widehat\phi(\eps\ell)}^2
 \left|
  \frac{\ell+k}
  {i\omega+\abs{\ell}^{\gamma}+\abs{\ell+k}^{\gamma}}
  -\frac{\ell}{2\abs{\ell}^{\gamma}}
 \right|\lesssim
 \begin{cases}
  \abs{\ell}^{\delta-1-2q},
  &0<\abs{\ell}\le4\langle\zeta\rangle_\gamma,\\
  \langle\zeta\rangle_\gamma
  \abs{\ell}^{\delta-2-2q},
  &\abs{\ell}>4\langle\zeta\rangle_\gamma.
 \end{cases}
\end{aligned}
\]
Assuming these bounds momentarily, we obtain
\[
\begin{aligned}
 \abs{L^\eps(\zeta)}
 &\lesssim
 \sum_{0<\abs\ell\le4\langle\zeta\rangle_\gamma}
 \abs{\ell}^{\delta-1-2q}
 +
 \langle\zeta\rangle_\gamma
 \sum_{\abs\ell>4\langle\zeta\rangle_\gamma}
 \abs{\ell}^{\delta-2-2q},
\end{aligned}
\]
where the second series is summable because the admissible parameter range
gives \(2q>\delta-1\).  The bound
\eqref{eq:Xtwo-loop-bound} follows directly.  It remains to prove the claimed bounds. For \(0<\abs\ell\le4\langle\zeta\rangle_\gamma\), the two terms in
brackets satisfy
\[
\begin{aligned}
 \left|
 \frac{\ell+k}
 {i\omega+\abs{\ell}^{\gamma}+\abs{\ell+k}^{\gamma}}
 \right|
 \le
 \frac{\abs{\ell+k}}
 {\abs{\ell}^{\gamma}+\abs{\ell+k}^{\gamma}}
 \lesssim
 \abs{\ell}^{1-\gamma}.
\end{aligned}
\]
For \(\abs\ell>4\langle\zeta\rangle_\gamma\),
\(\abs{\ell+\theta k}\asymp\abs\ell\) for \(0\le\theta\le1\).  The
difference in brackets is exactly
\[
\begin{aligned}
 &\int_0^1
 \partial_\theta
 \left[
 \frac{\ell+\theta k}
 {i\theta\omega+\abs{\ell}^{\gamma}
  +\abs{\ell+\theta k}^{\gamma}}
 \right]\dd\theta .
\end{aligned}
\]
The differentiated integrand is
\[
\begin{aligned}
 \frac{k}
 {i\theta\omega+\abs{\ell}^{\gamma}
  +\abs{\ell+\theta k}^{\gamma}}-
 \frac{(\ell+\theta k)
 \left[
 i\omega+\gamma k\operatorname{sgn}(\ell+\theta k)
 \abs{\ell+\theta k}^{\gamma-1}
 \right]}
 {\left[
 i\theta\omega+\abs{\ell}^{\gamma}
 +\abs{\ell+\theta k}^{\gamma}
 \right]^2}.
\end{aligned}
\]
Its modulus is bounded by
\[
\begin{aligned}
 \abs{k}\abs{\ell}^{-\gamma}
 +\abs{\omega}\abs{\ell}^{1-2\gamma}
 &\lesssim
 \langle\zeta\rangle_\gamma\abs{\ell}^{-\gamma},
\end{aligned}
\]
because
\(\abs{k}\le\langle\zeta\rangle_\gamma\),
\(\abs{\omega}\le\langle\zeta\rangle_\gamma^\gamma\), and
\(\langle\zeta\rangle_\gamma\lesssim\abs{\ell}\). This proves the second pointwise bound and completes
\eqref{eq:Xtwo-loop-bound}. 
\end{proof}

\begin{remark}[Chaos kernels in Fourier variables]
Every kernel \(G(\zeta_1,\ldots,\zeta_n)\) of a Wiener-chaos component
below is displayed in
the Fourier variables of its noise inputs.  We write
\(\mathcal I_n(G)\) for the usual physical-space Wiener--It\^o map
applied after inverse Fourier transformation in those \(n\) variables;
thus \(\mathcal I_n\) is not a separate Fourier-space stochastic
integral.  When the output is written as
\(\mathcal I_n(G)(\zeta)\), it denotes the Fourier transform of the
resulting physical-space random field.  The sign in a Fourier Wick
contraction follows from the bilinear Plancherel identity
\[
 \int f(z)g(z)\,\dd z
 =
 c\int\widehat f(\lambda)\widehat g(-\lambda)\,\dd\lambda.
\]
Consequently, the usual physical-space Wick contraction pairs the Fourier
frequencies \(\lambda\) and \(-\lambda\).  The output frequency of an
\(n\)-input kernel is \(\zeta_1+\cdots+\zeta_n\).  For a symmetric
kernel \(G\), the Wiener isometry gives the Fourier second-moment density
\[
 \mathcal C_G(\zeta)
 :=n!\int_{\zeta_1+\cdots+\zeta_n=\zeta}
 \abs{G(\zeta_1,\ldots,\zeta_n)}^2
 \dd\zeta_1\cdots\dd\zeta_{n-1}.
\]
In Fourier second-moment form,
\(\mathbb E\abs{\mathcal I_n(G)(\zeta)}^2=\mathcal C_G(\zeta)\).
\end{remark}

\subsection{Fractional Burgers enhancement}
\label{app:fractional-enhancement-proof}

We follow the Fourier-chaos organization of
\cite[Section~9.5]{GubinelliPerkowski2017KPZReloaded}, with fractional
parabolic multipliers and the additional response coordinates required here.
For every \(\eps>0\), the stationary driving terms
\((\Xone)^\eps=\Kcal_\gamma\mathopen{:}(X^\eps)^2\mathclose{:}\) and
\((\Xtwo)^\eps=\Kcal_\gamma(X^\eps\diamond(\Xone)^\eps)\) have the
following chaos decompositions.
\begin{equation}\label{eq:Xone-Xtwo-chaos-decompositions}
 (\Xone)^\eps
 =\mathcal I_2(G_{\Xone,2}^\eps),
 \qquad
 (\Xtwo)^\eps
 =\mathcal I_3(G_{\Xtwo,3}^\eps)
  +\mathcal I_1(G_{\Xtwo,1}^\eps),
\end{equation}
with
\begin{subequations}\label{eq:Xone-Xtwo-chaos-kernels}
\begin{equation}\label{eq:Xone-chaos-kernel}
 G_{\Xone,2}^\eps(\zeta_1,\zeta_2)
 =M_{12}H_1^\eps H_2^\eps.
\end{equation}
and
\begin{equation}\label{eq:Xtwo-chaos-kernels}
\begin{aligned}
 G_{\Xtwo,3}^\eps(\zeta_1,\zeta_2,\zeta_3)
 & =
 \operatorname{Sym}_3\left[
  M_{123}M_{12}H_1^\eps H_2^\eps H_3^\eps
 \right],\\
 G_{\Xtwo,1}^\eps(\zeta)
 &=2M_\gamma(\zeta)H_\alpha^\eps(\zeta)L^\eps(\zeta).
\end{aligned}
\end{equation}
\end{subequations}
Here \(L^\eps\) is the contracted vertex function defined in
\eqref{eq:Xtwo-loop-bound}.
Indeed, taking the space--time Fourier transform of the stationary equation
\(L_\gamma X^\eps=\abs{\partial_x}^{1-\alpha}\xi^\eps\) gives, for
\(k\ne0\),
\[
 (i\omega+\abs{k}^{\gamma})\widehat X^\eps(\omega,k)
 =
 \abs{k}^{1-\alpha}\widehat\phi(\eps k)\widehat\xi(\omega,k) \implies
 \widehat X^\eps(\zeta)
 =H_\alpha^\eps(\zeta)\widehat\xi(\zeta),
 \quad
 X^\eps=\mathcal I_1(H_\alpha^\eps).
\]
At the mollified level the Wiener product formula gives
\[
 (X^\eps)^2
 =\mathcal I_2(H_1^\eps H_2^\eps)
   +\mathbb E[(X^\eps)^2].
\]
Wick
centering removes the second term.  In Fourier variables,
\[
 \widehat{\Kcal_\gamma Y}(\zeta)
 =M_\gamma(\zeta)\widehat Y(\zeta),
 \qquad
 \zeta=\zeta_1+\zeta_2,
\]
so the Duhamel operator contributes
\(M_\gamma(\zeta_1+\zeta_2)=M_{12}\).  Consequently,
\(
 (\Xone)^\eps
 =\mathcal I_2(M_{12}H_1^\eps H_2^\eps)
\),
and this proves the first identities in
\eqref{eq:Xone-Xtwo-chaos-decompositions} and
\eqref{eq:Xone-chaos-kernel}. For
\((\Xtwo)^\eps=\Kcal_\gamma(X^\eps\diamond(\Xone)^\eps)\), the three
Bony pieces of the smooth product sum to
\(X^\eps(\Xone)^\eps\).  By the Fourier Wiener--It\^o convention above,
the elementary product formula\footnote{\interlinepenalty=10000
All tensor operations here act on the displayed Fourier kernels.  For
symmetric \(m\)- and \(n\)-input kernels,
\(f\otimes_{\mathrm{sym}}g\) is the normalized symmetrization
\(\operatorname{Sym}_{m+n}[f(\zeta_1,\ldots,\zeta_m)
g(\zeta_{m+1},\ldots,\zeta_{m+n})]\).  For a one-input kernel \(f\)
and a symmetric two-input kernel \(g\),
\((f\otimes_1g)(\zeta):=\int_{\R\times\mathbb Z}
f(\lambda)g(\zeta,-\lambda)\,\dd\lambda\).
}
\[
 \mathcal I_1(f)\mathcal I_2(g)
 =\mathcal I_3(f\otimes_{\mathrm{sym}}g)
  +2\mathcal I_1(f\otimes_1g)
\]
to \(f=H_\alpha^\eps\) and \(g=G_{\Xone,2}^\eps\) gives, before the
final \(\Kcal_\gamma\),
\[
 X^\eps(\Xone)^\eps
 =
 \mathcal I_3(
  H_\alpha^\eps\otimes_{\mathrm{sym}}G_{\Xone,2}^\eps)
 +2\mathcal I_1(
  H_\alpha^\eps\otimes_1G_{\Xone,2}^\eps).
\]
The one-index Wick contraction is
\begin{equation}\label{eq:Xtwo-one-contraction}
\begin{aligned}
 (H_\alpha^\eps\otimes_1G_{\Xone,2}^\eps)(\zeta)
 =
 \int_{\R\times\mathbb Z}
 H_\alpha^\eps(\lambda)
 G_{\Xone,2}^\eps(\zeta,-\lambda)\,\dd\lambda =
 H_\alpha^\eps(\zeta)L^\eps(\zeta).
\end{aligned}
\end{equation}
The second equality follows from
\eqref{eq:Xone-chaos-kernel},
\(H_\alpha^\eps(-\lambda)
=\overline{H_\alpha^\eps(\lambda)}\), the evenness of its modulus, and
\(\lambda\mapsto-\lambda\).  The factor \(2\) counts the two symmetric
inputs of \(G_{\Xone,2}^\eps\) that can be contracted. This proves the remaining identities in
\eqref{eq:Xone-Xtwo-chaos-decompositions} and
\eqref{eq:Xtwo-chaos-kernels}.

\begin{lemma}[Fourier second-moment bounds for \(\Xone\) and \(\Xtwo\)]
\label{lem:Xone-Xtwo-kernel-bounds}
The kernels in \eqref{eq:Xone-chaos-kernel} and
\eqref{eq:Xtwo-chaos-kernels} have Fourier
second-moment densities satisfying, uniformly in \(\eps\),
\begin{equation*}
\begin{aligned}
 \mathcal C_{G_{\Xone,2}^\eps}(\zeta)
 &\lesssim
 \langle\zeta\rangle_\gamma^{
  -(\gamma+1)-2(2q-\delta)+}, \notag \\
 \mathcal C_{G_{\Xtwo,j}^\eps}(\zeta) & \lesssim
 \langle\zeta\rangle_\gamma^{
  -(\gamma+1)-2\min\{3q-\delta,q\}+}, \quad j\in\{1,3\}.
\end{aligned}
\end{equation*}
\end{lemma}

\begin{proof}[Proof Outline]
In the present regime,
\[
 \abs{H_\alpha^\eps(\zeta)}^2
 \lesssim
 \langle\zeta\rangle_\gamma^{-2(\alpha+\delta)}
 \overset{\eqref{eq:fractional-exponent-relations}}{=}
 \langle\zeta\rangle_\gamma^{
  -(\gamma+1)-2(q-\delta)}.
\]
Since \(q\le\delta\), the convolution minimum in \eqref{eq:fractional-annular-convolution} is
\(\gamma+1+4(q-\delta)\), and therefore
\begin{align}
 \mathcal C_{G_{\Xone,2}^\eps}(\zeta)
 &=2\int_{\zeta_1+\zeta_2=\zeta}
   \abs{M_{12}H_1^\eps H_2^\eps}^2\,\dd\zeta_1 \notag \\ & \lesssim
 \langle\zeta\rangle_\gamma^{-2\delta}
 \int_{\R\times\mathbb Z}
 \frac{\dd\zeta_1}
 {\langle\zeta_1\rangle_\gamma^{
    \gamma+1+2(q-\delta)}
  \langle\zeta-\zeta_1\rangle_\gamma^{
    \gamma+1+2(q-\delta)}} 
 \lesssim
 \langle\zeta\rangle_\gamma^{
  -(\gamma+1)-2(2q-\delta)+}.
 \label{eq:Xone-Ito-bound}
\end{align}

For the third-chaos component, Jensen's inequality and permutation invariance
remove the normalized symmetrization.  Grouping its first two inputs
as the \(\Xone\)-subtree gives
\begin{align}
 \mathcal C_{G_{\Xtwo,3}^\eps}(\zeta)
 &\lesssim
 \abs{M_\gamma(\zeta)}^2
 \int_{\R\times\mathbb Z}
 \mathcal C_{G_{\Xone,2}^\eps}(\zeta')
 \abs{H_\alpha^\eps(\zeta-\zeta')}^2\,\dd\zeta'
 \lesssim
 \langle\zeta\rangle_\gamma^{
  -(\gamma+1)-2\min\{3q-\delta,q\}+},
 \label{eq:Xtwo-third-chaos-Ito-bound}
\end{align}
For the first-chaos component, the contracted vertex gives
\begin{align}
 \mathcal C_{G_{\Xtwo,1}^\eps}(\zeta)
 &=\abs{G_{\Xtwo,1}^\eps(\zeta)}^2
 =4\abs{M_\gamma(\zeta)H_\alpha^\eps(\zeta)}^2
   \abs{L^\eps(\zeta)}^2
 \lesssim
 \langle\zeta\rangle_\gamma^{
  -(\gamma+1)-2\min\{3q-\delta,q\}+}.
 \label{eq:Xtwo-first-chaos-Ito-bound}
\end{align}
\end{proof}

Since the first- and third-chaos components of \(\Xtwo\) are orthogonal,
\Cref{lem:Xone-Xtwo-kernel-bounds} also gives
\[
 s_{\alpha,\gamma}(\Xtwo)=\min\{3q-\delta,q\}.
\]

The preceding calculation gives
\(X^\eps=\mathcal I_1(H_\alpha^\eps)\) and the kernels of
\(\Xone\) and \(\Xtwo\) in
\eqref{eq:Xone-Xtwo-chaos-decompositions}--
\eqref{eq:Xtwo-chaos-kernels}.  For the other three driving terms
among the first six, the kernels of the highest-chaos components are
\[
\begin{aligned}
 G_{\mathbb X_Q,2}^\eps
 &=
 \operatorname{Sym}_2\left[
  \chi_\odot(k_1,k_2)M_1H_1^\eps H_2^\eps
 \right],\qquad
 G_{\Xfour,4}^\eps
 =
 \operatorname{Sym}_4\left[
  M_{1234}M_{12}M_{34}\prod_{i=1}^4H_i^\eps
 \right],\\
 G_{\XthreeRes,4}^\eps
 &=
 \operatorname{Sym}_4\left[
  \chi_\odot(k_4,k_{123})M_{1234}M_{123}M_{12}
  \prod_{i=1}^4H_i^\eps
 \right].
\end{aligned}
\]
Here \(\chi_\odot\) is the resonant multiplier defined in
\eqref{eq:resonant-multiplier}.

The zeroth-chaos component of \(\mathbb X_Q^\eps\) vanishes.  Indeed,
\[
 \int_\R
 M_\gamma(-\omega,-k)
 \abs{H_\alpha^\eps(\omega,k)}^2\,\dd\omega
 =-
 \frac{\pi ik\,\abs{k}^{2(1-\alpha)}
       \abs{\widehat\phi(\eps k)}^2}
      {2\abs{k}^{2\gamma}}.
\]
The right-hand side is odd in \(k\), whereas
\(\chi_\odot(-k,k)\) and the mollifier factor are even.  Its sum over
\(k\ne0\) is therefore zero.  Consequently,
\(\mathbb X_Q^\eps=\mathcal I_2(G_{\mathbb X_Q,2}^\eps)\).

For the direct construction of the remaining three driving coordinates,
assume without loss of generality that \(q\le\delta/2\).  Indeed, if
\(\delta/2<q\le\delta\), then \(2q-\delta>0\), and Bony continuity
together with the Schauder gain gives the continuous reconstructions
\[
 \mathbb X_Q\in\Ccal^{2q-\delta-},
 \qquad
 \Xfour,\XthreeRes\in\Ccal^{2q-}.
\]

\begin{lemma}[Kernel bounds for the first six driving coordinates]
\label{lem:first-six-enhancement-kernel-bounds}
Assume \(q\le\delta/2\).  In addition to
\eqref{eq:sbe-H-multiplier} and
\Cref{lem:Xone-Xtwo-kernel-bounds}, the remaining kernels satisfy,
uniformly in \(\eps\),
\[
 \mathcal C_{G_{\mathbb X_Q,2}^\eps}(\zeta)
 \lesssim
 \langle\zeta\rangle_\gamma^{-(\gamma+1)-2(2q-\delta)+}.
\]
Moreover, for \(\tau\in\{\Xfour,\XthreeRes\}\),
\[
 \sum_{n\in\{2,4\}}\mathcal C_{G_{\tau,n}^\eps}(\zeta)
 \lesssim
 \langle\zeta\rangle_\gamma^{-(\gamma+1)-2(4q-\delta)+}.
\]
The zeroth-chaos components of \(\Xfour\) and \(\XthreeRes\) vanish.
\end{lemma}
In this range, the estimate for \(\mathbb X_Q\) is the fractional
space--time Fourier version of the \(Q\circ X\) calculation in
\cite[Section~9.5]{GubinelliPerkowski2017KPZReloaded}.  Because
\(\chi_\odot\) compares only spatial modes, one first integrates the
temporal frequency and then sums the resonant spatial blocks.  This
separation, which is the only essential difference of \Cref{lem:first-six-enhancement-kernel-bounds} from above, is carried out explicitly in the proof of
\Cref{lem:lower-contraction-loop-estimates}. We therefore omit the proof, and it remains to construct the response coordinates in the singular regime.
Indeed,
\[
 5q>2\delta
 \quad\Longrightarrow\quad
 \XoneFourB,\ \XtwoThree,\ \XoneFourC
 \text{ are continuous reconstructions from the lower enhancement},
\]
with the reconstruction of \(\XoneFourC\) following from
\eqref{eq:Xthree-controlled} and \(\mathbb X_Q\).  We can therefore restrict
to the regime \eqref{eq:sbe-tree-response-regime}, where
\[
\frac{\delta}{3}<q\le\frac{2\delta}{5},\qquad
\begin{aligned}
 s(X)&=q-\delta,&
 s(\Xone)&=s(\mathbb X_Q)=2q-\delta,\\
 s(\Xtwo)&=3q-\delta,&
 s(\Xfour)&=s(\XthreeRes)=4q-\delta.
\end{aligned}
\]

\subsection{fractional Burgers enhancement: new trees}

We now focus on the new trees not constructed in
\cite{GubinelliPerkowski2017KPZReloaded}.  The next lemma supplies the
lower-chaos contraction bounds needed for
Proposition~\ref{prop:fractional-enhancement}.

\begin{lemma}[Bounds for lower-chaos contraction kernels]
\label{lem:lower-contraction-loop-estimates}
Assume
\[
 \frac{\delta}{3}<q\le\frac{2\delta}{5},
\]
For \(\zeta=(\omega,k)\), \(\zeta'=(\omega',k')\), and
\(\zeta''=(\omega'',k'')\), uniformly in \(\eps\),
\begin{subequations}\label{eq:lower-contraction-loop-bounds}
\begin{equation}\label{eq:lower-contraction-cutoff-loop}
\begin{aligned}
 &\sum_{K\in2^{\mathbb N_0}}
 \left|
 \sum_{1+\abs k\asymp K}\int_\R
 \abs{H_\alpha^\eps(\zeta)}^2
 \chi_\odot(-k,k+k')
 M_\gamma(\zeta+\zeta')\,\dd\omega
 \right|
 \lesssim
 \langle\zeta'\rangle_\gamma^{\delta-2q}.
\end{aligned}
\end{equation}
and
\begin{equation}\label{eq:lower-contraction-two-multiplier-loop}
\begin{aligned}
 \left|
 \int_{\R\times\mathbb Z}
 \abs{H_\alpha^\eps(\zeta)}^2
 \chi_\odot(-k,k+k')
 M_\gamma(\zeta'-\zeta)M_\gamma(\zeta''-\zeta)\,\dd\zeta
 \right| \lesssim
 \max\left\{\langle\zeta'\rangle_\gamma,
        \langle\zeta''\rangle_\gamma\right\}^{-2q+}.
\end{aligned}
\end{equation}
\end{subequations}
\end{lemma}

\begin{proof}
Estimate \eqref{eq:lower-contraction-two-multiplier-loop} is a direct
application of \eqref{eq:fractional-three-weight-convolution}.  Indeed,
\(\abs{\chi_\odot}\lesssim1\) and
\eqref{eq:sbe-basic-multipliers} give
\[
\begin{aligned}
 &\left|\int
 \abs{H_\alpha^\eps(\zeta)}^2
 \chi_\odot(-k,k+k')
 M_\gamma(\zeta'-\zeta)M_\gamma(\zeta''-\zeta)\,\dd\zeta\right|\\
 &\quad\lesssim
 \int\frac{\dd\zeta}
 {\langle\zeta\rangle_\gamma^{\gamma+1+2(q-\delta)}
  \langle\zeta'-\zeta\rangle_\gamma^\delta
  \langle\zeta''-\zeta\rangle_\gamma^\delta}
 \lesssim
 \max\{\langle\zeta'\rangle_\gamma,
        \langle\zeta''\rangle_\gamma\}^{-2q+}.
\end{aligned}
\]
Only \(\abs{\chi_\odot}\lesssim1\) was used, so the same estimate holds
if \(\chi_\odot\) is omitted or its arguments are interchanged.  We
therefore focus on \eqref{eq:lower-contraction-cutoff-loop}. Fix \(K\in2^{\mathbb N_0}\) and restrict the spatial frequency to
\(1+\abs k\asymp K\). At the frozen external frequency \(\zeta'=0\), the temporal identity used
in the proof of
\Cref{lem:sbe-basic-multipliers} gives
\[
\begin{aligned}
 \int_\R
 \abs{H_\alpha^\eps(\omega,k)}^2M_\gamma(\omega,k)\,\dd\omega
 =
 c\,i\operatorname{sgn}(k)
 \abs{k}^{\delta-1-2q}
 \abs{\widehat\phi(\eps k)}^2.
\end{aligned}
\]
The right-hand side is odd in \(k\), whereas
\(\chi_\odot(-k,k)\), the mollifier factor, and the spatial block are
even.  Hence the frozen block sum vanishes:
\[
 \sum_{1+\abs k\asymp K}\int_\R
 \abs{H_\alpha^\eps(\omega,k)}^2
 \chi_\odot(-k,k)M_\gamma(\zeta)\,\dd\omega=0.
\]
For \(K\ge4\langle\zeta'\rangle_\gamma\), subtract this identity from
the shifted contraction kernel and change the two factors one at a time:
\[
\begin{aligned}
 \chi_\odot(-k,k+k')M_\gamma(\zeta+\zeta')
  -\chi_\odot(-k,k)M_\gamma(\zeta)
 & =
 [\chi_\odot(-k,k+k')-\chi_\odot(-k,k)]
 M_\gamma(\zeta+\zeta')\\
 &\quad+
 \chi_\odot(-k,k)
 [M_\gamma(\zeta+\zeta')-M_\gamma(\zeta)].
\end{aligned}
\]
In this regime, both \(1+\abs k\) and \(1+\abs{k+k'}\) are comparable
to \(K\).  Thus
\(\abs{M_\gamma(\zeta+\zeta')}+\abs{M_\gamma(\zeta)}
\lesssim K^{-\delta}\).
Only Littlewood--Paley symbols at scale \(K\) enter the difference of
the resonant cutoffs.  Since their derivatives are \(O(K^{-1})\),
for \(0<\vartheta<1\),
\[
 \abs{\chi_\odot(-k,k+k')-\chi_\odot(-k,k)}
 \lesssim
 \min\left\{1,\frac{\abs{k'}}K\right\}
 \lesssim
 \left(\frac{\langle\zeta'\rangle_\gamma}{K}\right)^\vartheta.
\]
Along the segment from \(\zeta\) to \(\zeta+\zeta'\), the spatial
frequency remains comparable to \(K\).  The continuous extension of
\(M_\gamma\) in that variable satisfies
\[
 \abs{\partial_kM_\gamma}\lesssim K^{-\delta-1},
 \qquad
 \abs{\partial_\omega M_\gamma}\lesssim K^{-\delta-\gamma}.
\]
The mean-value theorem, together with the trivial bound
\(\abs{M_\gamma}\lesssim K^{-\delta}\), therefore yields
\[
 \abs{M_\gamma(\zeta+\zeta')-M_\gamma(\zeta)}
 \lesssim
 K^{-\delta}
 \min\left\{1,
 \frac{\abs{k'}}K+\frac{\abs{\omega'}}{K^\gamma}\right\}
 \lesssim
 K^{-\delta}
 \left(\frac{\langle\zeta'\rangle_\gamma}{K}\right)^\vartheta.
\]
Combining these two estimates gives
\[
\begin{aligned}
 &\abs{\chi_\odot(-k,k+k')-\chi_\odot(-k,k)}
  \abs{M_\gamma(\zeta+\zeta')}\\
 &\quad+
 \abs{\chi_\odot(-k,k)}
 \abs{M_\gamma(\zeta+\zeta')-M_\gamma(\zeta)}
 \lesssim
 \left(\frac{\langle\zeta'\rangle_\gamma}{K}\right)^\vartheta
 K^{-\delta}.
\end{aligned}
\]
The mass of the noise kernel on the same spatial block satisfies
\[
\begin{aligned}
 \sum_{1+\abs k\asymp K}\int_\R
 \abs{H_\alpha^\eps(\omega,k)}^2\,\dd\omega
 &\lesssim
 \sum_{1+\abs k\asymp K}
 \abs k^{2(1-\alpha)-\gamma}
 \lesssim K^{1+2(1-\alpha)-\gamma}
 \overset{\eqref{eq:fractional-exponent-relations}}{=}
 K^{2\delta-2q}.
\end{aligned}
\]
Together with the preceding estimate, this bounds the absolute value of
the \(K\)-block in \eqref{eq:lower-contraction-cutoff-loop} by
\[
 \langle\zeta'\rangle_\gamma^\vartheta
 K^{\delta-2q-\vartheta}. 
 \]
 For \(K\le4\langle\zeta'\rangle_\gamma\), the support of
\(\chi_\odot(-k,k+k')\) gives the spatial localization
\(1+\abs{k+k'}\asymp1+\abs k\asymp K\).  Hence
\(\abs{M_\gamma(\zeta+\zeta')}\lesssim K^{-\delta}\), and direct
estimation bounds the \(K\)-block by \(K^{\delta-2q}\).  Choosing
\(\delta-2q<\vartheta<1\) and summing the near and far
blocks yields
\[
\begin{aligned}
 \sum_{K\le4\langle\zeta'\rangle_\gamma}K^{\delta-2q}
 +\langle\zeta'\rangle_\gamma^\vartheta
  \sum_{K\ge4\langle\zeta'\rangle_\gamma}
  K^{\delta-2q-\vartheta}
 \lesssim
 \langle\zeta'\rangle_\gamma^{\delta-2q}.
\end{aligned}
\]
This proves \eqref{eq:lower-contraction-cutoff-loop}.  It is the
resonant-multiplier version of the
oddness cancellation in \eqref{eq:Xtwo-loop-bound}.
\end{proof}

\begin{lemma}[Fifth-chaos kernels]
\label{lem:response-leading-kernel-bounds}
In the regime \eqref{eq:sbe-tree-response-regime}, the fifth-chaos
kernels of the three response coordinates are
\begin{subequations}\label{eq:fifth-chaos-tree-kernels}
\begin{align}
 G_{\XoneFourB,5}^\eps
 &=
 \operatorname{Sym}_5\left[
  \chi_\odot(k_5,k_{1234})
  M_{12345}M_{1234}M_{12}M_{34}
  \prod_{i=1}^5H_i^\eps
 \right],
 \label{eq:worked-balanced-kernel}\\
 G_{\XtwoThree,5}^\eps
 &=
 \operatorname{Sym}_5\left[
  \chi_\odot(k_{12},k_{345})
  M_{12345}M_{12}M_{345}M_{34}
  \prod_{i=1}^5H_i^\eps
 \right],
 \label{eq:worked-two-tree-kernel}\\
 G_{\XoneFourC,5}^\eps
 &=
 \operatorname{Sym}_5\left[
  \chi_\odot(k_5,k_{1234})
  M_{12345}M_{1234}M_{123}M_{12}
  \prod_{i=1}^5H_i^\eps
 \right].
 \label{eq:worked-comb-kernel}
\end{align}
\end{subequations}
Moreover, uniformly in \(\eps\),
\[
 \mathcal C_{G_{\XoneFourB,5}^\eps}(\zeta)
 +\mathcal C_{G_{\XtwoThree,5}^\eps}(\zeta)
 \lesssim
 \langle\zeta\rangle_\gamma^{
  -(\gamma+1)-2(5q-\delta)+},
\]
and
\[
 \mathcal C_{G_{\XoneFourC,5}^\eps}(\zeta)
 \lesssim
 \langle\zeta\rangle_\gamma^{-(\gamma+1)-4q+}.
\]
\end{lemma}

\begin{proof}
We first prove the estimate for \(\XoneFourB\).  Its final root joins a noise
leaf to \(\Xfour\).  By \eqref{eq:sbe-H-multiplier} and
\Cref{lem:first-six-enhancement-kernel-bounds}, these inputs have
regularities \(q-\delta\) and \(4q-\delta\), up to arbitrarily small
losses. It remains to estimate the final resonant root.  If
\(\zeta'+\zeta''=\zeta\) and
\(\chi_\odot(k',k'')\ne0\), then spatial resonance and
\(k=k'+k''\) give
\(\abs{k}\lesssim(1+\abs{k'})\wedge(1+\abs{k''})
\le\langle\zeta'\rangle_\gamma\wedge
\langle\zeta''\rangle_\gamma\).
Together with \eqref{eq:sbe-M-multiplier}, this gives
\[
\begin{aligned}
 &\abs{\chi_\odot(k',k'')M_\gamma(\zeta)}^2\\
 &\quad\lesssim
 \abs{\chi_\odot(k',k'')}^2
 \langle\zeta\rangle_\gamma^{-2\delta}
 \min\!\left\{
  1,
  \frac{(\langle\zeta'\rangle_\gamma
          \wedge\langle\zeta''\rangle_\gamma)^2}
       {\langle\zeta\rangle_\gamma^2}
 \right\}.
\end{aligned}
\]
The covariance at the root is therefore bounded by
\[
\begin{aligned}
 \langle\zeta\rangle_\gamma^{-2\delta}
 \int_{\zeta'+\zeta''=\zeta}
 \frac{\abs{\chi_\odot(k',k'')}^2}
 {\langle\zeta'\rangle_\gamma^{\gamma+1+2(q-\delta)}
  \langle\zeta''\rangle_\gamma^{\gamma+1+2(4q-\delta)}}\cdot
 \min\!\left\{
  1,
  \frac{(\langle\zeta'\rangle_\gamma
          \wedge\langle\zeta''\rangle_\gamma)^2}
       {\langle\zeta\rangle_\gamma^2}
 \right\}\dd\zeta'.
\end{aligned}
\]
After using resonance in this inequality, we only need
\(\abs{\chi_\odot}\lesssim1\).  We first integrate the temporal
frequency and then sum full fractional-parabolic blocks; a block of
scale \(L\) has mass \(O(L^{\gamma+1})\). If \(\langle\zeta'\rangle_\gamma
\le\frac12\langle\zeta\rangle_\gamma\), then
\(\langle\zeta''\rangle_\gamma\asymp
\langle\zeta\rangle_\gamma\).  The minimum supplies two additional
powers of the smaller scale, and these blocks contribute at most
\[
 \langle\zeta\rangle_\gamma^{
  -2\delta-(\gamma+1)-2(4q-\delta)-2}
 \sum_{L\le\langle\zeta\rangle_\gamma}L^{2-2(q-\delta)}.
\]
If instead \(\zeta''\) is the smaller input, the same argument gives
\[
 \langle\zeta\rangle_\gamma^{
  -2\delta-(\gamma+1)-2(q-\delta)-2}
 \sum_{L\le\langle\zeta\rangle_\gamma}L^{2-2(4q-\delta)}.
\]
In the remaining region the two input blocks are comparable at a scale
\(L\gtrsim\langle\zeta\rangle_\gamma\); bounding the minimum by one
gives
\[
 \langle\zeta\rangle_\gamma^{-2\delta}
 \sum_{L\ge\langle\zeta\rangle_\gamma}
 L^{-(\gamma+1)-2(q-\delta)-2(4q-\delta)}.
\]
In the regime \eqref{eq:sbe-tree-response-regime}, the first two sums
are controlled by their largest block and the last by its smallest
block.  Restoring the suppressed losses, all three are bounded by
\[
 \langle\zeta\rangle_\gamma^{-(\gamma+1)-2(5q-\delta)+},
\]
which proves the estimate for \(\XoneFourB\). The other two trees require no new estimate.  For \(\XtwoThree\), the
two root inputs have regularities \(2q-\delta\) and
\(3q-\delta\), up to arbitrarily small losses, by
\Cref{lem:Xone-Xtwo-kernel-bounds}.  Repeating the same three-region
argument adds the root gain \(\delta\), and hence
\[
 \mathcal C_{G_{\XtwoThree,5}^\eps}(\zeta)
 \lesssim
 \langle\zeta\rangle_\gamma^{-(\gamma+1)-2(5q-\delta)+}.
\]
Finally, two applications of \eqref{eq:fractional-annular-convolution}
give regularity \(q-\) for the four-leaf comb subtree of
\(\XoneFourC\), using \(3q-\delta>0\).  With regularity
\(q-\delta\) at the other input and gain \(\delta\) at the root, this
gives \(2q-\), as claimed.
\end{proof}

\begin{lemma}[Contracted response kernels]
\label{lem:response-contracted-kernel-bounds}
Let \(G_{\tau,3}^\eps\) and \(G_{\tau,1}^\eps\) denote the symmetric
kernels obtained from all one-pair and two-pair Wick contractions of
\eqref{eq:fifth-chaos-tree-kernels}.  In the regime
\eqref{eq:sbe-tree-response-regime}, uniformly in \(\eps\),
\[
 \mathcal C_{G_{\tau,3}^\eps}(\zeta)
 \lesssim
 \begin{cases}
  \langle\zeta\rangle_\gamma^{
   -(\gamma+1)-2(5q-\delta)+},
  &\tau\in\{\XoneFourB,\XtwoThree\},\\
  \langle\zeta\rangle_\gamma^{-(\gamma+1)-4q+},
  &\tau=\XoneFourC,
 \end{cases}
\]
and, for all three coordinates,
\[
 \mathcal C_{G_{\tau,1}^\eps}(\zeta)
 \lesssim
 \langle\zeta\rangle_\gamma^{
  -(\gamma+1)-2(5q-\delta)+}.
\]
\end{lemma}

\begin{proof}
An arc joining leaves \(i,j\) sets
\((\zeta_i,\zeta_j)=(\zeta,-\zeta)\) and integrates in \(\zeta\).

\begin{center}
\begin{minipage}{\textwidth}
Representative Wick contractions contributing to the lower-chaos
components are displayed below; a colored dashed arc represents a pairing
of two noise leaves.

\medskip
{\centering
\setlength{\tabcolsep}{7pt}
\begin{tabular}{cccccc}
 \includegraphics[height=1.19cm]{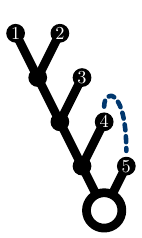}&
 \includegraphics[height=1.19cm]{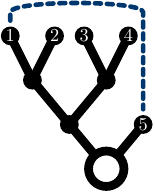}&
 \includegraphics[height=1.19cm]{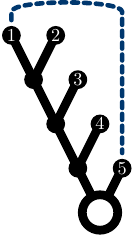}&
 \includegraphics[height=1.19cm]{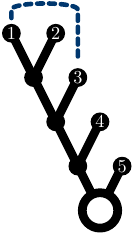}&
 \includegraphics[height=1.19cm]{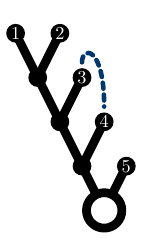}&
 \includegraphics[height=1.19cm]{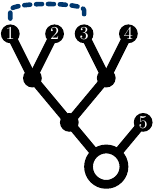}\\[-2pt]
 {\footnotesize\(\XoneFourC(45)\)}&
 {\footnotesize\(\XoneFourB(15)\)}&
 {\footnotesize\(\XoneFourC(15)\)}&
 {\footnotesize\(\XoneFourC(13)\)}&
 {\footnotesize\(\XoneFourC(34)\)}&
 {\footnotesize\(\XoneFourB(13)\)}
\end{tabular}

\vspace{5pt}
\begin{tabular}{ccccc}
 \includegraphics[height=1.19cm]{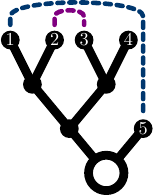}&
 \includegraphics[height=1.19cm]{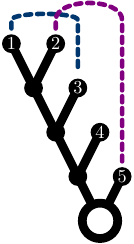}&
 \includegraphics[height=1.19cm]{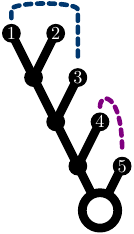}&
 \includegraphics[height=1.19cm]{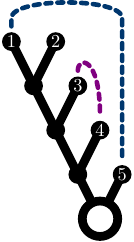}&
 \includegraphics[height=1.19cm]{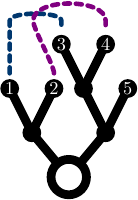}\\[-2pt]
 {\footnotesize\(\XoneFourB(15)(23)\)}&
 {\footnotesize\(\XoneFourC(13)(25)\)}&
 {\footnotesize\(\XoneFourC(13)(45)\)}&
 {\footnotesize\(\XoneFourC(15)(34)\)}&
 {\footnotesize\(\XtwoThree(13)(24)\)}
\end{tabular}
\par}
\end{minipage}
\end{center}

For an ordered one-pair contraction in the orbit \(\tau(ij)\), write
\(G_{\tau,3}^{\eps}(ij)\); for an ordered two-pair contraction in
\(\tau(ij)(k\ell)\), write \(G_{\tau,1}^{\eps}(ij)(k\ell)\).
As above, \(\mathcal C_G\) denotes the constrained \(L^2\)-density of
the indicated kernel. We now prove the kernel estimates for the contracted trees.

\smallskip
\noindent{\small\textbf{Single-multiplier one-loop contraction.}}

For the orbit \(\XoneFourC(45)\), substituting
\(\zeta_4=\zeta\) and \(\zeta_5=-\zeta\) in
\eqref{eq:worked-comb-kernel} bounds the absolute value of the resulting
ordered kernel, up to a harmless combinatorial factor, by
\[
 \abs{G_{\XoneFourC,3}^{\eps}(45)
       (\zeta_1,\zeta_2,\zeta_3)}
 \lesssim
 \abs{M_{123}}^2
 \abs{M_{12}H_1^\eps H_2^\eps H_3^\eps}
 \left|\int \abs{H_\alpha^\eps(\zeta)}^2
 \chi_\odot(-k,k+k_{123})
 M_\gamma(\zeta+\zeta_{123})\,\dd\zeta\right|.
\]
The loop in this expression is controlled by
\eqref{eq:lower-contraction-cutoff-loop} and contributes
\(\delta-2q\).  The remaining \(\Xone\)-subtree and noise leaf have
Fourier second-moment exponents
\(\gamma+1+2(2q-\delta)\) and
\(\gamma+1+2(q-\delta)\); hence \eqref{eq:fractional-annular-convolution},
together with the two outer \(M_{123}\)-factors, gives
\[
 \mathcal C_G(\zeta)
 \lesssim
 \langle\zeta\rangle_\gamma^{
 -(\gamma+1)-2(5q-\delta)+},
 \qquad G=G_{\XoneFourC,3}^{\eps}(45).
\]

\smallskip
\noindent{\small\textbf{Two-multiplier one-loop contractions.}}

For the representative orbit \(\XoneFourB(15)\), contracting leaves
\(1\) and \(5\) sets \(\zeta_5=-\zeta_1\).  Renaming
\(\zeta_1=\zeta\), we obtain, up to a harmless combinatorial factor,
\[
\begin{aligned}
 &\abs{G_{\XoneFourB,3}^{\eps}(15)
 (\zeta_2,\zeta_3,\zeta_4)}\\
 &\quad\lesssim
 \abs{M_{234}M_{34}H_2^\eps H_3^\eps H_4^\eps}
 \left|\int \abs{H_\alpha^\eps(\zeta)}^2
 \chi_\odot(-k,k+k_{234})
 M_\gamma(\zeta+\zeta_{234})
 M_\gamma(\zeta+\zeta_2)\,\dd\zeta\right|.
\end{aligned}
\]
After reversing \(\zeta\) and using \(\abs{\chi_\odot}\lesssim1\), the
absolute-value argument in
\eqref{eq:lower-contraction-two-multiplier-loop} bounds the integral by
\[\max\{\langle\zeta_{234}\rangle_\gamma,
\langle\zeta_2\rangle_\gamma\}^{-2q+}.\]  After relabelling,
the same argument applies to \(\XoneFourC(35)\) and
\(\XtwoThree(15)\).  A final application of
\eqref{eq:fractional-annular-convolution} gives
\[
 \mathcal C_G(\zeta)
 \lesssim
 \langle\zeta\rangle_\gamma^{
 -(\gamma+1)-2(5q-\delta)+},
 \quad
 G\in\left\{
 G_{\XoneFourB,3}^{\eps}(15),
 G_{\XoneFourC,3}^{\eps}(35),
 G_{\XtwoThree,3}^{\eps}(15)
 \right\}.
\]

For the orbits \(\XoneFourB(13)\) and \(\XoneFourC(14)\), the
cutoff-free or interchanged form of
\eqref{eq:lower-contraction-two-multiplier-loop} gives the same estimate:
\[
 \mathcal C_G(\zeta)
 \lesssim
 \langle\zeta\rangle_\gamma^{
 -(\gamma+1)-2(5q-\delta)+},
 \qquad
 G\in\left\{
 G_{\XoneFourB,3}^{\eps}(13),
 G_{\XoneFourC,3}^{\eps}(14)
 \right\}.
\]

\smallskip
\noindent{\small\textbf{Three-multiplier one-loop contractions.}}

Pairing leaves \(1\) and \(5\) creates one loop frequency.  It runs
through the three differentiated vertices on the path between them,
namely \(12\subset123\subset1234\).  At the root the paired frequencies
cancel, so the root depends only on leaves \(2,3,4\) and its multiplier
becomes \(M_{234}\):

\begin{center}
 \includegraphics[height=2.75cm]{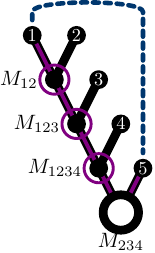}
\end{center}

By definition of the Fourier second-moment density, the three unpaired
leaf frequencies are integrated over the hypersurface
\(\zeta_2+\zeta_3+\zeta_4=\zeta\), and we seek decay in this fixed output
frequency \(\zeta\).  Write
\[
 \zeta=\zeta_{234}=(\omega,k),\qquad
 \eta=\zeta_2,\qquad \theta=\zeta_{23}.
\]
For fixed \(\zeta\), this parametrizes the hypersurface by
\[
 (\zeta_2,\zeta_3,\zeta_4)
 =(\eta,\theta-\eta,\zeta-\theta),
 \qquad \dd\zeta_2\dd\zeta_3=\dd\eta\dd\theta.
\]
With \((\zeta_1,\zeta_5)=(\zeta',-\zeta')\), where
\(\zeta'=(\omega',k')\), the tree gives the ordered kernel
\[
\begin{aligned}
 \abs{G_{\XoneFourC,3}^{\eps}(15)
 (\eta,\theta-\eta,\zeta-\theta)}
 & \lesssim
 \abs{M_\gamma(\zeta)H_\alpha^\eps(\eta)
 H_\alpha^\eps(\theta-\eta)H_\alpha^\eps(\zeta-\theta)}\\[-2pt]
 &\qquad\times
 \left|\int \abs{H_\alpha^\eps(\zeta')}^2
 \chi_\odot(-k',k'+k)
 \prod_{z\in\{\eta,\theta,\zeta\}}
 M_\gamma(\zeta'+z)\,\dd\zeta'\right|.
\end{aligned}
\]
Thus the three shifted multipliers come directly from
\(M_{12},M_{123},M_{1234}\), while \(M_\gamma(\zeta)\) is the root
multiplier.  This is the three-multiplier analogue of
\eqref{eq:lower-contraction-two-multiplier-loop}.  Put
\(\mu=\gamma+1+2(q-\delta)\).  Squaring the displayed kernel,
using \eqref{eq:sbe-H-multiplier}--\eqref{eq:sbe-M-multiplier}, and
changing signs in the two loop variables bounds its covariance by
\[
\begin{aligned}
 \langle\zeta\rangle_\gamma^{-2\delta}
 \int \frac{\dd\eta\dd\theta\dd\zeta'\dd\zeta''}
 {\langle\eta\rangle_\gamma^\mu
  \langle\theta-\eta\rangle_\gamma^\mu
  \langle\zeta-\theta\rangle_\gamma^\mu
  \langle\zeta'\rangle_\gamma^\mu
  \langle\zeta''\rangle_\gamma^\mu}
 \prod_{r\in\{\zeta',\zeta''\}}
 \frac{1}{
  \langle r+\eta\rangle_\gamma^\delta
  \langle r+\theta\rangle_\gamma^\delta
  \langle r+\zeta\rangle_\gamma^\delta}.
\end{aligned}
\]
Here \(\zeta',\zeta''\) are the two copies of the contracted frequency
created by squaring the kernel.  The five \(\mu\)-weights come from the
three unpaired leaves and \(\zeta',\zeta''\);
the six \(\delta\)-weights come from the three shifted multipliers in
each copy; and the prefactor is the squared root multiplier.

We now sum the fractional-parabolic blocks with essentially the same idea as in
\eqref{eq:fractional-annular-convolution}.  First restrict the preceding
integral to the region in which each of its eleven denominator brackets
is comparable to \(\langle\zeta\rangle_\gamma\).  Since each of
\(\eta,\theta,\zeta',\zeta''\) then ranges over a fractional-parabolic block of
measure \(O(\langle\zeta\rangle_\gamma^{\gamma+1})\), this part of the
integral satisfies
\[
\begin{aligned}
&\langle\zeta\rangle_\gamma^{-2\delta}
 \int_{\substack{\text{each denominator}\\
\text{ bracket }\asymp\langle\zeta\rangle_\gamma}}
 \frac{\dd\eta\dd\theta\dd\zeta'\dd\zeta''}
 {\langle\eta\rangle_\gamma^\mu
  \langle\theta-\eta\rangle_\gamma^\mu
  \langle\zeta-\theta\rangle_\gamma^\mu
  \langle\zeta'\rangle_\gamma^\mu
  \langle\zeta''\rangle_\gamma^\mu}
 \prod_{r\in\{\zeta',\zeta''\}}
 \frac{1}{
  \langle r+\eta\rangle_\gamma^\delta
  \langle r+\theta\rangle_\gamma^\delta
  \langle r+\zeta\rangle_\gamma^\delta}\\
&\quad\lesssim
 \langle\zeta\rangle_\gamma^{-2\delta}
 \langle\zeta\rangle_\gamma^{4(\gamma+1)}
 \langle\zeta\rangle_\gamma^{-5\mu}
 \langle\zeta\rangle_\gamma^{-6\delta}=
 \langle\zeta\rangle_\gamma^{-(\gamma+1)-2(5q-\delta)}.
\end{aligned}
\]
On the other hand, the weakest separation of the external scale from a lower dyadic scale
\(L\ll\langle\zeta\rangle_\gamma\) is the region
\[
\begin{gathered}
 \langle\eta\rangle_\gamma,\ \langle\theta-\eta\rangle_\gamma,
 \ \langle\zeta'\rangle_\gamma,\ \langle\zeta''\rangle_\gamma,
 \ \langle r+\eta\rangle_\gamma,\ \langle r+\theta\rangle_\gamma
 \asymp L
 \quad\text{for }r\in\{\zeta',\zeta''\},\\
 \langle\zeta-\theta\rangle_\gamma,
 \ \langle\zeta'+\zeta\rangle_\gamma,
 \ \langle\zeta''+\zeta\rangle_\gamma
 \asymp\langle\zeta\rangle_\gamma.
\end{gathered}
\]
The restriction of the covariance integral to this region is clearly bounded by
\[
\begin{aligned}
 \langle\zeta\rangle_\gamma^{-2\delta}
  \langle\zeta\rangle_\gamma^{-\mu-2\delta}
  L^{4(\gamma+1)-4\mu-4\delta}=
  \langle\zeta\rangle_\gamma^{-\mu-4\delta}
  L^{4(\delta-2q)}.
\end{aligned}
\]
Indeed, the four integration variables have block volume
\(O(L^{4(\gamma+1)})\), while four \(\mu\)-weights and four
\(\delta\)-weights have scale \(L\).  Summing over these separated
shells, and harmlessly extending the sum to
\(L\le\langle\zeta\rangle_\gamma\), gives
\[
 \langle\zeta\rangle_\gamma^{-\mu-4\delta}
 \sum_{\substack{L\le\langle\zeta\rangle_\gamma\\L\text{ dyadic}}}
 L^{4(\delta-2q)}
 \lesssim
 \langle\zeta\rangle_\gamma^{-(\gamma+1)-2(5q-\delta)}.
\]
A routine but tedious repetition of the same calculation shows that all
remaining dyadic regions contribute at most the same order. A similar argument applied to \(\XtwoThree(13)\) gives the same bound.
Thus
\[
 \mathcal C_G(\zeta)
 \lesssim
 \langle\zeta\rangle_\gamma^{
 -(\gamma+1)-2(5q-\delta)+},
 \qquad
 G\in\left\{
 G_{\XoneFourC,3}^{\eps}(15),
 G_{\XtwoThree,3}^{\eps}(13)
 \right\}.
\]

\smallskip
\noindent{\small\textbf{Contracted-vertex one-loop contractions.}}

For the orbits \(\XoneFourC(34)\) and \(\XtwoThree(35)\), the
contracted-vertex estimate
\(\abs{M_\gamma L^\eps}\lesssim
\langle\cdot\rangle_\gamma^{-2q+}\), followed by
\eqref{eq:fractional-annular-convolution}, gives
\[
 \mathcal C_G(\zeta)
 \lesssim
 \langle\zeta\rangle_\gamma^{
 -(\gamma+1)-2(5q-\delta)+},
 \qquad
 G\in\left\{
 G_{\XoneFourC,3}^{\eps}(34),
 G_{\XtwoThree,3}^{\eps}(35)
 \right\}.
\]

\smallskip
\noindent{\small\textbf{Nested one-loop contraction.}}

The remaining orbit \(\XoneFourC(13)\) gives the weaker bound.  Contracting the
\(X\)-leaf attached directly to the root of the \(\Xtwo\)-subtree with a
noise leaf of its inner \(\Xone\)-subtree leaves
\[
 M_{24}M_2H_2^\eps H_4^\eps
 \int \abs{H_\alpha^\eps(\zeta)}^2
 M_\gamma(\zeta+\zeta_2)\,\dd\zeta .
\]
By \eqref{eq:sbe-M-multiplier} and
\eqref{eq:Xtwo-loop-bound},
\[
 \left|M_2\int \abs{H_\alpha^\eps(\zeta)}^2
 M_\gamma(\zeta+\zeta_2)\,\dd\zeta\right|
 \lesssim \langle\zeta_2\rangle_\gamma^{-2q+}.
\]
After squaring, the covariance of the contracted \(\Xtwo\)-subtree is
bounded by
\(\langle\zeta_2\rangle_\gamma^{-(\gamma+1)-2(3q-\delta)+}\),
while the remaining noise leaf has second-moment exponent
\((\gamma+1)+2(q-\delta)\).  Since
\(3q-\delta>0\), \eqref{eq:fractional-annular-convolution} gives
\[
 \int_{\zeta_2+\zeta_4=\zeta'}
 \left|M_{24}M_2H_2^\eps H_4^\eps
 \int \abs{H_\alpha^\eps(\zeta)}^2
 M_\gamma(\zeta+\zeta_2)\,\dd\zeta\right|^2\dd\zeta_2
 \lesssim
 \langle\zeta'\rangle_\gamma^{-(\gamma+1)-2q+}.
\]
Thus this component of \((\Xthree)^\eps\) has regularity \(q-\).
At the final resonance the two inputs to the root have regularities
\(q-\delta\) and \(q-\).  The resonant-root estimate above therefore
gives
\[
 \mathcal C_G(\zeta)
 \lesssim
 \langle\zeta\rangle_\gamma^{-(\gamma+1)-4q+},
 \qquad G=G_{\XoneFourC,3}^{\eps}(13).
\]
This is the only one-pair orbit for which the weaker third-chaos bound is
needed.

\smallskip
\noindent{\small\textbf{Successive two-loop contraction.}}

For the representative \(\XoneFourB(15)(23)\), with
\(\zeta'=(\omega',k')\) the remaining input frequency, its first-chaos
kernel is, up to symmetrization,
\[
\begin{aligned}
 &M_\gamma(\zeta')H_\alpha^\eps(\zeta')
 \int\abs{H_\alpha^\eps(\zeta)}^2
 \chi_\odot(-k,k+k')M_\gamma(\zeta+\zeta')\\
 &\qquad\times
 \left[
  \int\abs{H_\alpha^\eps(\zeta'')}^2
 M_\gamma(\zeta+\zeta'')M_\gamma(\zeta'-\zeta'')\,\dd\zeta''
 \right]\dd\zeta .
\end{aligned}
\]
Using \(\abs{M_\gamma(-\zeta)}=\abs{M_\gamma(\zeta)}\), the
absolute-value argument of
\eqref{eq:lower-contraction-two-multiplier-loop} bounds the inner
bracket by
\(\max\{\langle\zeta\rangle_\gamma,
\langle\zeta'\rangle_\gamma\}^{-2q+}\).
Estimates \eqref{eq:lower-contraction-cutoff-loop} and
\eqref{eq:lower-contraction-two-multiplier-loop}, followed by
\eqref{eq:fractional-annular-convolution} with the remaining noise-leaf
exponent \(\gamma+1+2(q-\delta)\), show that the factor multiplying
\(H_\alpha^\eps(\zeta')\) is
\(O(\langle\zeta'\rangle_\gamma^{-4q+})\).  Thus
\[
 \mathcal C_G(\zeta)
 \lesssim
 \langle\zeta\rangle_\gamma^{
 -(\gamma+1)-2(5q-\delta)+},
 \qquad G=G_{\XoneFourB,1}^{\eps}(15)(23).
\]

\smallskip
\noindent{\small\textbf{Contracted-vertex two-loop contractions.}}

For the orbits \(\XoneFourC(13)(25)\) and \(\XoneFourC(15)(34)\),
\eqref{eq:sbe-M-multiplier} and
\eqref{eq:Xtwo-loop-bound} give, respectively,
\(\abs{M_\gamma L^\eps}\lesssim
\langle\cdot\rangle_\gamma^{-2q+}\) and
\(\abs{M_\gamma^2L^\eps}\lesssim
\langle\cdot\rangle_\gamma^{-\delta-2q+}\).
After \eqref{eq:fractional-annular-convolution}, both satisfy
\[
 \mathcal C_G(\zeta)
 \lesssim
 \langle\zeta\rangle_\gamma^{
 -(\gamma+1)-2(5q-\delta)+},
 \qquad
 G\in\left\{
 G_{\XoneFourC,1}^{\eps}(13)(25),
 G_{\XoneFourC,1}^{\eps}(15)(34)
 \right\}.
\]

\smallskip
\noindent{\small\textbf{Mixed-cutoff two-loop contraction.}}

For the orbit \(\XoneFourC(13)(45)\), the product of the cutoff
contraction in
\eqref{eq:lower-contraction-cutoff-loop} and the cutoff-free contraction
in \eqref{eq:Xtwo-loop-bound} gives
\[
 \mathcal C_G(\zeta)
 \lesssim
 \langle\zeta\rangle_\gamma^{
 -(\gamma+1)-2(5q-\delta)+},
 \qquad G=G_{\XoneFourC,1}^{\eps}(13)(45).
\]

\smallskip
\noindent{\small\textbf{Subtree-covariance two-loop contraction.}}

For the orbit \(\XtwoThree(13)(24)\), the contracted covariance is
\[
 M_\gamma(\zeta')H_\alpha^\eps(\zeta')
 \int\chi_\odot(k,k'-k)M_\gamma(\zeta'-\zeta)
 \mathcal C_{G_{\Xone,2}^\eps}(\zeta)\,\dd\zeta .
\]
Applying \eqref{eq:fractional-annular-convolution} with exponents
\(\gamma+1+2(2q-\delta)\) and \(\delta\) gives the
minimum \(4q-\delta\); the outer multiplier therefore makes the
factor multiplying \(H_\alpha^\eps(\zeta')\)
\(O(\langle\zeta'\rangle_\gamma^{-4q+})\).
Hence
\[
 \mathcal C_G(\zeta)
 \lesssim
 \langle\zeta\rangle_\gamma^{
 -(\gamma+1)-2(5q-\delta)+},
 \qquad G=G_{\XtwoThree,1}^{\eps}(13)(24).
\]

\smallskip
\noindent{\small\textbf{Reduced two-loop contractions.}}

Finally, the four unpictured orbits
\[
 \XoneFourC(14)(35),\quad \XoneFourC(14)(25),\quad
 \XtwoThree(13)(45),\quad \XtwoThree(13)(25)
\]
reduce to the preceding mechanisms by exchanging loop variables and
using the cutoff-free or interchanged variants of
\eqref{eq:lower-contraction-two-multiplier-loop}.  Therefore
\[
 \mathcal C_G(\zeta)
 \lesssim
 \langle\zeta\rangle_\gamma^{
 -(\gamma+1)-2(5q-\delta)+},
 \quad
 G\in\left\{
 G_{\XoneFourC,1}^{\eps}(14)(35),
 G_{\XoneFourC,1}^{\eps}(14)(25),
 G_{\XtwoThree,1}^{\eps}(13)(45),
 G_{\XtwoThree,1}^{\eps}(13)(25)
 \right\}.
\]

These paragraphs exhaust, up to the tree symmetries, all non-vanishing
contractions: eleven one-pair and nine two-pair orbits.  The omitted
contractions vanish because a sibling pairing kills the corresponding
differentiated multiplier or, for two pairs, forces the four-leaf
frequency to zero.
\end{proof}

Now that we have all the necessary ingredients, we can combine them and prove \Cref{prop:fractional-enhancement}. 

\begin{proof}[Proof Outline of \Cref{prop:fractional-enhancement}]
The deterministic reconstructions preceding
\Cref{lem:lower-contraction-loop-estimates} cover the complementary
parameter regimes, and \Cref{lem:first-six-enhancement-kernel-bounds}
constructs the first six coordinates.  It remains to work in the regime
\eqref{eq:sbe-tree-response-regime}.  Combining
\Cref{lem:response-leading-kernel-bounds,lem:response-contracted-kernel-bounds}
gives, uniformly in \(\eps\),
\begin{equation}\label{eq:sbe-tree-covariance-kernel}
 \sum_{n\in\{1,3,5\}}\mathcal C_{G_{\tau,n}^\eps}(\zeta)
 \lesssim
 \langle\zeta\rangle_\gamma^{
 -(\gamma+1)-2s_{\alpha,\gamma}(\tau)+}.
\end{equation}
For \(\tau\in\{\XoneFourB,\XtwoThree\}\),
\(s_{\alpha,\gamma}(\tau)=5q-\delta\) in this regime; for
\(\tau=\XoneFourC\), \(s_{\alpha,\gamma}(\tau)=2q\).
We first pass from the stationary space--time kernels to a fixed-time
estimate.  For the \(n\)-th chaos component, inverse Fourier
transformation at time \(t\) contributes the phase
\(e^{it(\omega_1+\cdots+\omega_n)}\).  Since this phase has modulus one,
the Wiener--It\^o isometry and the definition of \(\mathcal C_G\) give the control of the stationary trees:
\[
 \begin{aligned}
 \mathbb E\lvert\widehat\tau_{n,\mathrm{stat}}^\eps(t,k)\rvert^2
 &=n!\sum_{k_1+\cdots+k_n=k}\int_{\R^n}
   \abs{G_{\tau,n}^\eps(\zeta_1,\ldots,\zeta_n)}^2
   \,\dd\omega_1\cdots\dd\omega_n\\
 &=\int_\R\mathcal C_{G_{\tau,n}^\eps}(\omega,k)\,\dd\omega.
 \end{aligned}
\]
Chaos orthogonality and the comparison of the finite-time kernel
functions with their stationary versions in
\cite[Section~9.5, the induction leading to (87), and
(87)]{GubinelliPerkowski2017KPZReloaded} therefore imply
\[
 \mathbb E\lvert\widehat\tau^\eps(t,k)\rvert^2
 \lesssim
 \sum_{n\in\{1,3,5\}}\int_\R
 \mathcal C_{G_{\tau,n}^\eps}(\omega,k)\,\dd\omega.
\]

Put \(s=s_{\alpha,\gamma}(\tau)\).  In the present regime,
\(2s-\delta
\ge 4q-\delta>0\)
for all three driving terms.  Make the terminal loss in
\eqref{eq:sbe-tree-covariance-kernel} explicit as \(\eta>0\), chosen so that
\(\gamma+1+2s-\eta>\gamma\).  Under
\(\omega=(1+\abs k)^\gamma y\),
\[
 \begin{gathered}
 \langle(\omega,k)\rangle_\gamma
 =(1+\abs k)(1+\abs y^{1/\gamma}),
 \qquad
 \dd\omega=(1+\abs k)^\gamma\dd y.
 \end{gathered}
\]
Consequently,
\[
 \begin{aligned}
 \int_\R\langle(\omega,k)\rangle_\gamma^{
  -(\gamma+1)-2s+\eta}\,\dd\omega
 &=(1+\abs k)^{-1-2s+\eta}
   \int_\R(1+\abs y^{1/\gamma})^{
    -(\gamma+1)-2s+\eta}\,\dd y\\
 &\lesssim(1+\abs k)^{-1-2s+\eta}.
 \end{aligned}
\]
The \(y\)-integral is finite by the choice of \(\eta\).  Spatial
stationarity makes distinct Fourier modes orthogonal, while the block
\(\abs k\asymp2^j\) contains \(O(2^j)\) modes.  Consequently,
\begin{equation}\label{eq:sbe-tree-fixed-time-block}
 \begin{aligned}
 \mathbb E\abs{\Delta_j\tau^\eps(t,x)}^2
 &\lesssim
 \sum_{\abs k\asymp2^j}(1+\abs k)^{-1-2s+\eta}\\
 &\lesssim
 2^j2^{-j(1+2s-\eta)}
 =2^{-2j(s-\eta/2)}
 \lesssim2^{-2j(s_{\alpha,\gamma}(\tau)-)}.
 \end{aligned}
\end{equation}
The bottom spatial block contains only finitely many modes.

The same
strict inequality \(2s-\delta>0\) also allows the time increment.  For
\(0<\vartheta\le1/2\), the time-frequency integral remains
finite after choosing the suppressed loss within this strict margin,
and gives
\begin{equation}\label{eq:sbe-tree-time-increment-block}
 \abs{e^{it\omega}-e^{is\omega}}^2
\lesssim\abs{t-s}^{2\vartheta}
\langle(\omega,k)\rangle_\gamma^{2\gamma\vartheta} \implies \mathbb E\abs{\Delta_j(\tau^\eps(t)-\tau^\eps(s))(x)}^2
 \lesssim
 \abs{t-s}^{2\vartheta}
 2^{-2j(s_{\alpha,\gamma}(\tau)-\gamma\vartheta-)}.
\end{equation}
Since
\(\vartheta>0\) and these losses may be chosen arbitrarily small,
\eqref{eq:sbe-tree-fixed-time-block} and
\eqref{eq:sbe-tree-time-increment-block}, together with the standard
finite-chaos completion fixed at the beginning of this appendix, give
every spatial regularity \(s_{\alpha,\gamma}(\tau)-\), the corresponding
time modulus, and mollifier-independent convergence.  Comparing with the
stationary kernels after time translation makes these estimates uniform
over restart times in compact sets.  Together with
\Cref{lem:first-six-enhancement-kernel-bounds}, this proves joint
convergence of all nine driving coordinates.
\end{proof}

\begin{remark}[Uniform resolvent extension of \(\mathbb X_Q\)]
For \(\lambda\ge0\), set
\[
 M_{\gamma,\lambda}(\omega,k)
 :=\frac{ik}{i\omega+\abs{k}^{\gamma}+\lambda}.
\]
For the stationary version,
\(Q_{X,\lambda}^\eps
=\mathcal I_1(M_{\gamma,\lambda}H_\alpha^\eps)\),
the Wiener product formula and the Fourier definition of the Bony
resonance give the second-chaos kernel
\[
\begin{aligned}
 G_{Q,\lambda,2}^\eps(\zeta_1,\zeta_2)
 :=
 \operatorname{Sym}_2\left[
  \chi_\odot(k_1,k_2)
  M_{\gamma,\lambda}(\zeta_1)
  H_1^\eps H_2^\eps
 \right].
\end{aligned}
\]
At \(\lambda=0\), this is exactly the kernel
\(G_{\mathbb X_Q,2}^\eps\) above.  For every \(\lambda\ge0\), the
zeroth-chaos component vanishes by the same calculation.  Consequently,
\(\mathbb X_{Q,\lambda}^\eps
=\mathcal I_2(G_{Q,\lambda,2}^\eps)\).  Assume now without loss of
generality that \(q\le\delta/2\).  The elementary bound
\(\abs{M_{\gamma,\lambda}(\omega,k)}
\le\abs{M_\gamma(\omega,k)}\), Jensen's inequality, and the preceding
second-moment estimate for \(G_{\mathbb X_Q,2}^\eps\) give, 
\[
 \sup_{\lambda\ge0}
 \mathcal C_{G_{Q,\lambda,2}^\eps}(\zeta)
 \lesssim
 \langle\zeta\rangle_\gamma^{
  -(\gamma+1)-2(2q-\delta)+}.
\]

For \(t_0\le s\le t\), let
\(
 \Kcal_{\gamma;s,t}F
 :=\int_s^t\partial_xe^{-(t-r)\Lambda^\gamma}F(r)\,\dd r
\).
Integration by parts in the Laplace weight, as in
\cite[Lemma~2.16]{ZhangZhuZhu2022SingularHJB}, gives the identity below;
the same spatial oddness removes each zeroth-chaos contraction:
\[
\begin{aligned}
 \mathbb X_{Q,\lambda}^\eps(t)
 &=
 e^{-\lambda(t-t_0)}
 \bigl(\Kcal_{\gamma;t_0,t}X^\eps\bigr)
 \odot_{\rm ren}X^\eps(t)\\
 &\quad+
 \lambda\int_{t_0}^t e^{-\lambda(t-s)}
 \bigl(\Kcal_{\gamma;s,t}X^\eps\bigr)
 \odot_{\rm ren}X^\eps(t)\,\dd s.
\end{aligned}
\]
The positive coefficients on the right have total mass one.  Consequently,
for every \(\rho\in\R\),
\[
 \sup_{\lambda\ge0}\sup_{t_0\le t\le T}
 \norm{\mathbb X_{Q,\lambda}^\eps(t)}_{\Ccal^\rho}
 \le
 \sup_{t_0\le s\le t\le T}
 \norm{
  (\Kcal_{\gamma;s,t}X^\eps)
  \odot_{\rm ren}X^\eps(t)}_{\Ccal^\rho}.
\]

The uniform truncated-convolution estimate of
\cite[Lemma~6.1]{ZhangZhuZhu2022SingularHJB}, with the preceding
fractional second-moment estimate in place of its heat-kernel estimate,
gives the required bounds and mollifier convergence uniformly over
\(t_0\le s\le t\le T\).  The last inequality transfers these estimates
uniformly to the resolvent parameter. Thus, for every \(p<\infty\),
\begin{equation}\label{eq:uniform-resolvent-second-chaos}
\begin{aligned}
 &\sup_{0<\eps\le1}
 \mathbb E\left[
  \sup_{\lambda\ge1}
  \norm{\mathbb X_{Q,\lambda}^\eps}_{
   C_t\Ccal^{2q-\delta-}}^p
 \right]<\infty,\\
 &\lim_{\eps\downarrow0}
 \mathbb E\left[
  \sup_{\lambda\ge1}
  \norm{\mathbb X_{Q,\lambda}^\eps
        -\mathbb X_{Q,\lambda}}_{
   C_t\Ccal^{2q-\delta-}}^p
 \right]=0.
\end{aligned}
\end{equation}
The estimates are uniform over restart times in compact intervals and
prove the model bound and convergence asserted after
\eqref{eq:resolvent-model-uniform}.
\end{remark}

\section*{Acknowledgements}
This project originated in discussions with Felix Otto and Nicolas
Perkowski, to whom I am grateful.  Part of this work was carried out during
a visit to the Pattern Formation, Energy Landscapes and Scaling Laws group
at the Max Planck Institute for Mathematics in the Sciences in Leipzig.  I
gratefully acknowledge the hospitality of the group.  I also thank Bj\"orn
Bringmann, Wenhao Zhao, Fan Cheng and Carlos Villanueva Mariz for helpful discussions.

\begingroup
\hfuzz=2pt
\hbadness=5000
\renewcommand{\bibliofont}{\fontsize{9}{11.4}\selectfont}
\bibliographystyle{alphaurl}
\bibliography{sbe_references}
\endgroup

\end{document}